\documentclass{amsart}
\usepackage{amssymb}
\usepackage{relsize}
\usepackage{enumerate}
\allowdisplaybreaks

\usepackage{comment} 
\usepackage{soul}
\usepackage{float}

\let\oldtocsection=\tocsection 
\let\oldtocsubsection=\tocsubsection 
\renewcommand{\tocsection}[2]{\hspace{0mm}\oldtocsection{#1}{#2}}
\renewcommand{\tocsubsection}[2]{\hspace{2em}\oldtocsubsection{#1}{#2}}

\usepackage{hyperref}
\hypersetup{
    colorlinks,
    citecolor=black,
    filecolor=black,
    linkcolor=black,
    urlcolor=black
}

\usepackage{tikz-cd}
\tikzcdset{arrow style=math font}
\usetikzlibrary{external,arrows.meta}

\numberwithin{equation}{subsection}
\newtheorem{theorem}{Theorem}[subsection]
\newtheorem*{theorem*}{Theorem}
\newtheorem{corollary}[theorem]{Corollary}
\newtheorem*{corollary*}{Corollary}
\newtheorem*{conjecture*}{Conjecture}
\newtheorem{lemma}[theorem]{Lemma}
\newtheorem{proposition}[theorem]{Proposition}
\theoremstyle{definition}

\newtheorem{definition}[theorem]{Definition}
\newtheorem{remark}[theorem]{Remark}
\newtheorem{example}[theorem]{Example}
\newtheorem{notation}[theorem]{Notation}

\def\DGra{\mathsf{DGra}}
\def\Gra{\mathsf{Gra}}
\def\PairDGr{\mathsf{PairDGra}}
\def\Fun{\mathrm{Fun}}
\def\Hom{\mathrm{Hom}}
\def\colim{\mathrm{colim}}
\def\op{\mathrm{op}}
\def\Id{\mathrm{Id}} 
\def\id{\mathrm{id}} 
\def\SS{\mathbf{S}}
\def\Set{\mathsf{Set}}
\def\cSet{\mathsf{cSet}}
\def\Quiv{\mathsf{Quiv}}
\def\Out{\mathsf{Out}}
\def\In{\mathsf{In}}
\def\UU{\mathcal{U}}
\def\VV{\mathcal{V}}
\def\HH{\mathcal{H}}
\def\Ner{\mathsf{Ner}}

\def\Shr{\mathsf{Shr}}
\def\HShr{\mathsf{HShr}}
\def\Cant{\mathsf{Cant}}
\def\sSet{\mathsf{sSet}}
\def\Sing{\mathsf{Sing}}

\def\dist{\mathsf{dist}}
\def\Spc{\mathsf{Spc}}

\DeclareMathOperator{\diag}{diag}

\definecolor{carminered}{rgb}{1.0, 0.0, 0.22}
\definecolor{carrotorange}{rgb}{0.93, 0.57, 0.13}
\definecolor{amber}{rgb}{1.0, 0.75, 0.0}
\definecolor{armygreen}{rgb}{0.29, 0.33, 0.13}
\definecolor{blue(munsell)}{rgb}{0.0, 0.5, 0.69}

\begin{document}

\title{The discrete homotopy hypothesis for directed graphs}

\author[B. Eldridge]{Briony Eldridge}
\address{
Beijing Key Laboratory of Topological Statistics and Applications for Complex Systems, Beijing Institute of Mathematical Sciences and Applications (BIMSA), Beijing 101408, China.}
\email{eldridge@bimsa.cn}

\author[S. O. Ivanov]{Sergei O. Ivanov} 
\address{
Beijing Key Laboratory of Topological Statistics and Applications for Complex Systems, Beijing Institute of Mathematical Sciences and Applications (BIMSA), Beijing 101408, China.}
\email{ivanov.s.o.1986@gmail.com, ivanov.s.o.1986@bimsa.cn}

\author[X. Xu]{Xiaomeng Xu} 
\address{Mathematical Sciences, University of Southampton, Southampton SO17 1BJ, United Kingdom}
\email{xiaomeng.x.xu@gmail.com}

\author[S.-T. Yau]{Shing-Tung Yau}
\address{Yau Mathematical Sciences Center, Tsinghua
University, Beijing 100084, China. \\ Beijing Institute of Mathematical Sciences and Applications (BIMSA), Beijing 101408, China.}
\email{styau@tsinghua.edu.cn}

\author[M. Zhang]{Mengmeng Zhang}
\address{
Beijing Key Laboratory of Topological Statistics and Applications for Complex Systems, Beijing Institute of Mathematical Sciences and Applications (BIMSA), Beijing 101408, China.}
\email{mengmengzhang@bimsa.cn}

\begin{abstract} 
We develop a homotopy theory of directed graphs based on cubical
homotopy groups, also known as $A$-groups or reduced GLMY homotopy
groups. Localizing the category of directed graphs at morphisms
that induce isomorphisms on these groups yields an
$\infty$-category, denoted by ${\sf DGra}_\infty$. We prove that
${\sf DGra}_\infty$ is equivalent to the $\infty$-category of
spaces, establishing a directed version of the discrete homotopy
hypothesis of Carranza and Kapulkin.
\end{abstract}

\maketitle

\tableofcontents

\section*{\bf Introduction}

A homotopy theory of (simple) graphs was introduced in \cite{BKLW01,BL05,BBdLL06}, building on earlier ideas of Atkin \cite{Atk74,Atk76}. Its central invariant is the \emph{A-group} of a graph, denoted $A_n(G)$, an analogue of the $n$-th homotopy group of a topological space. In \cite{BBdLL06} it was conjectured that $A_n(G)$ is isomorphic to the $n$-th homotopy group of a suitable topological space. This was proved by Carranza and Kapulkin in \cite{CK24}, who constructed a functor to the category of cubical sets
\begin{equation}
N_1:\Gra \longrightarrow \cSet,    
\end{equation}
called the \emph{cubical nerve}, and showed that $A_n(G)\cong \pi_n(|N_1G|)$.

Independently, a homotopy theory of digraphs (directed graphs) was introduced in \cite{GLMY23} and further developed in \cite{LWYZ24,TWYZ25}. This theory is closely related to the path homology theory, also known as the GLMY-theory, developed earlier in \cite{GLMY12}. It has its own analogue of homotopy groups, the \emph{reduced GLMY homotopy groups} $\bar\pi_n(G)$. These groups extend the A-groups from graphs to digraphs, agreeing with $A_n(G)$ when $G$ is a graph, and—as we show—admit a description in terms of cubical sets. For this reason we also refer to them as the \emph{cubical homotopy groups} of $G$ and, throughout the paper, write
\begin{equation}
A_n(G):=\bar\pi_n(G).
\end{equation}
A digraph map that induces isomorphisms on $A_*$ is called a \emph{cubical weak equivalence}.

In \cite{CK26}, Carranza and Kapulkin formulated the following \emph{discrete homotopy hypothesis}. Let $\Gra_\infty$ denote the ($\infty$-categorical) localization of the category of graphs $\Gra$ at the class of cubical weak equivalences. Then the cubical nerve induces a functor to the $\infty$-category of spaces
\begin{equation}
N : \Gra_\infty \to \Spc.
\end{equation}
The discrete homotopy hypothesis asserts that this functor is an equivalence of $\infty$-categories. This conjecture remains open, although several partial results have been obtained; see \cite{KM24,Min26,CK26}. As explained in \cite{CK26}, the principal obstruction to the original discrete homotopy hypothesis is that pushouts of graphs admit a large number of non-degenerate cubes that are difficult to control. This issue can be resolved by passing to the category of digraphs, where the presence of directions imposes sufficient constraints to control the newly generated cubes arising in pushouts.

In this paper we establish the analogue of the discrete homotopy hypothesis in the directed setting, where it becomes a theorem. We first extend the results of \cite{CK24} to digraphs: we construct the cubical nerve functor $N_1 : \DGra \to \cSet$ and prove that $A_n(G)\cong \pi_n(|N_1G|)$. Let $\DGra_\infty$ denote the $\infty$-category obtained from $\DGra$ by localizing at the class of cubical weak equivalences. Then the functor $N_1$ induces a functor to the $\infty$-category of spaces $N:\DGra_\infty\to \Spc$. Our main result is the following.
\begin{theorem*}[Th. \ref{th:equivalence_to_spaces}]
The functor $N$ is an equivalence of $\infty$-categories
\begin{equation}
\DGra_\infty\simeq \Spc.
\end{equation}
\end{theorem*}
In particular, every homotopy type is realized by a digraph: for every pointed space $X$ there exists a pointed digraph $G$ such that $A_*(G)\cong \pi_*(X)$.

We further describe finite limits and small colimits in $\DGra_\infty$ by equipping $\DGra$ with the structures of a category of fibrant objects and of a homotopy cocomplete category of cofibrant objects, respectively. As an application, we recover the long exact sequence of cubical homotopy groups constructed in \cite{LWYZ24} (Corollary \ref{cor:long_exact_sequence}).

The paper is organized as follows. In Section \ref{sec:Cubical_homotopy_groups}, we develop a homotopy theory of digraphs, where homotopies are defined by digraph maps $G\otimes I_1 \to H$, and introduce the cubical homotopy groups $A_*(G)$. We consider the \emph{category of shrinkings} $\Shr$, whose objects are intervals and morphisms are shrinkings. We define a notion of a \emph{nice subcategory of shrinkings} $\SS \subseteq \Shr$ and show that $A_*(G)$ admits an equivalent definition in terms of any such $\SS$. This provides a unified framework encompassing several approaches in the literature. We also define the path and loop digraphs $P_\SS G$ and $\Omega_\SS G$ respectively, depending on $\SS$, and prove that they satisfy the expected properties whenever $\SS$ is cofiltered.

In Section \ref{sec:Cubical nerves of digraphs}, generalizing
the ideas of \cite{CK24}, for each $m\geq 1$ we introduce a functor
$N_m:\DGra\to\cSet$, called the \emph{cubical $m$-nerve}.
We also construct a functor $N_\infty:\DGra\to\cSet$ as a certain colimit (Definition~\ref{def:infinite_nerve}), together
with a natural comparison map
\begin{equation}
N_1G\longrightarrow N_\infty G.    
\end{equation}
Extending the results of \cite{CK24}, we prove that this map
is anodyne, that $N_\infty G$ is a Kan complex, and that its
homotopy groups are naturally isomorphic to the cubical homotopy
groups of $G$. Our proof of the anodyne property also corrects a gap in the proof of \cite[Th.~4.12]{CK24}, as explained in Remark~\ref{rem:CK24-gap}. Thus $N_\infty G$ provides a functorial fibrant
replacement of $N_1G$, and we obtain natural isomorphisms
\begin{equation}
\pi_*(N_1G)\cong \pi_*(N_\infty G)\cong A_*(G).    
\end{equation}
Consequently, the two descriptions of a cubical weak equivalence
coincide: a digraph map $f$ induces an isomorphism on $A_*$
if and only if $N_1f$ is a weak equivalence of cubical sets.
Our comparison results also show that $N_1$ and $N_\infty$ induce 
the same functor of $\infty$-categories
\begin{equation}
N:\DGra_\infty\longrightarrow\Spc.    
\end{equation}

In Section \ref{sec:Finite homotopy limits}, we introduce the notion of a cubical fibration of digraphs and show that $\DGra$, together with the classes of cubical weak equivalences and cubical fibrations, forms a category of fibrant objects. This framework allows us to describe finite limits in $\DGra_\infty$ and to prove the long exact sequence of cubical homotopy groups. To illustrate the notion of cubical fibration, we demonstrate that the 2-coverings introduced in \cite{DIMZ24} are cubical fibrations.

In Section \ref{sec:nerve_theorem}, we introduce the notion of an in-closed subdigraph and prove a nerve theorem for covers of a digraph by in-closed subdigraphs. This provides a useful tool for computing the homotopy type of $NG$ for certain classes of digraphs $G$.

In Section \ref{sec:homotopy_colimits}, we introduce the notion of an in-closed cofibration and show that the category $\DGra$, equipped with the classes of cubical weak equivalences and in-closed cofibrations, is a homotopy cocomplete category of cofibrant objects. Moreover, we prove that $N_1 : \DGra \to \cSet$ is a homotopy cocontinuous functor. Consequently, $\DGra_\infty$ admits small colimits and $N: \DGra_\infty \to \Spc$ is a colimit-preserving functor.

In Section \ref{section:equivalence}, we show that $NG$ can be described as a mapping space from the one-point digraph in the $\infty$-category $\DGra_\infty$:
\begin{equation}
NG \simeq {\DGra}_\infty(*,G).
\end{equation}
By the universal property of the $\infty$-category of spaces $\Spc$, we obtain a colimit-preserving functor
\begin{equation}
R: \Spc \longrightarrow \DGra_\infty
\end{equation}
satisfying $R(*) \simeq *$. Then, using the natural equivalence $N \simeq {\DGra}_\infty(*,-)$ we show that $R$ and $N$ are mutually inverse equivalences.

\section{\bf Cubical homotopy groups}
\label{sec:Cubical_homotopy_groups}

This section reviews the ordinary homotopy theory of digraphs, in which a homotopy is a map  $G\otimes I_1 \to H$. We define the category of digraphs $\DGra$, its homotopy category $\DGra/\!\sim$, the cubical homotopy groups $A_*(G)$, and the path and loop digraphs. While versions of this theory appear in the literature (see \cite{CK24,GLMY23,LWYZ24}), our primary contribution is the introduction of the notion of a nice subcategory of shrinkings $\SS \subseteq \Shr$, where $\Shr$ is the category of intervals and their ``shrinkings''. We show that the cubical homotopy groups $A_*(G)$ can be defined using any nice subcategory of shrinkings $\SS$, yielding many equivalent definitions. We also define the path digraph $P_{\SS}G$ and the loop digraph $\Omega_{\SS}G$ relative to a specific  $\SS$. We prove that these digraphs satisfy all expected properties provided that $\SS$ is cofiltered.

\subsection{Homotopy category of digraphs} In this subsection we collect basic definitions and results concerning the category of digraphs, the category of pairs of digraphs, and their homotopy categories. Most statements are given without proof, as the proofs are straightforward; we include proofs only where indicated. Many of these statements are direct generalizations of results from \cite{CK24,GLMY23,LWYZ24}. 

\begin{definition}[Quivers]
Let ${\sf q}$ be the category consisting of two objects $[0]=\{0\}$ and $[1]=\{0,1\}$ and all monotone maps between them. This category is generated by morphisms 
$d^0,d^1 : [0] \to [1]$,   $s^0 : [1]\to [0],$
where $d^0(0)=1$,  $d^1(0)=0$ and $s^0$ is the only map from $[1]$ to $[0]$. They satisfy relations 
$s^0 d^0 = \id_{\{0\}} = s^0d^1.$
The category of quivers is defined as the category of presheaves over ${\sf q}$ 
\begin{equation}
    \Quiv = \Fun ({\sf q}^\op, \Set). 
\end{equation}
It is easy to see that a quiver $Q$ is defined by two sets 
$V(Q):=Q([0])$ and   $E(Q):=Q([1])$
called the set of vertices and the set of arrows, and three maps
\begin{equation}
d_0,d_1 : E(Q)\to V(Q), \hspace{1cm} s_0:V(Q)\to E(Q)
\end{equation}
satisfying the relations 
$d_0s_0 = \id_{V(Q)} = d_1s_0.$ For $v\in V(Q)$, the arrow $s_0(v)\in E(Q)$ is called a degenerate arrow, and it is not depicted in a picture of a quiver. 
\end{definition}
 
\begin{remark}[(Co)limits of quivers]
Since $\Quiv$ is a category of presheaves over a small category, it is complete and cocomplete, and the limits and colimits are computed objectwise.
\end{remark}
\begin{definition}[Digraphs] 
A digraph $G$ is a quiver for which the map $(d_1,d_0) : E(G)\to V(G)\times V(G)$ is injective. Every digraph is isomorphic to one in which $E(G)\subseteq V(G)\times V(G)$, the maps $d_1$ and $d_0$ are the two projections, and $s_0$ is the diagonal map; we henceforth work with this description. In it, the arrows of the form $(g,g)$ are the degenerate arrows, and they are not depicted in a picture of a digraph. The full subcategory of quivers spanned by the digraphs is denoted by 
\begin{equation}
\DGra \subseteq \Quiv.
\end{equation}
By an abuse of notation, we write $g \in G$ to denote a vertex of $G$, rather than the more formal $g \in V(G)$.
\end{definition}
\begin{proposition}
The subcategory $\DGra\subseteq \Quiv$ is reflective. The reflector 
\begin{equation}
{\sf d} : \Quiv \longrightarrow \DGra
\end{equation}
is defined so that $V({\sf d}(Q))=V(Q)$ and 
$
 E({\sf d}(Q)) $ is defined as the image of the map $(d_1,d_0):E(Q) \to V(Q)\times V(Q).$
\end{proposition}
\begin{corollary}[(Co)limits of digraphs]
The category $\DGra$ is complete and cocomplete. Since $\DGra\subseteq \Quiv$ is reflective, the inclusion $i:\DGra\to \Quiv$ is a right adjoint and hence creates limits. In particular, for any functor $F:C\to \DGra$, the limit is computed in $\Quiv$ and is already a digraph,
\begin{equation}
\label{eq:lim_d}
\lim F \cong \lim iF.
\end{equation}
Colimits in $\DGra$, on the other hand, are obtained as  colimits in $\Quiv$ followed by application of the reflector ${\sf d}$. More precisely, for a functor $F : C \to \DGra$, we have 
\begin{equation}
\label{eq:colim_d}
\colim\: F  = {\sf d}(\colim\: iF).
\end{equation}
\end{corollary}

\begin{remark}[${\sf d}$ does not preserve finite limits]
Note that the reflector ${\sf d}:\Quiv \to \DGra$ does not preserve finite limits; more precisely, it does not preserve equalizers. Indeed, let $I_1$ be the digraph with two vertices $0$ and $1$ and a single non-degenerate arrow, going from $0$ to $1$, and let $Q$ be the quiver with the same vertices but with two distinct non-degenerate arrows from $0$ to $1$. 
\begin{equation}
I_1: 
\begin{tikzcd}
0 
\ar[r]
& 
1 
\end{tikzcd}
\hspace{1cm}
Q: 
\begin{tikzcd}
0 
\ar[r, bend left=20]
\ar[r, bend right=20]
& 
1
\end{tikzcd}
\end{equation}
Consider the two quiver maps 
\begin{equation}
\varphi,\psi:I_1\to Q    
\end{equation}
that are identical on vertices and send the non-degenerate arrow $0\to 1$ to the two distinct arrows of $Q$. Their equalizer in $\Quiv$ is the discrete digraph $\{0,1\}$, which ${\sf d}$ leaves unchanged. On the other hand, ${\sf d}(Q)\cong I_1$ and ${\sf d}(\varphi)={\sf d}(\psi)=\id_{I_1}$, so the equalizer of ${\sf d}(\varphi)$ and ${\sf d}(\psi)$ is $I_1$. Since $\{0,1\}\not\cong I_1$, the reflector ${\sf d}$ does not preserve this equalizer.   
\end{remark}
 
\begin{lemma}
\label{lemma:DGra_closed_filtered_colimits}
The full subcategory $\DGra\subseteq \Quiv$ is closed under filtered colimits. 
\end{lemma}
\begin{proof}
Let $C$ be a filtered category and let $F:C\to \Quiv$ be a functor whose image lies in $\DGra$. By the definition of a digraph, for each $c\in C$ the map $(d_1,d_0):E(F(c)) \to V(F(c))^2$ is injective. Colimits in $\Quiv$ are computed objectwise, so $E(\colim\: F)=\colim_{c\in C} E(F(c))$ and $V(\colim\: F)=\colim_{c\in C} V(F(c))$. Since filtered colimits in $\Set$ commute with finite limits, they preserve monomorphisms and commute with finite products. Therefore the map 
\begin{equation}
E(\colim\: F)=\colim_{c\in C} E(F(c)) \longrightarrow \colim_{c\in C} V(F(c))^2 \cong \Big(\colim_{c\in C} V(F(c))\Big)^2 = V(\colim\: F)^2
\end{equation}
is injective. Hence $\colim\: F$ is a digraph, and so $\DGra$ is closed under filtered colimits in $\Quiv$.  
\end{proof}
 
\begin{proposition} 
\label{prop:filt_colimits_fin_limits}
In $\DGra$, filtered colimits commute with finite limits. 
\end{proposition}
\begin{proof}
In any presheaf category, filtered colimits commute with finite limits; in particular, this holds in $\Quiv$. Moreover, finite limits in $\DGra$ coincide with finite limits in $\Quiv$, since the inclusion $\DGra\hookrightarrow \Quiv$ creates limits, and filtered colimits in $\DGra$ coincide with filtered colimits in $\Quiv$ by Lemma \ref{lemma:DGra_closed_filtered_colimits}. The statement follows. 
\end{proof}

\begin{corollary}
\label{cor:filt_colimits_fin_limits_pointed}
In $\DGra_*$, filtered colimits commute with finite limits. 
\end{corollary}
\begin{proof}
The forgetful functor $U:\DGra_*\to \DGra$ creates nonempty limits and all colimits, and reflects isomorphisms. Hence, by Proposition \ref{prop:filt_colimits_fin_limits}, filtered colimits commute with nonempty finite limits in $\DGra_*$; they also commute with the terminal object, since a filtered colimit of the one-point pointed digraph is again the one-point pointed digraph. 
\end{proof}

\begin{definition}[Opposite digraph]
For a digraph $G$, we denote by $G^\op$ the digraph with the same set of vertices and the opposite directions of arrows 
$E(G^\op)={\sf tw}(E(G)),$
where ${\sf tw}:V(G)^2\to V(G)^2$ is defined by ${\sf tw}(g,g')=(g',g)$. 
\end{definition}

\begin{definition}[Graphs] 
A graph $G$ is a digraph such that $G=G^\op$. The category of graphs is the full subcategory of digraphs 
\begin{equation}
\Gra \subseteq \DGra.
\end{equation}
\end{definition}
\begin{proposition}
The subcategory $\Gra\subseteq \DGra$ is reflective. The reflector
\begin{equation}
{\sf g} : \DGra \longrightarrow \Gra
\end{equation}
is defined by the formulas $V({\sf g}(G))=V(G)$ and 
$E({\sf g}(G)) = E(G) \cup E(G^\op).$
\end{proposition}
\begin{remark}[Colimits of graphs]
The subcategory $\Gra\subseteq \DGra$ is closed under limits and colimits. 
\end{remark}

\begin{definition}[Induced subdigraph] A subquiver of a quiver is a sub-presheaf of a quiver. A subdigraph of a digraph is a subquiver of a digraph. A subdigraph $H\subseteq G$ is called induced if every arrow of $G$ between vertices of $H$ is an arrow of $H$.  
\end{definition} 

\begin{remark}[Pushout along an inclusion of an induced subdigraph]\label{rem:pushout_along_inclusion}
Let $H\subseteq G$ be an induced subdigraph and $\varphi:H\to H'$ be a digraph map. Then the pushout $G'=G\sqcup_H H'$ can be described as the commutative diagram 
\begin{equation}
\begin{tikzcd}
        H \arrow[r,"\varphi"] \arrow[d,"i"]
        \arrow[dr, phantom, "\ulcorner" very near end]
        & H' \arrow[d, "i'"] \\
        G \arrow[r,"\varphi'"] 
        & G'
    \end{tikzcd}    
\end{equation}  
where the set of vertices $V(G')$ is defined by 
\begin{equation}
V(G')=(V(G)\setminus V(H))\cup V(H'). 
\end{equation}
The map $\varphi':V(G)\to V(G')$ is defined by $\varphi'(g)=g$ for $g\in V(G)\setminus V(H)$ and $\varphi'(h)=\varphi(h)$ for $h\in V(H)$. The set of arrows is defined as 
\begin{equation}
E(G')= (\varphi'\times \varphi')(E(G)) \cup  E(H') .    
\end{equation}
As a corollary of this description of a pushout along an inclusion of an induced subdigraph we obtain an isomorphism 
\begin{equation}\label{eq:G_H}
G'\setminus H' \cong G\setminus H.
\end{equation}
\end{remark}

\begin{definition}[Box product] 
Let $G,H$ be digraphs. The box-product $G\otimes H$ is a digraph such that $V(G\otimes H)=V(G)\times V(H)$ and $E(G\otimes H)$ is the image of the map
\begin{equation}
(E(G) \times V(H)) \underset{V(G)\times V(H)}\coprod    (V(G)\times E(H)) \to (V(G)\times V(H))^2.
\end{equation} induced by the inclusions
\begin{equation}
 E(\Gamma)\hookrightarrow  V(\Gamma)\times V(\Gamma), \hspace{1cm} {\sf diag}:V(\Gamma)\to V(\Gamma)\times V(\Gamma)   
\end{equation}
for $\Gamma=G,H,$ and the twist map $V(G)^2 \times V(H)^2\to (V(G)\times V(H))^2$. In other words, there is an arrow $(g,h)\to (g',h')$ in $G\otimes H$ if and only if either there is an arrow $g\to g'$ and $h=h'$, or there is an arrow $h\to h'$ and $g=g'.$ The box product is natural in $G$ and $H.$ Note that the box product is a subdigraph of the product
\begin{equation}
G \otimes H \subseteq G\times H.
\end{equation}
\end{definition}
\begin{definition}[Box hom] 
The box hom of two digraphs $G$ and $H$ is defined as the digraph $\Hom^\otimes(G,H)$ whose set of vertices is $\Hom(G,H)$ and in which, for any two morphisms $\varphi,\psi :G\to H$, there is an arrow  $\varphi\to \psi$ if and only if there is an arrow  $\varphi(g)\to \psi(g)$ in $H$ for every $g\in G.$  The box hom is natural in $G$ and $H$.
\end{definition}
\begin{proposition}
$(\DGra,*, \otimes, \Hom^\otimes)$ forms a closed monoidal category, where $*$ is the one-point digraph. 
\end{proposition}

\begin{definition}[Connected components] 
For a digraph $G$, we denote by $\pi_0(G)$ the set of its connected components, that is, the quotient of $V(G)$ by the equivalence relation generated by $g\sim g'$ whenever $g\to g'$ or $g'\to g$ is an arrow of $G$. In other words, connectedness ignores the orientation of the arrows.
\end{definition}

\begin{definition}[Homotopy relation] \label{def:homotopy}
We say that two maps of digraphs $\varphi,\psi:G\to H$ are homotopic, denoted by $\varphi\sim \psi$, if they are in the same connected component of $\Hom^\otimes(G,H).$ The quotient set is denoted by 
\begin{equation}
[G,H] = \pi_0( \Hom^\otimes(G,H)).
\end{equation}  
\end{definition}
\begin{proposition} The homotopy relation is a congruence on $\DGra.$ Therefore, we can consider the homotopy category
\begin{equation}
\DGra/\!\sim.
\end{equation}
\end{proposition}

\begin{proposition}\label{prop:box_product_and_hom_preserve_homotopy}
The box product and box hom induce bifunctors
\[
-\otimes-:
(\mathrm{DGra}/\!\sim)^2
\longrightarrow
\mathrm{DGra}/\!\sim
\]
and
\[
\operatorname{Hom}^{\otimes}(-,-):
(\mathrm{DGra}/\!\sim)^{\mathrm{op}}
\times(\mathrm{DGra}/\!\sim)
\longrightarrow
\mathrm{DGra}/\!\sim.
\]
\end{proposition}

\begin{proof}
Let $f_0,f_1:A\to B$ be homotopic, and choose a homotopy
$h:A\otimes J\to B$, writing $h_j=h(-,j)$. For any $g:C\to D$, the
formulas
\[
((a,c),j)\longmapsto \bigl(h_j(a),g(c)\bigr),
\qquad
(\alpha,j)\longmapsto g\circ\alpha\circ h_j
\]
define homotopies from $f_0\otimes g$ to $f_1\otimes g$, and from
$\Hom^{\otimes}(f_0,g)$ to
$\Hom^{\otimes}(f_1,g)$, respectively. The second variable
is treated similarly. Changing the two variables successively proves that
both bifunctors descend to $\mathrm{DGra}/\!\sim$.
\end{proof}

\begin{definition}[Pairs of digraphs and pointed digraphs] 
A pair of digraphs is a pair $(G,H),$ where $H\subseteq G$ is an induced subdigraph. A morphism of pairs of digraphs $\varphi:(G,H)\to (G',H')$ is a morphism $\varphi:G\to G'$ such that $\varphi(H)\subseteq H'.$ The category of pairs of digraphs is denoted by $\PairDGr.$ It can be viewed as a full subcategory of the category of morphisms of digraphs 
\begin{equation}
\PairDGr \subseteq \Fun([1],\DGra),
\end{equation}
where $[1]=\{0<1\}$ is the category with two objects and one non-identity morphism. A pointed digraph is a pair of digraphs whose subdigraph is a single
vertex 
$(G,g) = (G,\{g\}).$
The full subcategory of pointed digraphs is denoted by 
\begin{equation}
\DGra_*\subseteq \PairDGr.
\end{equation}
\end{definition}

\begin{definition}[Box product of pairs]
The box product of pairs of digraphs $(G,H)$ and $(G',H')$ is defined by 
\begin{equation}
(G,H) \otimes (G',H') = (G\otimes G', (G \otimes H') \cup (H\otimes G')).
\end{equation}
Since $H\subseteq G$ and $H'\subseteq G'$ are induced, the subdigraph $(G \otimes H') \cup (H\otimes G')$ is an induced subdigraph of $G\otimes G'$, so the right-hand side is a well-defined pair of digraphs.
\end{definition}

\begin{definition}[Box hom of pairs] 
The box hom between  pairs of digraphs $(G,H)$ and $(G',H')$ is a digraph $\Hom^\otimes((G,H),(G',H'))$ whose set of vertices is $\Hom((G,H),(G',H'))$ and in which, for maps $\varphi,\psi:(G,H)\to (G',H')$, there is an arrow $\varphi\to \psi$ if and only if there is an arrow $\varphi(g)\to \psi(g)$ for every $g\in G$ and $\varphi(h)=\psi(h)$ for every $h\in H.$ 

We denote by $\underline{\Hom}^\otimes((G,H),(G',H'))$ a pair of digraphs 
\begin{equation}
\underline{\Hom}^\otimes((G,H),(G',H')) = (\Hom^\otimes((G,H),(G',H')), \Hom(G,H')),
\end{equation}
where $\Hom(G,H')$ is the induced subdigraph of $\Hom^\otimes((G,H),(G',H'))$ consisting of those maps $G\to G'$ whose image lies in $H'$. If $G=(G,g)$ and $G'=(G',g')$ are pointed digraphs, we set 
\begin{equation}
\Hom^{\otimes}_*(G,G') = \underline{\Hom}^{\otimes}(G,G').
\end{equation}
\end{definition}
\begin{proposition}
$(\PairDGr, (*,\emptyset), \otimes, \underline{\Hom}^\otimes )$ is a closed monoidal category. 
\end{proposition}
\begin{definition}[Relative homotopy relation] \label{def:relative_homotopy}
Two maps of pairs $\varphi,\psi:(G,H)\to (G',H')$ are called  homotopic relative to $H$, denoted by $\varphi \sim_{\sf rel}\psi$, if they are in the same connected component of $\Hom^\otimes((G,H),(G',H')).$ 
The quotient set is denoted by 
\begin{equation}
[(G,H),(G',H')]=\pi_0(\Hom^\otimes ((G,H),(G',H'))).
\end{equation}
\end{definition}

\begin{proposition}
\label{prop:homotopy_pairs}
The relative homotopy  relation is a congruence on $\PairDGr$. Therefore we can consider the homotopy category of pairs 
\begin{equation}
\PairDGr/\!\sim_{\sf rel}.
\end{equation}
\end{proposition}

\begin{proposition}
\label{prop:hom_adjunction}
For any three pairs of digraphs $(G,H),$ $(G',H')$ and $(G'',H'')$, there is a natural isomorphism
\begin{equation}
[(G,H)\otimes (G',H'), (G'',H'')]\cong [(G,H),\underline{\Hom}^\otimes((G',H'),(G'',H''))].
\end{equation}
\end{proposition}

\begin{definition}[Interval]
\label{def:interval} 
An interval is a digraph $J$ with vertices $\{0,\dots, k\}$ that has exactly $k$ non-degenerate arrows: for each $0\leq i<k$ either $i\to i+1$ or $i+1\to i$ is an arrow. 
\begin{equation}
 J : 0 \to 1 \to 2 \leftarrow 3 \to 4 \to 5 \leftarrow 6 
\end{equation}
If $k=0$, the interval is called trivial. 
For any interval $J$, we denote by $\partial J$ the discrete digraph defined by the endpoints $\{0,k\}.$ We also denote by $s,t:*\to J$ the inclusions of the endpoints $s(*)=0$ and $t(*)=k$. For any two intervals $J$ and $J'$ with the sets of vertices $\{0,\dots,k\}$ and $\{0,\dots,k'\}$ we denote by 
$J\vee J'$
their ``wedge'', i.e. the interval with the set of points $\{0,\dots,k+k'\}$ and the set of edges $E(J)\cup (E(J')+(k,k))$, where $E(J')+(k,k)=\{(i+k,j+k)\mid (i,j)\in E(J')\}$.
\end{definition}

\begin{definition}[Standard intervals]
We denote by $I_n,$ $n\geq 0,$ the interval with the set of vertices  $\{0,1,\dots, n\}$ and arrows $2i\to 2i+1$ and $2j+2 \to 2j+1$ for all possible $i,j$. We also consider the opposite interval $I^\op_n$.
\begin{equation}
I_4: \  0 \to 1 \leftarrow 2 \to 3 \leftarrow 4, \hspace{1cm}
I_4^\op: \  0 \leftarrow 1 \to 2 \leftarrow 3 \to 4.  
\end{equation}
The digraphs $I_n,I_n^\op, n\geq 0$ are called standard intervals. We will also use notations 
\begin{equation}
I_n^{+1} = I_n, \hspace{1cm} I_n^{-1}=I_n^\op.
\end{equation}
\end{definition}

\begin{remark}[Homotopy relation via intervals] 
Two digraph maps $\varphi,\psi: G\to G'$ are homotopic if and only if there is an interval $J$ with the set of vertices $\{0,\dots,k\}$ and a map $\eta:G\otimes J \to G'$ such that $\eta(g,0)  =\varphi(g)$ and $\eta(g,k)=\psi(g)$ for every $g\in G$. The map $\eta$ is called a homotopy from $\varphi$ to $\psi$.  

If $H\subseteq G$ and $H'\subseteq G'$ are  induced subdigraphs, then $\varphi,\psi:(G,H)\to (G',H')$ are homotopic relative to $H$  if and only if there exists a homotopy $\eta: G\otimes J\to G'$ from $\varphi$ to $\psi$ such that $\varphi(h)=\eta(h,j)=\psi(h)$ for every  $h\in H$ and every $j\in J.$  
\end{remark}

\begin{lemma}
\label{lemma:homotopy-as-a-coequaliser}
For any two pairs of digraphs $(G,H)$ and $(G',H')$, the set $[(G,H),(G',H')]$ is the coequalizer of the two maps   
\begin{equation}
\Hom((G\otimes I_1,H\otimes I_1),(G',H')) \rightrightarrows \Hom((G,H),(G',H'))
\end{equation}
induced by the endpoints $*\to I_1.$ 
\end{lemma}
\begin{corollary}
\label{cor:homotopy-as-a-coequaliser}
The conclusion of Lemma \ref{lemma:homotopy-as-a-coequaliser} holds with $I_1$ replaced by any non-trivial interval $J$: for any two pairs of digraphs $(G,H)$ and $(G',H')$, the set $[(G,H),(G',H')]$ is the coequalizer of the two maps   
\begin{equation}
\Hom((G\otimes J,H\otimes J),(G',H')) \rightrightarrows \Hom((G,H),(G',H'))
\end{equation}
induced by the endpoints $*\to J.$ 
\end{corollary}
\begin{proof}
By Lemma \ref{lemma:homotopy-as-a-coequaliser}, it suffices to show that $J$-homotopy generates the same equivalence relation on $\Hom((G,H),(G',H'))$ as $I_1$-homotopy. A $J$-homotopy restricts, along the vertices of $J$, to a path in $\Hom^\otimes((G,H),(G',H'))$ joining its two endpoints, so $J$-homotopic maps are homotopic. Conversely, since $J$ is non-trivial, collapsing it at one of its arrows yields a shrinking $\sigma:J\to I_1$ or $\sigma:J\to I_1^\op$, according to the direction of that arrow; precomposition with $\id_G\otimes\sigma$ then sends any $I_1$- or $I_1^\op$-homotopy to a $J$-homotopy with the same endpoints, so homotopic maps are $J$-homotopic. 
\end{proof}

\begin{proposition}
\label{prop:compact}
Let $(G,H)$ be a pair of finite digraphs. Then $(G,H)$ is a compact object of $\PairDGr$ and a compact object of  $\PairDGr/\!\sim_{\sf rel}.$  
\end{proposition}
\begin{proof}
The proof of the fact that a pair of finite digraphs $(G,H)$ is a compact object of $\PairDGr$ is standard (similar to the proof of the fact that a finite set is compact in the category of sets). The fact that $(G,H)$ is a compact object of $\PairDGr/\sim_{\sf rel}$ follows from the fact that the pair $(G\otimes I_1,H\otimes I_1)$  also consists of finite digraphs, Lemma \ref{lemma:homotopy-as-a-coequaliser}, and the fact that coequalizers commute with filtered colimits. 
\end{proof}

\subsection{Shrinkings}  
In this subsection, we introduce the category of shrinkings. Its objects are intervals (Definition \ref{def:interval}), and its morphisms are certain digraph maps between intervals called shrinkings. We will be interested not only in the category of shrinkings itself, but also in its various subcategories. To unify the subsequent theory, we introduce the notion of a nice subcategory of shrinkings and prove some of its properties.

\begin{definition}[Category of shrinkings] 
A shrinking is a morphism of intervals $J\to J'$ which is surjective and monotone on vertices \cite[Def.4.4]{GLMY23}. It is easy to see that shrinkings are closed under composition. The category of all intervals and shrinkings is denoted by
$\Shr$
and called the category of shrinkings.   
\end{definition}
\begin{definition}[Homotopy category of shrinkings]
The category of shrinkings can be treated as a subcategory of the category of pairs of digraphs 
\begin{equation}
    \Shr \subseteq \PairDGr
\end{equation}
whose objects are pairs $(J,\partial J)$ and morphisms are shrinkings.  Then the image of $\Shr$ in $\PairDGr/\sim_{\sf rel}$ is called the homotopy category of shrinkings and denoted by 
\begin{equation}
\HShr \subseteq \PairDGr/\!\sim_{\sf rel}.
\end{equation}
Its objects are intervals and its morphisms are shrinkings $J\to J'$ up to homotopy  relative to $\partial J$.  
\end{definition}

\begin{proposition}\label{propa:shrinkings_homo}
Any two shrinkings $J \to J'$ are homotopic relative to $\partial J.$ In other words, the homotopy category of shrinkings  $\HShr$ is a posetal category, i.e. a thin category: there is at most one morphism between any two of its objects. 
\end{proposition}
\begin{proof}
Let $J$ and $J'$ be two intervals with sets of vertices $\{0,\dots,n\}$ and $\{0,\dots,m\}$ and sets of arrows $E=E(J)$ and $E'=E(J').$ We prove that any two shrinkings $J\to J'$ are homotopic relative to $\partial J.$ The proof is by induction on $n+m.$ 
The base case $n+m=0$ is obvious. The cases $n<m$ and $m=0$ are also obvious, so we assume that $n\geq m\geq 1.$ We denote by $\bar J\subseteq J$ and $\bar J'\subseteq J'$ sub-intervals with the sets of vertices $\{0,\dots,n-1\}$ and $\{0,\dots,m-1\}$ respectively.  If the directions of the last arrows are different, $(n-1, n)\in E$ and $(m,m-1)\in E'$ or $(n,n-1)\in E$ and $(m-1,m)\in E'$, then  any shrinking $J\to J'$ factors through a shrinking $\bar J\to J'.$  In this case the result follows from the inductive assumption. 

Now assume that the directions of the last arrows are the same. Without loss of generality assume the last arrows are $(n-1,n)\in E$ and $(m-1,m)\in E'.$ Then there are shrinkings $s:J\to J'$ of two types: 
\begin{enumerate}
    \item $s(n-1)=m-1$; 
    \item  $s(n-1)=m$.
\end{enumerate}
Shrinkings of the first type are in bijection with shrinkings $\bar J\to \bar J'.$ Shrinkings of the second type are in bijection with shrinkings $\bar J\to J'.$ Therefore, by the inductive assumption all shrinkings of the first type are homotopic to each other relative to $\partial J$, and all shrinkings of the second type are homotopic to each other relative to $\partial J.$ If one of the two sets is empty, there is nothing to prove. So we assume that both of them are nonempty. Therefore, we need to show that a shrinking of the second type is homotopic to a shrinking of the first type relative to $\partial J.$

Take a shrinking of the second type $s:J\to J'$. Then $s(n-1) = m.$ Consider a map $s':\{0,\dots,n\}\to \{0,\dots, m\}$ defined by 
\begin{equation}
s'(x) = 
\begin{cases}
s(x), & s(x)<m\\
m-1, & x\neq n, s(x)=m,\\
m, & x=n.\\
\end{cases}
\end{equation}
We claim that $s'$ is a shrinking of the first type which is homotopic to $s$ relative to $\partial J.$ We first verify that $s'$ is a digraph map. For an arrow $(x,y)\in E$ such that $s(x),s(y)<m$, we have $(s'(x),s'(y))=(s(x),s(y))\in E'.$ If $(x,y)\in E$ such that $s(x)<m$ and $s(y)=m$, then $s(x)=m-1.$ In this case either $(s'(x),s'(y))=(m-1,m-1)\in E'$ or $(s'(x),s'(y))=(m-1,m)\in E'$. Since $(m-1,m)\in E$, the case that $(x,y)\in E$ and $s(x)=m$ and $s(y)<m$ is impossible. If $(x,y)\in E$ such that $s(x)=s(y)=m,$ we have either $(s'(x),s'(y))=(m-1,m-1)\in E'$, or $(s'(x),s'(y))=(m-1,m)\in E'$ or $(s'(x),s'(y))=(m,m)\in E$.

In order to check that $s$ and $s'$ are homotopic relative to $\partial J,$ we just note that for any $0\leq x\leq n,$ there is an arrow $s'(x)\to s(x),$ and that $s'(0)=s(0)$ and $s'(n)=s(n)$.
\end{proof}

\begin{proposition}\label{prop:HShr:cofiltered}
The posetal category $\HShr$ is cofiltered. 
\end{proposition}
\begin{proof}
It follows from the fact that for any two intervals $J$ and $J'$ there are shrinkings $J\vee J'\to J$ and $J\vee J'\to J'$. 
\end{proof}

\begin{remark}
The category $\Shr$ is not cofiltered. Indeed, for any two different shrinkings $s_1,s_2: J' \to J,$ there is no shrinking $s:J''\to J'$ such that $s_1s=s_2s.$  
\end{remark}

\begin{definition}[Nice subcategory of shrinkings]
\label{def:nice_subcategory}
A subcategory $\SS\subseteq \Shr$ is called a nice subcategory of shrinkings if the following two properties hold:
\begin{enumerate}
\item for any interval $J$, there is an interval $J'\in \SS$ such that there exists a shrinking $J'\to J;$
\item for any $J,J'\in \SS$ such that there is a shrinking $J'\to J,$ there is a morphism $J'\to J$ in $\SS$. 
\end{enumerate}
\end{definition}

\begin{remark}
Intuitively, property (1) in Definition \ref{def:nice_subcategory} means that for any finite sequence of right and left arrows there is an interval $J$ in $\SS$ whose sequence of arrows contains this sequence as a subsequence.    
\end{remark}

\begin{definition}[Cofinal functors] 
A functor $F:D\to C$ is called left cofinal if for any  $c\in C$, the relative over-category $F_{/c}$ is (not empty and) connected. Precomposition with left cofinal functors does not change limits (see, e.g., \cite{KS06}). Dually, a functor $F:D\to C$ is called right cofinal if for any $c\in C$, the relative under-category $F_{c/}$ is (not empty and) connected; precomposition with right cofinal functors does not change colimits. 
\end{definition}

\begin{remark}[Properties of left cofinal functors]
\label{rem:cofinal_facts}
We shall use two standard facts about cofinal functors (see, e.g., \cite{KS06,Cis19}). First, a finite product of left cofinal functors is again left cofinal: the relative over-category $\big(\prod_i F_i\big)_{/(c_1,\dots,c_n)}$ is the product $\prod_i (F_i)_{/c_i}$ of the individual over-categories, and a finite product of nonempty connected categories is nonempty and connected. Second, a small category $C$ is cofiltered if and only if the diagonal functor $C\to C^n$ is left cofinal for every $n\geq 0$. 
\end{remark}

\begin{proposition}
\label{prop:hom_lef_cof} If $\SS$ is a nice subcategory of shrinkings, the functor  
\begin{equation}
\SS \to \HShr
\end{equation}
is left cofinal. 
\end{proposition}
\begin{proof}
Let us denote the functor by $F:\SS\to \HShr$. We must show that the relative over-category $F_{/J}$ is connected for every $J\in \HShr$. Property (1) implies that $F_{/J}$ is nonempty. Take two objects $\alpha_1:F(J_1)\to J$ and $\alpha_2:F(J_2)\to J$ of $F_{/J}$. Consider an object $J_3$ of $\SS$ such that there is a shrinking $J_3\to J_1 \vee J_2$. Since we have shrinkings $J_1\vee J_2\to J_1$ and $J_1\vee J_2\to J_2$, this defines an object $\alpha_3 : F(J_3)\to J$ in $F_{/J}$ with two morphisms $\alpha_1 
\leftarrow \alpha_3 \to \alpha_2 $ in $F_{/J}$. Therefore $F_{/J}$ is connected.   
\end{proof}

\begin{definition}[Left, right and central shrinkings] 
The standard left and right shrinkings are defined as maps between standard intervals 
\begin{equation}
l : I_{n+1}^{-\varepsilon} \to I_{n}^{\varepsilon}, \hspace{1cm} r:I_{n+1}^\varepsilon\to I_{n}^\varepsilon,    
\end{equation}
where $n\geq 1,$ $\varepsilon\in \{-1,1\},$
defined by the formulas 
\begin{equation}
l(x)=\max(x-1,0), \qquad r(x)=\min(x,n) 
\end{equation}
(see \cite[p.2861]{CK24}).  We also consider central shrinkings $c=lr=rl$ and their squares 
\begin{equation}
c: I_{n+2}^{-\varepsilon} \to I^{\varepsilon}_{n}, \hspace{1cm}  
c^2 : I_{n+4} \to I_n.
\end{equation}
\end{definition}

\begin{example}[Examples of nice subcategories of shrinkings]\label{example:nice_subcategories}\ 
\begin{enumerate}
    \item $\Shr$ is a nice subcategory of shrinkings. 
    \item The full subcategory of $\Shr$ spanned by the standard intervals is a nice subcategory of shrinkings.  
    \item \label{standard_interval_category}
    The subcategory defined by the sequence 
    \[ \dots \xrightarrow{r} I_{4} \xrightarrow{l}  I^\op_{3} \xrightarrow{r}  I^\op_{2} \xrightarrow{l}  I_1 \]
    is a nice subcategory of shrinkings. Here we mean that all objects of this category are objects of this sequence and morphisms are compositions of morphisms from this sequence. This category is denoted by 
    \[{\sf St} \subseteq \Shr \]
    and called the standard interval category.  
    \item The subcategory defined by the sequence 
    \[ \dots \xrightarrow{r} I_{4} \xrightarrow{r} I_{3} \xrightarrow{r} I_{2} \xrightarrow{r} I_1 \]
    is a nice subcategory of shrinkings. 
    \item The subcategory defined by the sequence 
    \[ \dots \xrightarrow{l} I_{4}^\op \xrightarrow{l} I_{3} \xrightarrow{l} I_{2}^\op \xrightarrow{l} I_1 \]
    is a nice subcategory of shrinkings.  
    \item The subcategory defined by the sequence 
    \[\dots \xrightarrow{c^2} I_{16} \xrightarrow{c^2} I_{12} \xrightarrow{c^2} I_8
    \xrightarrow{c^2} I_4 \]
    is a nice subcategory of shrinkings. 
    \item The subcategory whose objects are the intervals $I_n$ and whose morphisms are the maps $ I_{nm}\to I_n$  defined by  $x\mapsto [x/m]$ with $m$ odd is a nice subcategory of shrinkings. 
\end{enumerate}
\end{example}

\begin{example}[Cantor intervals] Here we give another example of a nice subcategory of shrinkings.  
A Cantor interval $I^{\Cant}_n, n\geq 1$ is an interval whose set of vertices is $\{0,\dots,2^n-1\}$.  Arrows of $I^{\Cant}_n$ are defined as follows. Here we will use the binary number system $[a_0,\dots,a_{n-1}]=\sum 2^{n-i-1} a_i,$ $a_i\in \{0,1\}.$ Then there is an arrow 
\[[a_0,\dots,a_{k-1},0,1,\dots,1] \to [a_0,\dots,a_{k-1},1,0,\dots,0]\] for even $k$, and an arrow  \[[a_0,\dots,a_{k-1},1,0,\dots,0] \to [a_0,\dots,a_{k-1},0,1,\dots,1]\] for odd $k.$ For example: 
\[I_1^\Cant : [0] \to [1],\]
\[I_2^\Cant: [00] \leftarrow [01] \to [10] \leftarrow [11],\]
\[I^\Cant_3: [000] \to [001] \leftarrow [010] \to [011] \to [100] \to [101] \leftarrow [110] \to [111]. \]
For $m\geq n$ we define the projection map
\[{\sf pr}:I^\Cant_m \longrightarrow I^\Cant_n \]
by the formula ${\sf pr}([a_0,\dots,a_{m-1}]) = [a_0,\dots,a_{n-1}].$ The subcategory defined by the sequence of Cantor intervals 
    \[ \dots \xrightarrow{\sf pr} I^\Cant_4  \xrightarrow{\sf pr} I^\Cant_3 \xrightarrow{\sf pr} I^\Cant_2 \xrightarrow{\sf pr} I^\Cant_1\]
    is a nice subcategory of shrinkings.
\end{example}

\subsection{Cubical homotopy groups} 

In this subsection we show that several different definitions of cubical homotopy groups $A_*(G)$ of digraphs are equivalent. 

\begin{definition}[Cubical homotopy groups]
For a pointed digraph $G=(G,g)$  and $n\geq 0$, we set 
\begin{equation}
A_n(G) = \underset{(J_1,\dots,J_n)\in (\Shr^\op)^n }\colim\: \left[\bigotimes_{i=1}^n (J_i,\partial J_i),G\right],
\end{equation}
where the colimit is taken over $(\Shr^\op)^n$, the product of $n$ copies of the opposite category of shrinkings.   
\end{definition}

\begin{remark}
From the definition it is clear that if $G=(G,g)$ and $G'=(G',g')$ are pointed digraphs, then any pointed homotopy equivalence  $\varphi:G\to G'$ induces an isomorphism $A_*(G)\cong A_*(G')$. 
\end{remark}

\begin{remark}[The structure of a group]
In this section we do not use the group structure on $A_n(G).$ For now, it suffices to regard it as a set that depends naturally on the pointed digraph. The reader who wants to know an elementary definition of the group structure can find it in \cite[Prop.4.7]{LWYZ24} for digraphs and in \cite[Prop.3.5]{BKLW01} for graphs. Further, we will show that $A_n(G)\cong \pi_n(N_\infty G),$ where $N_\infty G$ is a cubical Kan complex (Proposition \ref{prop:A_n=pi_n}). This statement automatically determines the group structure on $A_n(G).$ 
\end{remark}

\begin{proposition}\label{prop:A_n:S_i}
Let $\SS_1,\dots,\SS_n$ be $n$ nice subcategories of shrinkings. Then, for any pointed digraph $G$ and $n\geq 0$, the map 
\begin{equation}
\underset{(J_1,\dots,J_n)\in \SS_1^\op \times \dots \times \SS_n^\op}{\colim}\: \left[\bigotimes_{i=1}^n (J_i,\partial J_i),G\right] \longrightarrow A_n(G)
\end{equation}
is a bijection. 
\end{proposition} 
\begin{proof}
By the definition $A_n(G)$ is defined as a colimit of a functor $F:(\Shr^n)^\op\to \Set$, where $F(J_1,\dots,J_n) = \left[\bigotimes_{i=1}^n (J_i,\partial J_i),G\right]$. It is easy to see that this functor factors through a functor from the homotopy category  $\tilde F:(\HShr^n)^\op\to \Set$. By Proposition \ref{prop:hom_lef_cof} the functors $\Shr\to\HShr$ and $\SS_i\to\HShr$ are left cofinal, hence so are the products $\Shr^n\to\HShr^n$ and $\SS_1\times\dots\times\SS_n\to\HShr^n$ by Remark \ref{rem:cofinal_facts}. Passing to opposite categories, $(\Shr^n)^\op \to (\HShr^n)^\op $ and $\SS_1^\op\times \dots \times \SS_n^\op \to (\HShr^n)^\op$ are right cofinal. The result then follows from the fact that precomposition with right cofinal functors preserves colimits. 
\end{proof}

\begin{proposition}
Let $\SS$ be a nice subcategory of shrinkings. Then, for any pointed digraph $G$  and $n\geq 0$,  the map 
\begin{equation}
\label{eq:colim_C}
\underset{J\in \SS^\op}{\colim}\: \left[(J,\partial J)^{\otimes n},G\right] \longrightarrow A_n(G)
\end{equation}
is a bijection. 
\end{proposition}
\begin{proof}
Since $\HShr$ is cofiltered (Proposition \ref{prop:HShr:cofiltered}), the diagonal functor $\HShr \to \HShr^n$ is left cofinal by Remark \ref{rem:cofinal_facts}. By Proposition \ref{prop:hom_lef_cof}, the functor $\SS\to \HShr$ is left cofinal. Therefore $\SS^\op \to (\HShr^n)^\op$ is right cofinal.  The rest of the proof is similar to the proof of Proposition \ref{prop:A_n:S_i}.
\end{proof}
\begin{corollary}
For a pointed digraph $G$ and $n\geq 0$, there is a natural bijection 
\begin{equation}
A_n(G) \cong \underset{J\in \Shr^\op}\colim [(J,\partial J)^{\otimes n},G].
\end{equation}
\end{corollary}

\begin{corollary}
For any pointed digraph $G$ and $n\geq 0$, the cubical homotopy group $A_n(G)$ is isomorphic to the colimit of each of the following sequences. 
\begin{align}
& [(I_1,\partial I_1)^{\otimes n},G] \xrightarrow{l^*} [(I_2^\op,\partial I_2^\op)^{\otimes n},G] \xrightarrow{r^*} [(I_3^\op,\partial I_3^\op)^{\otimes n},G] \xrightarrow{l^*} \dots \\
&[(I_1,\partial I_1)^{\otimes n},G] \xrightarrow{r^*} [(I_2,\partial I_2)^{\otimes n},G] \xrightarrow{r^*} [(I_3,\partial I_3)^{\otimes n},G] \xrightarrow{r^*} \dots \\
&[(I_1,\partial I_1)^{\otimes n},G] \xrightarrow{l^*} [(I_2^\op,\partial I_2^\op)^{\otimes n},G] \xrightarrow{l^*} [(I_3,\partial I_3)^{\otimes n},G] \xrightarrow{l^*} \dots \\
&[(I_4,\partial I_4)^{\otimes n},G] \xrightarrow{(c^2)^*} [(I_8,\partial I_8)^{\otimes n},G] \xrightarrow{(c^2)^*} [(I_{12},\partial I_{12})^{\otimes n},G] \xrightarrow{(c^2)^*}  \dots \\
&[(I_1^\Cant,\partial I^\Cant_1)^{\otimes n},G] \xrightarrow{{\sf pr}^*} [(I_2^\Cant,\partial I^\Cant_2)^{\otimes n},G] \xrightarrow{{\sf pr}^*}  \dots 
\end{align}
\end{corollary}

\begin{definition}[Boundary of a cube]
\label{def:cube_boundary}
For an interval $J$ and $n\geq 0$, the boundary $\partial(J^{\otimes n})$ is the induced subdigraph of $J^{\otimes n}$ given by the second component of the pair $(J,\partial J)^{\otimes n}$; explicitly,
\begin{equation}
\partial(J^{\otimes n}) = \bigcup_{i=1}^n J^{\otimes (i-1)}\otimes \partial J\otimes J^{\otimes(n-i)} \ \subseteq\ J^{\otimes n}
\end{equation}
(for $n=0$ the union is empty, so $\partial(J^{\otimes 0})=\emptyset$). 
\end{definition}

\begin{definition}[Digraph spheres]
For any interval $J$ and $n\geq 0$, we denote by $S^n_J$ the quotient of $J^{\otimes n}$ by its boundary $\partial(J^{\otimes n})$ (Definition \ref{def:cube_boundary}).  
\begin{equation}
\begin{tikzcd}
\partial(J^{\otimes n}) 
\ar[r]
\ar[d]
\arrow[dr, phantom, "\ulcorner" very near end]
& 
J^{\otimes n} 
\ar[d]
\\
* 
\ar[r]
& 
S^n_J
\end{tikzcd}
\end{equation}
If $J=I_m$, we denote it by $S^n_m.$ 
\end{definition}

\begin{corollary}
For a nice subcategory of shrinkings $\SS$, a pointed digraph $G$ and $n\geq 0$, there is a natural bijection 
\begin{equation}
A_n(G) \cong \underset{J\in \SS^\op}\colim\:  [S^n_{J},G]_*.
\end{equation}
\end{corollary}

\subsection{Path and loop digraphs} 

Both Carranza--Kapulkin \cite{CK24} and Li--Wu--Yau--Zhang \cite{LWYZ24} introduce notions of a path digraph and a loop digraph. While their definitions satisfy similar properties, they are not equivalent, even when restricted to the case of ordinary graphs. In this subsection, we present a unified approach to the path digraph construction. We define a path digraph, denoted by $P_{\SS} G$, which depends on a nice subcategory of shrinkings $\SS$ (Definition \ref{def:nice_subcategory}). By varying the choice of $\SS$, one recovers the specific versions developed in both \cite{CK24} and \cite{LWYZ24}. However, while our definition subsumes both previous constructions, the subsequent theory we develop generalizes only that of Carranza--Kapulkin, as our results rely on the assumption that $\SS$ is cofiltered.

\begin{definition}[Path and loop digraphs]
Let $\SS$ be a nice subcategory of shrinkings (Definition \ref{def:nice_subcategory}) and $G$ be a digraph. Then the path digraph $P_\SS G$ is defined by the formula 
\begin{equation}
P_\SS G = \underset{J\in \SS^\op}\colim\: \Hom^\otimes(J,G). 
\end{equation}
This construction defines a functor 
$P_\SS : \DGra \to \DGra.$
The inclusions of the endpoints $j_0,j_1: *\to J$ define two maps $p_0,p_1: P_\SS G \to G$. 

If $G=(G,g)$ is a pointed digraph, the loop digraph $\Omega_\SS G$ is defined by the formula 
\begin{equation}
\Omega_\SS G = \underset{J\in \SS^\op}{\colim}\: \Hom^\otimes_*(S^1_J,G). 
\end{equation}
This construction defines a functor 
$\Omega_\SS : \DGra_* \to \DGra_*.$
\end{definition} 
\begin{remark} The path digraph considered by 
Li--Wu--Yau--Zhang in \cite{LWYZ24} is the path  digraph $P_\SS G$, where $\SS=\Shr$ is the whole category of shrinkings. The path digraph considered by Carranza--Kapulkin in \cite{CK24} is the path digraph $P_\SS G$, where $\SS = {\sf St}$ is the standard interval category (see Example  \ref{example:nice_subcategories}). 
We find the choice $\SS={\sf St}$ more natural, as ${\sf St}$  is cofiltered whereas $\Shr$ is not. However, there are many other cofiltered nice subcategories of shrinkings that can be useful. Therefore, we develop this theory for an arbitrary cofiltered nice subcategory of shrinkings $\SS$. 
\end{remark}

\begin{lemma} 
\label{lemma:Omega_pullback}
For a cofiltered nice subcategory of shrinkings
 $\SS$, the commutative square 
\begin{equation}
\begin{tikzcd}
\Omega_\SS G 
\ar[r]
\ar[d]
\arrow[dr, phantom, "\lrcorner",  very near start]
& 
P_\SS G
\ar[d,"{(p_0,p_1)}"]
\\
* 
\ar[r,"{(g,g)}"]
& 
G\times G
\end{tikzcd}
\end{equation}
is a pullback. 
\end{lemma}
\begin{proof}
For any $J$, the square 
\begin{equation}
\begin{tikzcd}
\Hom_*^\otimes (S^1_J,G) 
\ar[r]
\ar[d]
\arrow[dr, phantom, "\lrcorner",  very near start]
& 
\Hom^\otimes(J,G)
\ar[d,"{(p_0,p_1)}"]
\\
* 
\ar[r,"{(g,g)}"]
& 
G\times G
\end{tikzcd}
\end{equation}
is a pullback. Note that the forgetful functor $\DGra_*\to \DGra$ creates nonempty limits and all colimits. Then the statement follows from the fact that, in $\DGra$, filtered colimits commute with finite limits (Proposition \ref{prop:filt_colimits_fin_limits}).
\end{proof}

\begin{proposition}
\label{prop:omega_finite_limits} For a cofiltered nice subcategory of shrinkings
 $\SS$, the functors $\Omega_\SS$ and $P_\SS$ preserve finite limits.   
\end{proposition}
\begin{proof}
It follows from the fact that filtered colimits commute with finite limits in $\DGra$ (Proposition \ref{prop:filt_colimits_fin_limits}) and in $\DGra_*$ (Corollary \ref{cor:filt_colimits_fin_limits_pointed}), together with the fact that the functors $\Hom^\otimes(J,-)$ and $\Hom^\otimes_*(S^1_J,-)$, being right adjoints, preserve finite limits.
\end{proof}

\begin{proposition}
For a cofiltered nice subcategory of shrinkings
 $\SS$, any pointed digraph $G$ and any $n\geq 0$, there is a natural isomorphism  
\begin{equation}
A_{n+1}(G)\cong A_n( \Omega_\SS G ).
\end{equation}
\end{proposition}
\begin{proof}
By Proposition \ref{prop:hom_adjunction}, there is a natural isomorphism between the set 
\begin{equation}
[(J_1,\partial J_1)\otimes \dots \otimes (J_n,\partial J_n) \otimes (J_{n+1},\partial J_{n+1}), G]    
\end{equation}
and the set
\begin{equation}
[(J_1,\partial J_1)\otimes \dots \otimes (J_n,\partial J_n) , \underline{\Hom}^\otimes( (J_{n+1}, \partial J_{n+1}),G)]
\end{equation}
Note that 
\begin{equation}
    \underline{\Hom}^\otimes( (J_{n+1}, \partial J_{n+1}),G) \cong  \Hom_*^\otimes( S^1_{J_{n+1}},G).
\end{equation}
Then the statement follows from the fact that $(J_1,\partial J_1)\otimes \dots \otimes (J_n,\partial J_n)$ is a compact object of $\PairDGr/\sim_{\sf rel}$ (Proposition \ref{prop:compact}) and Proposition \ref{prop:A_n:S_i}.
\end{proof}

\begin{corollary}
For a cofiltered nice subcategory of shrinkings
 $\SS$, any pointed digraph $G$ and any $n\geq 0$, there is a natural isomorphism 
 \begin{equation}
    A_n(G) \cong \pi_0( \Omega^n_\SS G). 
 \end{equation}
\end{corollary}

\section{\bf Cubical nerves of digraphs}
\label{sec:Cubical nerves of digraphs}

In this section we study the cubical nerve functors $N_m \colon \DGra \to \cSet$ from the category of digraphs to the category of cubical sets, together with their colimit $N_\infty \colon \DGra \to \cSet$. We show that, for every digraph $G$, the cubical set $N_\infty G$ is a fibrant replacement of $N_1 G$, and that $A_*(G)\cong \pi_*(N_\infty G)\cong \pi_*(|N_1 G|)$. These results generalize those of \cite{CK24}.

\subsection{Cubical sets}

In this subsection we collect the facts about cubical sets that we use later. More detailed accounts of cubical sets with connections can be found in \cite{Cis14,CK23,CK24}.

\begin{definition}[(Co)cubical objects]
By the \emph{box category} $\Box$ we mean the box category with connections (both positive and negative). It is the subcategory of the category of posets whose objects are the cubes $[1]^n$, $n\geq 0$, and whose morphisms are the composites of the maps
\begin{alignat}{2}
 \partial^{i,\varepsilon}&\colon [1]^{n-1}\to [1]^{n}, \hspace{1cm} &&1\leq i\leq n, \ \varepsilon\in \{0,1\},
 \\
\sigma^i &\colon [1]^n \to [1]^{n-1}, \hspace{1cm}
&&1\leq i\leq n,
\\
\gamma^{i,\varepsilon}&\colon [1]^n \to [1]^{n-1}, \hspace{1cm}
&&1\leq i\leq n-1, \ \varepsilon\in\{0,1\},
\end{alignat}
called \emph{coface}, \emph{codegeneracy} and \emph{coconnection} maps respectively, given by the formulas
\begin{align}
\partial^{i,\varepsilon}(x_1,\dots,x_{n-1})
&=
(x_1,\dots,x_{i-1},\varepsilon,x_i,\dots,x_{n-1}),\\
\sigma^i(x_1,\dots,x_n) &= (x_1,\dots,x_{i-1},x_{i+1},\dots,x_n),\\
\gamma^{i,0}(x_1,\dots,x_n) &= (x_1,\dots,x_{i-1},\max(x_i,x_{i+1}),
x_{i+2},\dots,x_n),\\
\gamma^{i,1}(x_1,\dots,x_n) &= (x_1,\dots,x_{i-1},\min(x_i,x_{i+1}),
x_{i+2},\dots,x_n).
\end{align}
For a category $C$, the category of cubical objects of $C$ is the functor category
\begin{equation}
{\sf c}C=\Fun(\Box^\op,C).
\end{equation}
Equivalently, a cubical object is a sequence of objects $X_n$, $n\geq 0$, equipped with maps
\begin{alignat}{2}
 &\partial_{i,\varepsilon}\colon X_n\to X_{n-1}, \hspace{1cm} &&1\leq i\leq n, \ \varepsilon\in \{0,1\},   \\
&\sigma_i \colon X_{n-1}\to X_n, \hspace{1cm}
&&1\leq i\leq n,
\\
&\gamma_{i,\varepsilon}\colon X_{n-1}\to X_n, \hspace{1cm}
&&1\leq i\leq n-1, \ \varepsilon\in\{0,1\},
\end{alignat}
called \emph{face}, \emph{degeneracy} and \emph{connection} maps respectively, subject to the cubical identities (see \cite{Cis14,CK23}):
\begin{equation}
\partial_{i,\varepsilon} \partial_{j,\varepsilon'} = \partial_{j,\varepsilon'} \partial_{i+1,\varepsilon} \ \text{ for } j\leq i,
\hspace{10mm}
\partial_{i,\varepsilon}\sigma_j =
\begin{cases}
\sigma_j \partial_{i-1,\varepsilon},&  j<i,\\
\id, &  j=i,\\
\sigma_{j-1}\partial_{i,\varepsilon},&  j>i,
\end{cases}
\end{equation}

\begin{equation}
\sigma_{j}\sigma_i = \sigma_{i+1}\sigma_j \ \text{ for } j\leq i, \hspace{10mm} \gamma_{i,\varepsilon}\gamma_{j,\varepsilon'}
=
\begin{cases}
\gamma_{j+1,\varepsilon'} \gamma_{i,\varepsilon}, & j>i,\\
\gamma_{i+1,\varepsilon} \gamma_{i,\varepsilon}, & i=j,\ \varepsilon=\varepsilon',
\end{cases}
\end{equation}
\begin{equation}
\partial_{i,\varepsilon} \gamma_{j,\varepsilon'}
=
\begin{cases}
\gamma_{j,\varepsilon'}\partial_{i-1,\varepsilon},
&
j<i-1,\\
\id, & j\in \{i-1,i\},\ \varepsilon=\varepsilon',\\
\sigma_j\partial_{j,\varepsilon}, & j\in \{i-1,i\},\ \varepsilon\neq \varepsilon',\\
\gamma_{j-1,\varepsilon'} \partial_{i,\varepsilon}, & j>i,
\end{cases}
\hspace{5mm}
\gamma_{i,\varepsilon} \sigma_j =
\begin{cases}
\sigma_j \gamma_{i-1,\varepsilon}, & j<i,\\
\sigma_i \sigma_i, & j=i,\\
\sigma_{j+1}\gamma_{i,\varepsilon}, & j>i.
\end{cases}
\end{equation}
Cocubical objects of $C$ are cubical objects of $C^\op$. \emph{Cubical sets} are the cubical objects in the category of sets,
\begin{equation}
\cSet = \Fun(\Box^\op,\Set);
\end{equation}
in other words, cubical sets are presheaves on the box category. For a cubical set $X$, the elements of $X_n=X([1]^n)$ are called the \emph{$n$-cubes} of $X$. An $n$-cube $x$ is \emph{non-degenerate} if it cannot be written as $x=\sigma_i(y)$ or $x=\gamma_{i,\varepsilon}(y)$ for some $(n-1)$-cube $y$.
\end{definition}

\begin{notation}[Standard cube, boundary, horn and sphere]
We use the standard notation for the basic cubical sets. The standard cubes, their boundaries, horns and spheres are denoted as follows:
\begin{alignat}{2}
 &\Box^n= \Box(-,[1]^n), \hspace{1cm}
 &&\partial \Box^n = \bigcup_{(i,\varepsilon)} \partial^{i,\varepsilon} (\Box^{n-1}), \\
&\sqcap^n_{i,\varepsilon}= \bigcup_{(j,\delta)\neq (i,\varepsilon)} \partial^{j,\delta} (\Box^{n-1}),\hspace{1cm}  &&S^n = \Box^n/\partial \Box^n.
\end{alignat}
\end{notation}

\begin{theorem}[{\cite[Th.~1.34, Prop.~1.35]{DKLS24}, \cite[Th.~1.7]{Cis14}}]
There is a proper, cofibrantly generated model structure on $\cSet$ in which
\begin{alignat}{2}
&\sqcap^n_{i,\varepsilon}\hookrightarrow \Box^n, \hspace{1cm}  &&1\leq i\leq n,\ \varepsilon\in \{0,1\},
\intertext{is a set of generating acyclic cofibrations,}
&\partial \Box^n \hookrightarrow \Box^n, \hspace{1cm} &&n\geq 0,
\end{alignat}
is a set of generating cofibrations, and the cofibrations are precisely the monomorphisms. Moreover, the weak equivalences are closed under finite products.
\end{theorem}

\begin{definition}[Kan fibrations, Kan complexes, anodyne maps]
In this model structure, the fibrations are called \emph{Kan fibrations}, the fibrant objects are called \emph{Kan complexes}, and the acyclic cofibrations are called \emph{anodyne maps}.
\end{definition}

\begin{definition}[Pointed cubical sets]
For any model category $M$, one may form the model category $M_*$ of pointed objects, in which a morphism is a cofibration, fibration or weak equivalence precisely when its underlying morphism in $M$ is \cite[Prop.~1.1.8]{Hov07}. We write $\cSet_*$ for the resulting model category of pointed cubical sets.
\end{definition}

\begin{definition}[Geometric realization and singular cubical set]
There is an evident cocubical topological space given by the topological cubes $[0,1]^n$, with connections induced by the maps $\max,\min\colon [0,1]^2\to [0,1]$. By the universal property of the Yoneda embedding \cite[Prop.~II.1.3]{GZ12}, this cocubical object induces an adjunction
\begin{equation}\label{eq:geom:Sing}
|\cdot | \colon \cSet \leftrightarrows  {\sf Top} \colon {\sf Sing}_\Box,
\end{equation}
where $|\cdot|$ is the colimit-preserving functor with $|\Box^n|\cong [0,1]^n$ and ${\sf Sing}_\Box(T)_n = \Hom_{\sf Top}([0,1]^n, T)$.
\end{definition}
\begin{theorem}[{\cite[Cor.~2.25]{CK23}}]
\label{th:quillen_equivalence}
The adjunction \eqref{eq:geom:Sing} is a Quillen equivalence, where ${\sf Top}$ carries the classical Quillen model structure.
\end{theorem}
\begin{corollary}
\label{cor:weak_equiv}
A map of cubical sets $f$ is a weak equivalence if and only if its geometric realization $|f|$ is a weak equivalence.
\end{corollary}
\begin{proof}
Since $(|\cdot|,\Sing_\Box)$ is a Quillen equivalence and every object of $\cSet$ is cofibrant, this follows from \cite[Cor.~1.3.16, Lem.~1.1.12]{Hov07}.
\end{proof}
\begin{corollary}
\label{cor:coproducts_of_weak_equivalences}
Weak equivalences of cubical sets are closed under (possibly infinite) coproducts.
\end{corollary}

\begin{corollary}
For any topological space $T$, the cubical set $\Sing_\Box(T)$ is a Kan complex.
\end{corollary}

\begin{definition}[Geometric product of cubical sets]
Let $\otimes\colon\Box \times \Box \to \Box$ be the functor defined by $[1]^n\otimes [1]^m = [1]^{n+m}$. The \emph{geometric tensor product} of cubical sets is the left Kan extension
\begin{equation}
\begin{tikzcd}
\Box\times \Box \ar[r,"\otimes"]
\ar[d]
&
\Box \ar[r]
&
\cSet \\
\cSet \times \cSet
\ar[rru,"\otimes"']
\end{tikzcd}
\end{equation}
This makes $\cSet$ a monoidal category with unit $\Box^0$. Moreover, geometric realization carries the geometric product to the product of spaces:
\begin{equation}
|X\otimes Y| \cong |X|\times |Y|.
\end{equation}
The geometric product is not symmetric in general, that is, $X\otimes Y \not \cong Y\otimes X$; however, it is biclosed. For a cubical set $X$ we write
\begin{equation}
\Hom_L(X,-), \hspace{1cm}
\Hom_R(X,-)
\end{equation}
for the right adjoints of $-\otimes X$ and $X\otimes -$ respectively.
\end{definition}

\begin{proposition}[Cylinder]
\label{prop:geometric_cylinder}
For any cubical set $X$, the object $X\otimes \Box^1$, equipped with the maps
\begin{equation}
 X \sqcup X \xrightarrow{\ (\partial^0,\partial^1)\ } X \otimes \Box^1 \xrightarrow{\ \sf pr \ } X,
\end{equation}
is a cylinder object for the model structure on $\cSet$. Likewise, for any pointed cubical set $(X,x_0)$, the object $(X\otimes \Box^1)/(\{x_0\} \otimes \Box^1)$, equipped with the induced maps
\begin{equation}
 X \vee  X \longrightarrow  (X\otimes \Box^1)/(\{x_0\} \otimes \Box^1)  \longrightarrow  X,
\end{equation}
is a cylinder object in $\cSet_*$.
\end{proposition}
\begin{proof}
The map $X\sqcup X \to X \otimes \Box^1$ is clearly a monomorphism. Since $|X|\times [0,1] \to |X|$ is a weak equivalence, so is $X\otimes \Box^1\to X$ by Corollary~\ref{cor:weak_equiv}. Hence $X \otimes \Box^1$ is a cylinder object in $\cSet$. The pointed case is analogous.
\end{proof}

\begin{definition}[Homotopy of cubical sets]
Let $A$ and $X$ be cubical sets. We write $\sim$ for the left-homotopy equivalence relation on $\cSet(A,X)$ generated by elementary homotopies which is determined by the functorial cylinder object $A\otimes \Box^1$, and we denote the quotient set by
\begin{equation}
[A,X]= \cSet(A,X)/\sim.
\end{equation}
When $X$ is a Kan complex, $[A,X]$ is isomorphic to the hom-set in the homotopy category, $[A,X]\cong {\sf Ho}(\cSet)(A,X)$.

A map of cubical sets $f\colon X\to Y$ is a \emph{homotopy equivalence} if there exists a map $g\colon Y\to X$ with $fg\sim \id$ and $gf\sim \id$.

In the pointed setting, pointed homotopy is defined via the cylinder object $(A \otimes \Box^1)/(\{a_0\}\otimes \Box^1)$, and the corresponding quotient set is denoted
\begin{equation}
 [A,X]_* = \cSet_*(A,X)/\sim.
\end{equation}
When $X$ is a pointed Kan complex,
$[A,X]_* \cong {\sf Ho}(\cSet_*)(A,X).$
\end{definition}

\begin{proposition}\label{prop:homotopy_equiv}
A homotopy equivalence of cubical sets is a weak equivalence.
\end{proposition}
\begin{proof}
If $f$ is a homotopy equivalence of cubical sets, then $|f|$ is a homotopy equivalence of spaces, and the claim follows from Corollary~\ref{cor:weak_equiv}.
\end{proof}

\begin{definition}[Homotopy group of a Kan complex]
For a pointed Kan complex $X\in \cSet_*$, the $n$-th homotopy group is
\begin{equation}
\pi_n(X) = [S^n,X]_*.
\end{equation}
\end{definition}
\begin{theorem}[{\cite[Th.~3.24, 3.25]{CK23}}]
\label{th:homotopy:of:topological}
For any pointed Kan complex $X$ and any pointed topological space $T$, there are natural isomorphisms
\begin{equation}
\pi_*(X)\cong \pi_*(|X|), \hspace{1cm} \pi_*(T)\cong \pi_*(\Sing_\Box(T)).
\end{equation}
\end{theorem}

\begin{corollary}\label{cor:weak_equivalences_of_kan}
Let $f\colon X\to Y$ be a map of Kan complexes. Then the following are equivalent:
\begin{enumerate}
\item $f$ is a weak equivalence;
\item $f$ is a homotopy equivalence;
\item $f$ induces a bijection $\pi_0(X)\to\pi_0(Y)$, and for every $x_0\in X_0$ and every $n\geq 1$, the map
$\pi_n(X,x_0)\to \pi_n(Y,f(x_0))$ is an isomorphism.
\end{enumerate}
\end{corollary}
\begin{proof}
This follows from Theorem~\ref{th:homotopy:of:topological}, Corollary~\ref{cor:weak_equiv}, and the fact that in any model category the homotopy equivalences between fibrant--cofibrant objects coincide with the weak equivalences \cite[Prop.~1.2.8]{Hov07}.
\end{proof}

\subsection{Cubical nerves}

\begin{definition}[$J$-nerve and $J$-realization]
\label{def:J-nerve}
Let $J\in \DGra$ be an interval with vertex set $\{0,\dots,m\}$ (Definition~\ref{def:interval}). The maximum and minimum functions evidently define digraph maps
\begin{equation}
\max \colon J\otimes J \to J, \hspace{1cm} \min\colon J \otimes J \to J.
\end{equation}
These allow us to define a cocubical digraph
\begin{equation}
\Box \longrightarrow \DGra, \hspace{1cm} [1]^n \mapsto J^{\otimes n},
\end{equation}
whose coface, codegeneracy and coconnection maps are given by
\begin{align}
\partial^{i,\varepsilon}(j_1,\dots,j_{n-1}) &= (j_1,\dots,j_{i-1},\varepsilon m , j_{i},\dots, j_{n-1}), \hspace{1cm} \varepsilon\in \{0,1\},\\
\sigma^i(j_1,\dots,j_n) &= (j_1,\dots,j_{i-1},j_{i+1},\dots,j_n),\\
\gamma^{i,0}(j_1,\dots,j_n) &= (j_1,\dots,j_{i-1},\max(j_i,j_{i+1}),
j_{i+2},\dots,j_n),\\
\gamma^{i,1}(j_1,\dots,j_n) &= (j_1,\dots,j_{i-1},\min(j_i,j_{i+1}),
j_{i+2},\dots,j_n).
\end{align}
By the universal property of the Yoneda embedding we obtain an adjunction
\begin{equation}
|\cdot|_J \colon \cSet \leftrightarrows \DGra \colon N_J,
\end{equation}
where $|\cdot|_J$ is the colimit-preserving functor determined by $|\Box^n|_J=J^{\otimes n}$, and $(N_JG)_n=\Hom(J^{\otimes n},G)$. The functor $|\cdot|_J$ is called the \emph{$J$-realization}, and $N_J$ the \emph{cubical $J$-nerve}. We also use the notation
\begin{equation}
    |\cdot |^\varepsilon_{n} = |\cdot|_{I_n^\varepsilon}, \hspace{1cm} N_n^\varepsilon = N_{I_n^\varepsilon},
\end{equation}
where $\varepsilon\in \{-1,1\}$, together with
\begin{equation}
|\cdot|_n=|\cdot|_n^{+1}, \hspace{5mm}  |\cdot|_n^{\op} = |\cdot |^{-1}_{n},\hspace{5mm}  N_n = N^{+1}_n,\hspace{5mm}  N^\op_n=N^{-1}_n.
\end{equation}
\end{definition}

\begin{proposition}
\label{prop:|otimes|_J}
For cubical sets $X$ and $Y$ and an interval $J$, there is a natural isomorphism
\begin{equation}
|X\otimes Y|_J \cong |X|_J \otimes |Y|_J.
\end{equation}
\end{proposition}
\begin{proof}
Both $|-\otimes -|_J$ and $|-|_J\otimes |-|_J$ are left Kan extensions of their restrictions to $\Box\times \Box \to \DGra$, so it suffices to prove the isomorphism on these restrictions. There it follows from the natural isomorphism $J^{\otimes n} \otimes J^{\otimes m} \cong J^{\otimes n+m}$.
\end{proof}

\begin{proposition}
\label{prop:N_J_hom_L}
Let $X$ be a cubical set, $G$ a digraph and $J$ an interval. Then there are isomorphisms
\begin{equation}
\Hom_L(X,N_JG)\cong N_J(\Hom^\otimes(|X|_J,G)) \cong \Hom_R(X,N_JG),
\end{equation}
natural in $X$ and $G$.
\end{proposition}
\begin{proof}
By Proposition~\ref{prop:|otimes|_J}, the square
\begin{equation}
\begin{tikzcd}
\cSet
\ar[rr,"-\otimes X"]
\ar[d,"|-|_J"]
&&
\cSet
\ar[d,"|-|_J"]
\\
\DGra
\ar[rr,"-\otimes |X|_J"]
&& \DGra
\end{tikzcd}
\end{equation}
commutes up to natural isomorphism. Hence so does the corresponding square of right adjoints,
\begin{equation}
\begin{tikzcd}
\DGra
\ar[rr,"{\Hom^\otimes(|X|_J,-)}"]
\ar[d,"{N_J}"]
&&
\DGra
\ar[d,"{N_J}"]
\\
\cSet
\ar[rr,"{\Hom_L(X,-)}"]
&& \cSet
\end{tikzcd}
\end{equation}
The natural isomorphism of left adjoints in the preceding square is natural in $X$. Taking right-adjoint mates therefore yields an isomorphism natural in $X$. The second isomorphism is proved similarly.
\end{proof}

\begin{definition}[The infinite cubical nerve]\label{def:infinite_nerve}
The \emph{infinite cubical nerve} $N_{\infty}G$ is the colimit
\begin{equation}
N_{\infty} G = \underset{J\in {\sf St^{\op}}}\colim\: N_JG,
\end{equation}
where $J$ ranges over the standard interval category (Example~\ref{example:nice_subcategories}~\eqref{standard_interval_category}). Equivalently, $N_{\infty}G$ is the colimit of the sequence
\begin{equation}
\label{N_inf1}
N_1G
\xrightarrow{l^*}
N_2^{\op}G
\xrightarrow{r^*}
N_3^{\op}G
\xrightarrow{l^*}
N_4 G
\xrightarrow{r^*}
\dots
\end{equation}
Equivalently, $N_\infty G$ can be computed by the following direct sequence:
\begin{equation}\label{N_inf4}
N_{4} G \xrightarrow{(c^2)^*} N_8 G \xrightarrow{(c^2)^*} N_{12} G \xrightarrow{(c^2)^*} N_{16} G \xrightarrow{(c^2)^*} \dots.
\end{equation} 
\end{definition}

\begin{proposition}
\label{prop:N-inf-fin-limits}
The functor $N_\infty \colon \DGra \to \cSet$ preserves finite limits.
\end{proposition}
\begin{proof}
Each $N_m$ preserves finite limits, and filtered colimits in $\cSet$ commute with finite limits.
\end{proof}

\subsection{\texorpdfstring{$N_\infty G$}{} is a Kan complex}
\label{subsection:N-Kan}

In this subsection we generalize \cite[Th.~4.1(1)]{CK24} from graphs to digraphs.

\begin{theorem}
\label{thm:kancomplex}
For any digraph $G$, the cubical set $N_\infty G$ is a Kan complex.
\end{theorem}

\begin{lemma}\label{lem:extension}
For all integers $m\geq 0$, $n\geq 1$, every $1\leq i\leq n$, and every
$\varepsilon\in\{0,1\}$, there is a digraph map
\begin{equation}
\Phi\colon |\Box^n|_{6m} \to |\sqcap^n_{i,\varepsilon}|_{2m}
\end{equation}
making the following diagram commute:
\begin{equation}
\label{eq:triang_Mengmeng}
\begin{tikzcd}
{|\sqcap^n_{i,\varepsilon}|_{6m}}
\ar[r,"c^{2m}"]
\ar[d,hookrightarrow]
&
{|\sqcap^n_{i,\varepsilon}|_{2m}}
\\
{|\Box^n|_{6m}}
\ar[ru,"\Phi"']
\end{tikzcd}
\end{equation}
\end{lemma}
\begin{proof}
The case $n=1$ is clear, so assume $n\geq 2$. Consider the map $\phi\colon I_{6m}\to I_{6m}$ given by
\begin{equation}
\phi(x) =
\begin{cases}
2m, & x\leq 2m,\\
x, & 2m\leq  x \leq 4m,\\
4m, & 4m \leq x,
\end{cases}
\end{equation}
and for $1\leq i\leq n$ the map $D_i: I^{\otimes n}_{6m} \to I_{2m}$ defined by
\begin{equation}
D_i(v)= \max\limits_{1\leq j\neq i\leq n}\{|v_j-\phi (v_j)|\}.
\end{equation}
Recall that $c^{2m}\colon I_{6m} \to I_{2m}$ is given by
\begin{equation}
c^{2m}(x)=
\begin{cases}
0,   & x \le 2m,\\
x-2m, & 2m \le x \le 4m,\\
2m,   & x \ge 4m.
\end{cases}
\end{equation}
We now define a map $\tilde \Phi_{i,\varepsilon}\colon I^{\otimes n}_{6m} \to I^{\otimes n}_{2m}$ by
\begin{equation}
\tilde  \Phi_{i,\varepsilon}(v)_{k} =
\begin{cases}
 c^{2m}(v_k), & k\neq i,\\
 \max\{c^{2m}(v_i),\, 2m-D_i(v)\}, & k=i,\ \varepsilon=0,\\
 \min\{c^{2m}(v_i),\, D_i(v)\}, & k=i,\ \varepsilon=1.
\end{cases}
\end{equation}

The following figure illustrates the map $\tilde\Phi_{2,0}:I_6^{\otimes 2}\to I_2^{\otimes2}$:

\begin{figure}[H]
\centering
\begin{tikzpicture}[
  x=0.92cm,y=-0.56cm,
  phi vertex/.style={
    font=\scriptsize,
    inner sep=1pt,outer sep=0.6pt,
    rounded corners=1.2pt
  },
  phi edge/.style={
    -{Stealth[length=1.1mm,width=0.85mm]},
    line width=0.3pt,draw=black!55
  },
  phi target/.style={
    font=\small,
    inner sep=2.2pt,outer sep=0.8pt,
    rounded corners=2pt,line width=0.35pt
  }
]
  \definecolor{phi00}{HTML}{2F5D9B}
  \definecolor{phi01}{HTML}{006F82}
  \definecolor{phi02}{HTML}{236F55}
  \definecolor{phi12}{HTML}{875C00}
  \definecolor{phi22}{HTML}{A94B18}
  \definecolor{phi21}{HTML}{B43D67}
  \definecolor{phi20}{HTML}{744B9D}
  \definecolor{phiInk}{HTML}{283C50}

  \foreach \a in {0,...,6} {
    \pgfmathtruncatemacro{\phix}{min(2,max(0,\a-2))}
    \pgfmathtruncatemacro{\phibase}{min(2,\a,6-\a)}
    \foreach \b in {0,...,6} {
      \pgfmathtruncatemacro{\phiy}{
        max(min(2,max(0,\b-2)),\phibase)
      }
      \edef\phicolor{phi\phix\phiy}
      \node[
        phi vertex,
        text=\phicolor,
        fill=\phicolor!7
      ] (s-\a-\b) at (\a,\b) {$(\phix,\phiy)$};
    }
  }

  \foreach \a/\b in {0/1,2/1,2/3,4/3,4/5,6/5} {
    \foreach \k in {0,...,6} {
      \draw[phi edge] (s-\a-\k) -- (s-\b-\k);
      \draw[phi edge] (s-\k-\a) -- (s-\k-\b);
    }
  }

  \draw[
    -{Stealth[length=2mm]},
    line width=0.55pt,draw=phiInk
  ]
    (6.6,3) --
    node[above,font=\normalsize] {$\tilde\Phi_{2,0}$}
    (7.7,3);

  \foreach \name/\u/\v/\px/\py in {
    00/0/0/8.7/1,
    01/0/1/8.7/3,
    02/0/2/8.7/5,
    12/1/2/9.95/5,
    22/2/2/11.2/5,
    21/2/1/11.2/3,
    20/2/0/11.2/1
  } {
    \edef\phicolor{phi\name}
    \node[
      phi target,
      text=\phicolor,
      fill=\phicolor!9,
      draw=\phicolor!25
    ] (t-\name) at (\px,\py) {$(\u,\v)$};
  }
  \foreach \from/\to in {00/01,02/01,02/12,22/12,22/21,20/21} {
  \draw[phi edge] (t-\from) -- (t-\to);
}
\end{tikzpicture}
\end{figure}

We first check that $\tilde \Phi_{i,\varepsilon}$ is a digraph map. Take a non-degenerate arrow $v\to v'$ in $I_{6m}^{\otimes n}$. Then there is an index $k$ such that $v_k\to v_k'$ is a non-degenerate arrow in $I_{6m}$ and $v_j= v_j'$ for all $j\neq k$. By the definition of $I_{6m}$, the vertex $v_k$ is even, $v_k'$ is odd, and $|v_k-v_k'|=1$. We must produce an arrow $\tilde \Phi_{i,\varepsilon}(v)\to \tilde \Phi_{i,\varepsilon}(v')$ in $I_{2m}^{\otimes n}$. We consider three cases.
\begin{itemize}
\item[(1)] Suppose $k\neq i$ and $c^{2m}(v_k) =c^{2m}(v_k')$. Since $v_j=v'_j$ for $j\neq k$, we have $\tilde \Phi_{i,\varepsilon}(v)_s =\tilde \Phi_{i,\varepsilon}(v')_s$ for all $s\neq i$, so it remains to produce an arrow $\tilde \Phi_{i,\varepsilon}(v)_i\to \tilde \Phi_{i,\varepsilon}(v')_i$. The equality $c^{2m}(v_k)=c^{2m}(v_k')$ forces either $0\leq v_k,v_k'\leq 2m$ or $4m \leq v_k,v'_k \leq  6m$. Using that $v_k$ is even and $v_k'$ is odd, we obtain arrows $D_i(v)\to D_i(v')$ and $2m - D_i(v) \to 2m - D_i(v')$ (recall that there is always an arrow from a vertex to itself). Hence there is an arrow $\tilde \Phi_{i,\varepsilon}(v)_i\to \tilde \Phi_{i,\varepsilon}(v')_i$.

\item[(2)] Suppose $k\neq i$ and there is a non-degenerate arrow $c^{2m}(v_k) \to c^{2m}(v_k')$. Then $2m\leq v_k , v_k'\leq 4m$ and $D_i(v) = D_i(v')$. Thus $\tilde \Phi_{i,\varepsilon}(v)_j = \tilde \Phi_{i,\varepsilon}(v')_j$ for $j\neq k$, and there is an arrow $\tilde \Phi_{i,\varepsilon}(v)_k \to \tilde \Phi_{i,\varepsilon}(v')_k$. Hence there is an arrow $\tilde\Phi_{i,\varepsilon}(v) \to \tilde\Phi_{i,\varepsilon}(v')$.

\item[(3)] Suppose $k=i$. Then $\tilde \Phi_{i,\varepsilon}(v)_j = \tilde \Phi_{i,\varepsilon}(v')_j$ for $j\neq i$. Since $v_j=v_j'$ for $j\neq i$, we have $D_i(v) = D_i(v')$, so the arrow $c^{2m}(v_i)\to c^{2m}(v_i')$ induces an arrow $\tilde \Phi_{i,\varepsilon}(v)_i \to \tilde \Phi_{i,\varepsilon}(v')_i$.
\end{itemize}
Therefore $\tilde \Phi_{i,\varepsilon}$ is a digraph map.

Next we check that the image of $\tilde \Phi_{i,\varepsilon}$ lies in the subdigraph
\begin{equation}
|\sqcap^n_{i,\varepsilon} |_{2m} \subseteq I_{2m}^{\otimes n},
\end{equation}
which consists of the vertices $v=(v_1,\dots,v_n)$ such that either $v_k\in \{0,2m\}$ for some $k\neq i$, or $v_i=(1-\varepsilon) 2m$. Fix a vertex $v=(v_1,\dots,v_n)\in I_{6m}^{\otimes n}$ and consider two cases.
\begin{itemize}
\item[(1)] Suppose $c^{2m}(v_k)\in \{0,2m\}$ for some $k\neq i$. Then $\tilde \Phi_{i,\varepsilon}(v)_k\in \{0,2m\}$, so $\tilde \Phi_{i,\varepsilon}(v)\in |\sqcap^n_{i,\varepsilon}|_{2m}$.

\item[(2)] Suppose $0< c^{2m}(v_k)<2m$ for all $k\neq i$. Then $2m < v_k < 4m$ for all $k\neq i$, and hence $D_i(v) =0$. It follows that
\[\tilde \Phi_{i,0}(v)_i =\max\{c^{2m}(v_i),\, 2m-D_i(v)\}=2m,\qquad
\tilde \Phi_{i,1}(v)_i=\min\{D_i(v),\,c^{2m}(v_i)\}=0.\]
Thus $\tilde \Phi_{i,\varepsilon}(v)_i = (1-\varepsilon) 2m$ and $\tilde \Phi_{i,\varepsilon}(v)  \in |\sqcap^n_{i,\varepsilon}|_{2m}$.
\end{itemize}
Therefore we proved that $\tilde \Phi_{i,\varepsilon}$ sends vertices of $I^{\otimes n}_{6m}$ to $|\sqcap^n_{i,\varepsilon} |_{2m}$. Moreover, for $m\ge 1$, it is obvious that $|\sqcap^n_{i,\varepsilon}|_{2m}\subseteq I^{\otimes n}_{2m}$ is an induced subdigraph. Therefore we deduce that $\tilde\Phi_{i,\varepsilon}$ factors through an induced digraph map
\begin{equation}
\Phi_{i,\varepsilon} \colon I^{\otimes n}_{6m} \longrightarrow |\sqcap^n_{i,\varepsilon}|_{2m}.
\end{equation}
For $m=0$ the case is trivial.

Finally, we verify that diagram~\eqref{eq:triang_Mengmeng} commutes; that is, that $\Phi_{i,\varepsilon}(v)=c^{2m}(v)$ for every $v=(v_1,\dots,v_n)\in |\sqcap^n_{i,\varepsilon}|_{6m} \subseteq I^{\otimes n}_{6m}$. Since $\Phi_{i,\varepsilon}(v)_k=c^{2m}(v_k)$ for all $k\neq i$, it suffices to check that $\Phi_{i,\varepsilon}(v)_i=c^{2m}(v_i)$. We distinguish four cases.
\begin{itemize}
\item[(1)] Suppose $\varepsilon=0$ and $v_k \in  \{0,6m\}$ for some $k\neq i$. Then $D_i(v) =2m$ and $\max\{c^{2m}(v_i),\, 2m-D_i(v)\} =  c^{2m}(v_i)$.
\item[(2)] Suppose $\varepsilon=0$ and $v_i=6m$. Then $c^{2m}(v_i) = 2m$, so $\max\{c^{2m}(v_i),\, 2m-D_i(v)\} = 2m = c^{2m}(v_i)$.
\item[(3)] Suppose $\varepsilon=1$ and $v_k\in \{ 0,6m\}$ for some $k\neq i$. Then $D_i(v) =2m$ and $\min\{D_i(v),\, c^{2m}(v_i)\} = c^{2m}(v_i)$.
\item[(4)] Suppose $\varepsilon=1$ and $v_i= 0$. Then $c^{2m}(v_i) = 0$, so $\min\{D_i(v),\, c^{2m}(v_i)\} = 0 = c^{2m}(v_i)$.
\end{itemize}
Hence the diagram commutes.
\end{proof}
\begin{proof}[Proof of Theorem~\ref{thm:kancomplex}]
Since $\sqcap^n_{i,\varepsilon}$ has finitely many non-degenerate cubes, it is a compact object of $\cSet$. Using the direct sequence  \eqref{N_inf4} of Definition~\ref{def:infinite_nerve}, we obtain an isomorphism
\begin{equation}
\Hom_{\cSet}(\sqcap^n_{i,\varepsilon}, N_\infty G)\cong
\underset{m}\colim \:
\Hom_{\cSet}(\sqcap^n_{i,\varepsilon},N_{4m}G),
\end{equation}
where the transition maps are induced by $c^{2}\colon I_{4(m+1)}\to I_{4m}$. Hence, for any $f\colon \sqcap^n_{i,\varepsilon} \to N_\infty G$, there exist an integer $m\geq 1$ and a cubical map $g\colon \sqcap^n_{i,\varepsilon} \to N_{4m}G$ whose composite with $N_{4m}G \to N_\infty G$ equals $f$. Denote by $\bar g\colon  |\sqcap^n_{i,\varepsilon}|_{4m} \to G$ the adjoint map. Applying Lemma~\ref{lem:extension} with $m$ replaced by $2m$, we obtain a map $\bar h\colon   |\Box^n|_{12m} \to G$ making the diagram
\begin{equation}
 \begin{tikzcd}
{|\sqcap^n_{i,\varepsilon}|_{12m}} \ar[r,"{c^{4m}}"]
\ar[d,hookrightarrow]
&
{|\sqcap^n_{i,\varepsilon}|_{4m}} \ar[r,"{\overline{g}}"]
&
G
\\
{|\Box^n|_{12m}}
\ar[urr,"\bar h"']
& &
\end{tikzcd}
\end{equation}
commute. Consequently there is a map $h\colon \Box^n \to N_{12m} G$ making the diagram
\begin{equation}
\begin{tikzcd}
\sqcap^n_{i,\varepsilon} \ar[r,"g"]
\ar[d]
& N_{4m} G
\ar[d,"(c^{4m})^*"]
\\
\Box^n \ar[r,"h"]
&
N_{12m} G
\end{tikzcd}
\end{equation}
commute. Therefore $f$ extends to a map $\hat f\colon\Box^n \to N_\infty G$, namely the composite of $h$ with $N_{12m} G \to N_\infty G$.
\end{proof}

\subsection{Equivalence of cubical nerves}
\label{subsection:equivalence_N}

In this subsection we prove that for any digraph $G$, the canonical map $ N_1G\to N_\infty G $ is anodyne. We will prove a slightly stronger result: for any interval $J,K\neq I_0$, and shrinking map $s:K\to J$, the induced morphism of nerves $s^*:N_JG\to N_K G$ is anodyne for all digraph $G$.
	
	\begin{lemma}\label{lem:N_J_weak_equiv}
		For any interval $J\neq I_0$, the $J$-nerve $N_J$ sends digraph homotopy equivalences to weak equivalences of cubical sets. 
	\end{lemma}
    \begin{proof}
    	Let $ h_0:G\otimes J\to H $ be an elementary digraph homotopy from $f_0:G\to H$ to $f_1:G\to H$ parametrized by $J$. Note that this parametrization is valid from Corollary~\ref{cor:homotopy-as-a-coequaliser}. By adjunction we consider the map $\tilde{h}_0:G\to \Hom^\otimes(J,H)$. As a result of Proposition~\ref{prop:N_J_hom_L}, after taking $N_J$ we have
        \begin{equation}
            (\tilde{h}_0)_*:=N_J(\tilde{h}_0):N_JG \to N_J\Hom^\otimes(J,H)\cong \Hom_L(\Box^1,N_JH).
        \end{equation}
        By adjunction we obtain a morphism 
        \begin{equation}
            (h_0)_*:N_JG\otimes \Box^1\to N_JH,
        \end{equation}
        which is a homotopy from $N_J(f_0)$ to $N_J(f_1)$. 
        
        Now suppose that we have a digraph homotopy $h:G\otimes I_n\to H$ from $f$ to $g$. Let $f_i=h(-,i)$ for $i\in\{0,\ldots,n\}$, so $f=f_0$ and $g=f_n$. Then for each elementary step we proved that $N_J(f_i)\simeq N_J(f_{i+1})$. Therefore on the level of topological spaces, we know that $|N_J(f)|$ is homotopic to $|N_J(g)|$. 

        As a consequence, for any digraph homotopy equivalence $\varphi:G\to H$, we proved that $|N_J(\varphi)|:|N_JG|\to |N_JH|$ is a homotopy equivalence of topological spaces, therefore $N_J(\varphi)$ is in particular a weak equivalence of cubical sets.
    \end{proof}

    \begin{definition}
        Let $K,J$ be two intervals. For any digraph $G$ we define the bicubical set $N_{K,J}G$ as
        \begin{equation}
            (N_{K,J}G)_{m,n}=\Hom_{\DGra}(K^{\otimes m}\otimes J^{\otimes n},G).
        \end{equation}
        For shrinking maps $s:K'\to K$ and $t:J'\to J$, there is a canonical bicubical map
        \begin{equation}
            \varphi(s,t):N_{K,J}G \to N_{K',J'}G
        \end{equation}
        induced by
        \begin{equation}
            (K')^{\otimes m}\otimes (J')^{\otimes n}\xrightarrow{s^{\otimes m}\otimes t^{\otimes n}}K^{\otimes m}\otimes J^{\otimes n}.
        \end{equation}
    \end{definition}

    \begin{lemma}\label{lem:diag_of_shrinking_is_weak_equiv}
        For any intervals $K,J,K',J'$ with $K,J'\neq I_0$, shrinking maps $s:K'\to K$ and $t:J'\to J$, the diagonal map
        \begin{equation}
            \diag \varphi(s,t):\diag N_{K,J}G \to \diag N_{K',J'}G
        \end{equation}
        is a weak equivalence for all digraph $G$.
    \end{lemma}
    \begin{proof}
        For any fixed column $n$, we have 
        \begin{equation}
            (N_{K,J}G)_{\bullet,n}=N_{K}(\Hom^{\otimes}(J^{\otimes n},G)).
        \end{equation}
        Therefore the cubical map
        \begin{equation}
            \varphi(\id_K,t)_{\bullet,n}:(N_{K,J}G)_{\bullet,n}\to (N_{K,J'}G)_{\bullet,n}
        \end{equation}
        is obtained by applying $N_K$ on the digraph map
        \begin{equation}
            \Hom^{\otimes}(t^{\otimes n},\id):\Hom^{\otimes}(J^{\otimes n},G)\to \Hom^{\otimes}((J')^{\otimes n},G).
        \end{equation}
        Since $t^{\otimes n}:(J')^{\otimes n}\to J^{\otimes n}$ is a digraph homotopy equivalence, by Proposition~\ref{prop:box_product_and_hom_preserve_homotopy} and the fact that $t:J'\to J$ is a digraph homotopy equivalence, we know that $\Hom^{\otimes}(t^{\otimes n},\id)$ is a digraph homotopy equivalence, then Lemma~\ref{lem:N_J_weak_equiv} implies that $\varphi(\id_K,t)_{\bullet,n}$ is a weak equivalence. Therefore we proved that $\varphi(\id_K,t)$ is a columnwise weak equivalence, then by \cite[Prop. 3.12]{ACK25} we proved that
        \begin{equation}
            \diag \varphi(\id_K,t): \diag N_{K,J}G \to \diag N_{K,J'}G
        \end{equation}
        is a weak equivalence. Similarly, we can show that
        \begin{equation}
            \diag \varphi(\id_{J'},s): \diag N_{J',K}G \to \diag N_{J',K'}G
        \end{equation}
        is a weak equivalence. Let 
        \begin{equation}
            \tau: \Fun(\Box^{\op}\times \Box^{\op},\Set)\to \Fun(\Box^{\op}\times \Box^{\op},\Set)
        \end{equation}
        denote the functor interchanging the two cubical directions, that is
        \begin{equation}
            (\tau N_{J,K}G)_{m,n}=(N_{J,K}G)_{n,m}=\Hom_{\DGra}(J^{\otimes n}\otimes K^{\otimes m},G).
        \end{equation}
        Then there is a canonical isomorphism
        \begin{equation}
            \phi_{K,J}: N_{K,J}G \xrightarrow{\cong} \tau N_{J,K}G
        \end{equation}
        induced by the natural isomorphism $K^{\otimes m}\otimes J^{\otimes n}\cong J^{\otimes n}\otimes K^{\otimes m}$. Therefore we obtain a canonical isomorphism
        \begin{equation}
            \diag \phi_{K,J}: \diag N_{K,J}G \xrightarrow{\cong} \diag\tau N_{J,K}G=\diag N_{J,K}G.
        \end{equation}
        As a result, we proved that the following composition
        \begin{equation}
            \diag \varphi(s,t)=\diag \phi_{J',K'}\circ \diag \varphi(\id_{J'},s)\circ \diag\phi_{K,J'}\circ\diag \varphi(\id_K,t)
        \end{equation}
        is a weak equivalence as desired.
    \end{proof}

    \begin{proposition}\label{prop:shrinking_induces_weak_equiv}
        For any interval $K,J$ with $K,J\neq I_0$, and any shrinking map $s:K\to J$, the induced cubical map
        \begin{equation}
            s^*:N_JG\to N_KG
        \end{equation}
        is anodyne for all digraph $G$.
    \end{proposition}
    \begin{proof}
        It is straightforward to see that $s^*$ is a monomorphism. Therefore it suffices to prove that it is a weak equivalence. Consider the following commutative diagram:
        \begin{equation}
            \begin{tikzcd}[column sep=5em, row sep=2em]
                N_JG \ar[r,"s^*"]\ar[d,"{\varphi(\id_J,*)}"'] & N_KG \ar[d,"{\varphi(\id_K,*)}"] \\
                \diag N_{J,J}G \ar[r,"{\diag \varphi(s,\id_J)}"'] & \diag N_{K,J}G
            \end{tikzcd}
        \end{equation}
        where $*$ stands for the terminal map $J\to I_0$. By Lemma~\ref{lem:diag_of_shrinking_is_weak_equiv} the vertical morphisms and the bottom morphism are weak equivalences, therefore the top morphism $s^*$ is also a weak equivalence.
    \end{proof}

    \begin{theorem}\label{th:N_1toN}
        For any digraph $G$, the canonical map 
        \begin{equation}
            N_1G \to N_\infty G
        \end{equation}
        is anodyne.
    \end{theorem}
    \begin{proof}
        By Definition~\ref{def:infinite_nerve}
        \begin{equation}
            N_\infty G =\colim (N_1G \xrightarrow{l^*} N^{\op}_2G \xrightarrow{r^*} N_3^{\op}G \xrightarrow{l^*}N_4 G \xrightarrow{r^*}\cdots).
        \end{equation}
        Therefore, by Proposition~\ref{prop:shrinking_induces_weak_equiv} we deduce that $N_1G \to N_\infty G$ is the composition of the above sequence of anodyne maps, which is also anodyne.
    \end{proof}
    \begin{remark}\label{rem:CK24-gap}
The argument above corrects a gap in the proof of \cite[Th.~4.1(ii)]{CK24}, specifically in the proof of \cite[Th.~4.12, pp.~2888--2889]{CK24}. We use the notation of that paper. Fix $m>0$, and let
\[
X_0=\operatorname{im}\bigl(l^*:N_mG\hookrightarrow N_{m+1}G\bigr).
\]
Write $\langle-\rangle$ for the cubical subset of $N_{m+1}G$ generated by the specified cubes, including their faces, degeneracies, and connections. For $n\geq1$, the relevant stages of the filtration are
\[
X_{n,n}
=
\left\langle
X_0,\;
x\bar\lambda^k_{m,i-1}
\;\middle|\;
\begin{array}{l}
1\leq i\leq k\leq n,\\
x:I_{m+1}^{\otimes i}\otimes I_m^{\otimes(k-i)}\to G
\end{array}
\right\rangle
\]
and
\[
X_{n+1,1}
=
\left\langle
X_{n,n},\;
x\bar\lambda^{n+1}_{m,0}
\;\middle|\;
x:I_{m+1}\otimes I_m^{\otimes n}\to G
\right\rangle.
\]
In particular, the definition of $X_{n+1,1}$ includes the cube $x\bar\lambda^{n+1}_{m,0}$ for every such $x$, without imposing any condition on its distinguished face.

However, for the inclusion $X_{n,n}\hookrightarrow X_{n+1,1}$, the proposed open box attachments are indexed only by
\[
S_{n+1,1}
=
\left\{
x(\operatorname{id}_{I_{m+1}}\otimes l^{\otimes n})
\;\middle|\;
\begin{array}{l}
x:I_{m+1}\otimes I_m^{\otimes n}\to G,\\
x(\operatorname{id}_{I_{m+1}}\otimes l^{\otimes n})
\notin X_{n,n}
\end{array}
\right\}.
\]
The corresponding attaching cube is $x\bar\lambda^{n+1}_{m,0}$, whose distinguished face satisfies
\[
x\bar\lambda^{n+1}_{m,0}\partial_{n+2,1}
=
x(\operatorname{id}_{I_{m+1}}\otimes l^{\otimes n}).
\]
Thus, if this face is outside $X_{n,n}$, then so is the attaching cube. However, the converse need not hold: one can have
\[
x\bar\lambda^{n+1}_{m,0}\notin X_{n,n},
\qquad
x(\operatorname{id}_{I_{m+1}}\otimes l^{\otimes n})
\in X_{n,n}.
\]
Such a cube is included in the definition of $X_{n+1,1}$ but omitted from the proposed attachments. Consequently, the assertion on p.~2889 that these attachments account for all new generators of $X_{n+1,1}$ is false, and the claimed pushout description does not hold in general.
\end{remark}

\subsection{Cubical weak equivalences}

In this subsection we prove the isomorphism $A_*(G)\cong \pi_*(N_\infty G)$ and give several equivalent characterizations of cubical weak equivalences.

\begin{lemma}
\label{lemma:compact}
Let $X$ be a pointed cubical set with finitely many non-degenerate cubes. Then the functor
\begin{equation}
[X,-]_*\colon \cSet_* \longrightarrow \Set_*
\end{equation}
commutes with filtered colimits.
\end{lemma}
\begin{proof}
A pointed cubical set with finitely many non-degenerate cubes is a compact object of $\cSet_*$. The claim then follows from the fact that $[X,Y]_*$ is the coequalizer
\begin{equation}
\label{eq:coeq_homotopy}
\cSet_*\bigl((X\otimes \Box^1)/ (\{x_0\}\otimes \Box^1) ,Y\bigr) \rightrightarrows \cSet_*(X,Y),
\end{equation}
together with the fact that $(X\otimes \Box^1)/ (\{x_0\}\otimes \Box^1)$ is also compact.
\end{proof}

\begin{lemma}
\label{lemma:homotopy_adjunction}
For a pointed cubical set $X$, a pointed digraph $G$ and an integer $m\geq 1$, there is a natural isomorphism
\begin{equation}
[X,N_m G]_* \cong [|X|_m, G]_*.
\end{equation}
\end{lemma}
\begin{proof}
We have $|X\otimes \Box^1|_m \cong |X|_m \otimes I_m$, and since $|-|_m$ is a left adjoint,
\begin{equation}
|(X\otimes \Box^1)/ (\{x_0\}\otimes \Box^1)|_m \cong (|X|_m \otimes I_m)/(\{x_0\}\otimes I_m).
\end{equation}
The claim now follows from the adjunction $|-|_m \dashv N_m$, together with the descriptions of $[X,Y]_*$ as the coequalizer~\eqref{eq:coeq_homotopy} and of $[H,G]_*$ as the coequalizer
\begin{equation}
\DGra_*\bigl((H \otimes I_m)/ (\{h_0\}\otimes I_m) ,G\bigr) \rightrightarrows \DGra_*(H,G)
\end{equation}
(Corollary~\ref{cor:homotopy-as-a-coequaliser}).
\end{proof}

\begin{proposition}\label{prop:A_n=pi_n}
For any pointed digraph $G$ and any $n\geq 0$, there is a natural isomorphism
\begin{equation}
A_n(G)\cong \pi_n(N_\infty G).
\end{equation}
\end{proposition}
\begin{proof}
Since $S^n= \Box^n/\partial \Box^n$ has finitely many non-degenerate cubes, Lemmas~\ref{lemma:compact} and~\ref{lemma:homotopy_adjunction} give
\begin{equation}
[S^n,N_\infty G]_*  \cong \underset{m}{\colim}\:  [S^n,N_{4m}G]_*  \cong \underset{m}\colim \: [S^n_{4m}, G]_*.
\end{equation}
Therefore $\pi_n(N_\infty G) \cong A_n(G)$.
\end{proof}

\begin{definition}[Cubical weak equivalence]
A digraph map $\varphi\colon G\to G'$ is a \emph{cubical weak equivalence} if, $A_0(G)\to A_0(G')$ is a bijection, and for every vertex $g$ of $G$ and every $n\geq 1$, the map $A_n(G,g)\to A_n(G',\varphi(g))$ is an isomorphism.
\end{definition}

\begin{proposition}
\label{prop:cubical_weak_equivalences}
Let $\varphi\colon G\to G'$ be a digraph map. Then the following are equivalent:
\begin{enumerate}
    \item $\varphi$ is a cubical weak equivalence;
    \item $N_\infty \varphi$ is a weak equivalence of cubical sets;
    \item $N_1\varphi$ is a weak equivalence of cubical sets.
\end{enumerate}
\end{proposition}
\begin{proof}
$(1)\Leftrightarrow (2)$. This follows from Proposition~\ref{prop:A_n=pi_n} and the fact that a map of Kan complexes is a weak equivalence if and only if it induces a
bijection on $\pi_0$ and isomorphisms on all higher homotopy groups at every
basepoint (Corollary~\ref{cor:weak_equivalences_of_kan}).

$(2)\Leftrightarrow (3)$. This follows from Theorem~\ref{th:N_1toN} and the two-out-of-three property for weak equivalences.
\end{proof}

\begin{corollary}
\label{cor:transf_comp}
Cubical weak equivalences satisfy the two-out-of-three property and are closed under (possibly infinite) coproducts.
\end{corollary}
\begin{proof}
This follows from the corresponding properties of weak equivalences of cubical sets (Corollary~\ref{cor:coproducts_of_weak_equivalences}) and the fact that $N_1$ preserves filtered colimits and coproducts.
\end{proof}

\section{\bf Finite homotopy limits}
\label{sec:Finite homotopy limits}

 In this section, we describe finite limits of the $\infty$-category $\DGra_\infty$ by introducing a structure of a category of fibrant objects on $\DGra$, similar to \cite[Th.~5.9]{CK24}, whose weak equivalences are cubical weak equivalences and whose fibrations are the so-called ``cubical fibrations.'' As a corollary, we obtain that $\DGra_\infty$ admits all finite limits and that the derived cubical nerve functor $N:\DGra_\infty \to \Spc$ preserves finite limits. To provide examples of cubical fibrations, we show that the 2-coverings introduced in \cite{DIMZ24} are cubical fibrations.

\subsection{Cubical fibrations} In this subsection we introduce cubical fibrations of digraphs and prove that the maps $G\to PG \to G\times G$ give a factorization of the diagonal map into a cubical weak equivalence followed by a cubical fibration.   

\begin{definition}[Cubical fibration of digraphs] 
A digraph map $p:G\to H$ is called cubical fibration, if $N_\infty p :N_\infty G\to N_\infty H$ is a fibration of cubical sets. 
\end{definition}

\begin{proposition}
A digraph map $p : G\to H$ is a cubical fibration if and only if, for any commutative square 
\begin{equation}
\begin{tikzcd}
{|\sqcap^n_{i,\varepsilon}|_{4m}}
\ar[r]
\ar[d,hookrightarrow]
& G 
\ar[d]
\\
I_{4m}^{\otimes n} 
\ar[r]
& H
\end{tikzcd}
\end{equation}
there exists $k\geq 0$ and a map $I^{\otimes n}_{4(m+k)}\to G$ such that the diagram 
\begin{equation}
\begin{tikzcd}
{|\sqcap^n_{i,\varepsilon}|_{4(m+k)}}
\ar[r,"c^{2k}_*"]
\ar[d,hookrightarrow]
&
{|\sqcap^n_{i,\varepsilon}|_{4m}}
\ar[r]
\ar[d,hookrightarrow]
& G 
\ar[d]
\\
I_{4(m+k)}^{\otimes n}
\ar[r,"c^{2k}_*"']
\ar[rru,dashed]
& I_{4m}^{\otimes n} 
\ar[r]
& H
\end{tikzcd}
\end{equation}
is commutative. 
\end{proposition}
\begin{proof}
Since $\sqcap^n_{i,\varepsilon}$ and $\Box^n$ are compact objects in $\cSet$ and $N_\infty G = \colim\:  N_{4m} G$, we obtain that $N_\infty p : N_\infty G\to N_\infty H$ is a fibration if and only if, for any commutative square 
\begin{equation}
\begin{tikzcd}
{\sqcap^n_{i,\varepsilon}}
\ar[r]
\ar[d,hookrightarrow]
& N_{4m}G 
\ar[d]
\\
\Box^n  
\ar[r]
& N_{4m} H,
\end{tikzcd}
\end{equation}
there is $k\geq 0$ and  a map $\Box^n \to N_{4(m+k)}G$ such that the diagram 
\begin{equation}
\begin{tikzcd}
{\sqcap^n_{i,\varepsilon}}
\ar[r]
\ar[d,hookrightarrow]
& N_{4m}G \ar[r,"{(c^{2k})^*}"] \ar[d] & N_{4(m+k)} G 
\ar[d]
\\
\Box^n  
\ar[r]
\ar[rru,dashed]
& N_{4m} H 
\ar[r,"{(c^{2k})^*}"']
& N_{4(m+k)} H
\end{tikzcd}
\end{equation}
is commutative. Then the result follows from the adjunction $|\cdot|_{m} \dashv N_m$. 
\end{proof}

\begin{notation}
For a digraph $G$ and its vertex $g$, we denote by 
\begin{equation}
P G = P_{\sf St} G, \hspace{1cm} \Omega (G,g) = \Omega_{\sf St} (G,g)
\end{equation}
the path and loop digraphs associated with the standard nice subcategory of shrinkings. 
\end{notation}

\begin{proposition}
\label{prop:G-PG-GxG}
For a digraph $G$, the maps   
\begin{equation}
G \longrightarrow P G \longrightarrow G\times G
\end{equation}
give a factorization of the diagonal map into a cubical weak equivalence followed by a cubical fibration. 
\end{proposition}
\begin{proof}
By Proposition \ref{prop:N_J_hom_L} we have an isomorphism 
\begin{equation}
N_{4m}(\Hom^\otimes(I_{4m},G)) \cong \Hom_R(\Box^1,N_{4m}G).
\end{equation}
Therefore, 
\begin{align}
N_\infty (P G) &=  \colim_m\:  N_{4m} ( \colim_k \Hom^\otimes(I_{4k},G) )\\
& \cong \colim_m \colim_k\:  N_{4m} ( \Hom^\otimes(I_{4k},G) ) \\
&\cong \colim_m\:  N_{4m}(\Hom^\otimes(I_{4m},G))\\
&\cong  \colim_m\: \Hom_R(\Box^1,N_{4m}G)\\
&\cong \Hom_R( \Box^1, N_\infty G ). 
\end{align}
It follows that after applying $N_\infty$ to the maps $G\to PG \to G\times G$ we obtain the factorization of the diagonal map 
\begin{equation}
X \longrightarrow \Hom_R(\Box^1, X) \longrightarrow X\times X, 
\end{equation}
into a weak equivalence followed by a fibration \cite[Prop.2.22]{CK23}, where $X=N_\infty G$.
\end{proof}

\subsection{Category of fibrant objects} In this subsection we show that the category $\DGra$ equipped with the classes of cubical weak equivalences and cubical fibrations is a category of fibrant objects. Using this we prove that $\DGra_\infty$ admits finite limits and that the functor $N:\DGra_\infty \to \Spc$ preserves finite limits.  

\begin{definition}[$\infty$-category with weak equivalences and fibrations]
\label{def:inf-cat_fib}
Following Cisinski  \cite[Def. 7.4.6; 7.4.12]{Cis19}, we define an $\infty$-category with weak equivalences and fibrations as an $\infty$-category $C$ equipped with two classes of morphisms $W$ and $F$, called weak equivalences and fibrations, such that the following holds. 
\begin{enumerate}
\item $W$ and $F$ contain all equivalences and closed under composition.
\item $C$ has a terminal object $1$\\ (an object $x$ is called fibrant if a map $x\to 1$ is a fibration).
\item $W$ satisfies $2$-out-of-$3$ property. 
\item For any fibrant objects $x,y,y'$ any fibration $f:x\to y$ and any morphism $v:y'\to y$ the pullback $x\times_y y'$ exists and the map $f ' :  x \times_y y' \to y'$ is a fibration. Moreover, if $f\in F \cap W$, then $f'\in F\cap W.$
\item Any map $u:x\to y$, where $y$ is fibrant, factors as a weak equivalence $w: x\to x'$ followed by a fibration $f:x'\to y.$ 
\end{enumerate}
Morphisms from $W\cap F$ are called acyclic fibrations.  

If $C$ is an ordinary category equipped with two classes of morphisms $W$ and $F$, we say that it is a category with weak equivalences and fibrations if its nerve $\Ner(C)$ is an $\infty$-category of weak equivalences and fibrations. 

An $\infty$-category  with weak equivalences and cofibrations is defined dually. 
\end{definition}

\begin{definition}[$\infty$-category of fibrant objects]
An $\infty$-category of fibrant objects is an $\infty$-category with weak equivalences and fibrations such that all its objects are fibrant \cite[Def. 7.5.7]{Cis19}. An (ordinary) category of fibrant objects is a category with weak equivalences and fibrations such that all its objects are fibrant. Note that any $\infty$-category with finite limits can be equipped with a structure of the category of fibrant objects, where weak equivalences are equivalences and all morphisms are fibrations. This structure will be called the trivial structure of a category of fibrant objects. An $\infty$-category of cofibrant objects is defined dually. 
\end{definition}

The following lemma is well-known but we could not find a reference in the $\infty$-categorical setting, so we give it with a proof. 

\begin{lemma}[Factorization lemma, cf. {\cite{Bro73}}]
\label{lemma:factorisation_lemma}
Let $C$ be an $\infty$-category  equipped with two classes of morphisms $W$ and $F$ satisfying axioms (1)-(4) from  Definition \ref{def:inf-cat_fib}. Assume that all objects are fibrant.  Then $C$ is an $\infty$-category of fibrant objects if and only if the following condition is  satisfied: 
\begin{itemize}
    \item[(5')] for any  object $y$ the diagonal map $y\to y\times y$ factors as a weak equivalence $w: y\to p$ followed by a fibration $f:p\to y\times y$.  
\end{itemize}
Moreover, for a map $u:x\to y$ the factorization of $u$ to a weak equivalence  followed by a fibration  is defined by morphisms 
\begin{equation}
x \xrightarrow{w_u} p_u \xrightarrow{{\sf pr}_2 f_u} y
\end{equation}
where $w_u:x\to p_u$ and $f_u :p_u\to x\times y$ are defined by the diagram 
\begin{equation}
\label{eq:p_u}
\begin{tikzcd}
x
\ar[r,"u"]  
\ar[rdd,"{(\id_x,u)}"']
\ar[rd,dashed,"w_u"]
& 
y 
\ar[rd,"w"]
\\
& p_u 
 \ar[r]
 \ar[d,"f_u"]
 \arrow[dr, phantom, "\lrcorner",  very near start]
 & 
 p 
 \ar[d,"f"]
 \\ 
 &
 x\times y 
 \ar[r,"u\times \id_y"]
 & 
 y\times y
\end{tikzcd}
\end{equation}
where $p_u$ is the pullback. 
\end{lemma}
\begin{proof} The implication $(5)\Rightarrow (5')$ is immediate. For the inverse direction, since all objects are fibrant, a product of any two objects $x\times y$ exists and projections ${\sf pr}_1 : x\times y\to x$ and ${\sf pr}_2: x\times y \to y$ are fibrations. In the diagram \eqref{eq:p_u} the map  $f_u:p_u\to x\times y$ is a fibration.  Therefore, the map ${\sf pr}_2 f_u : p_u \to y$ is a fibration. Note that the square 
\begin{equation}
\begin{tikzcd}
 p_u  
 \ar[r]
 \ar[d,"{\sf pr}_1 f_u"']
 \arrow[dr, phantom, "\lrcorner",  very near start]
 & 
 p 
 \ar[d,"{\sf pr}_1 f"]
 \\ 
 x 
 \ar[r,"u"]
 & 
 y
\end{tikzcd}
\end{equation}
is also a pullback. Using the 2-out-of-3 property and the fact that $w$ is a weak equivalence, and the homotopy ${\sf pr}_1 fw\simeq \id_y$, we obtain that ${\sf pr}_1 f$ is an acyclic fibration. 
Therefore, ${\sf pr}_1 f_u$ is also an acyclic fibration. Using the 2-out-of-3 property again, and the homotopy ${\sf pr}_1 f_u w_u\simeq \id_x$,  we obtain that $w_u$ is a weak equivalence. Therefore the map $u:x\to y$ factors as a weak equivalence $w_u:x\to p_u$ followed by a fibration ${\sf pr}_2 f_u:p_u\to y$ as required. 
\end{proof}

\begin{definition}[Left exact functor] Let $C$ and $D$ be $\infty$-categories of weak equivalences and fibrations. 
A functor $F:C\to D$ is called left exact \cite[Def. 7.5.2]{Cis19}, if
\begin{enumerate}
    \item it preserves the terminal object;
    \item it preserves fibrations and acyclic fibrations between fibrant objects; 
    \item it preserves pullbacks along fibrations between fibrant objects.
\end{enumerate} 
\end{definition}

\begin{proposition}
\label{prop:fibrant_obj}
The category $\DGra$ equipped with classes of cubical weak equivalences and cubical fibrations is a category of fibrant objects. Moreover, the functor $N_\infty : \DGra \to \cSet$ is left exact. 
\end{proposition}
\begin{proof} Let us prove that $\DGra$ is a category of fibrant objects. 
(1) Cubical weak equivalences and cubical fibrations contain all isomorphisms and closed under compositions. 
(2) $\DGra$ has a terminal object $*$. 
(3) Cubical weak equivalences satisfy 2-out-of-3 property. 
(4) The functor $N_\infty : \DGra \to \cSet$ preserves pullbacks (Proposition \ref{prop:N-inf-fin-limits}). Therefore the fact that cubical fibrations and acyclic cubical fibrations are closed under pullbacks follows from the fact that fibrations and acyclic fibrations of cubical sets are closed under pullbacks. (5')  Let us use Lemma \ref{lemma:factorisation_lemma}. For any digraph $G$, $N_\infty G$ is a Kan complex (Theorem \ref{thm:kancomplex}). It follows that any digraph is fibrant.  Then the assumption of Lemma \ref{lemma:factorisation_lemma} follows from Proposition  \ref{prop:G-PG-GxG}. 

The fact that $N_\infty$ is left exact follows from Propositions \ref{prop:N-inf-fin-limits} and \ref{prop:cubical_weak_equivalences} and the definition of cubical fibrations.   
\end{proof}

\begin{definition}[$\infty$-category of digraphs]
The $\infty$-category of digraphs is defined as the  localization (in the sense of $\infty$-categories) by the class of cubical weak equivalences $W_{\sf cub}$
\begin{equation}
\DGra_\infty = \DGra[W_{\sf cub}^{-1}].
\end{equation}
We will treat $\DGra_\infty$ as an $\infty$-category with the trivial structures of an $\infty$-category of fibrant and cofibrant objects. Further we identify the $\infty$-category of spaces with the localization (in the sense of $\infty$-categories) of the category of cubical sets 
\begin{equation}
\Spc = \cSet[W^{-1}].
\end{equation}
\end{definition}
\begin{definition}[Derived cubical nerve functor]
Since the functor $N_1:\DGra\to \cSet$ sends cubical weak equivalences to weak equivalences of cubical sets, it induces a functor of $\infty$-categories that we denote by 
\begin{equation}
\label{eq:nerve-inf}
N : \DGra_\infty \longrightarrow \Spc
\end{equation}
and call the derived cubical nerve functor. Note that, since $N_1 G \to N_\infty G$ is anodyne, $N$ can be also described as a functor induced by $N_\infty$.
\end{definition}

\begin{proposition}\label{prop:N_is_conservative}
    The derived cubical nerve $N:\DGra_\infty\to \Spc$ is conservative.
\end{proposition}
\begin{proof}
    Let $\gamma:\DGra\to\DGra_\infty$ denote the localization functor, and let
$
\alpha:\gamma G\to \gamma H$
be a morphism such that $N(\alpha)$ is an equivalence. By Proposition~\ref{prop:fibrant_obj}, we know that $\DGra$ is a category of fibrant objects, therefore the acyclic fibrations determine a right calculus of fractions for cubical weak equivalences \cite[Prop.~7.4.13 and Cor.~7.2.18]{Cis19}. By \cite[Th.~7.2.8 and Rem.~7.2.10]{Cis19}, there exists a span of digraph maps
\[
G\xleftarrow[\sim]{s}K\xrightarrow{f}H,
\]
with $s\in W_{\sf cub}$, such that 
$
\alpha\simeq\gamma(f)\circ\gamma(s)^{-1}$ 
in \(\DGra_\infty(\gamma G,\gamma H)\). Then
\[
N(\alpha)\simeq
N(\gamma(f))\circ N(\gamma(s))^{-1}.
\]
By the two-out-of-three property, we deduce that $N(\gamma(f))$ is an equivalence. Let $L:\cSet\to \Spc$ be the localization functor. Since $N\gamma\simeq LN_1$ and the weak equivalences of cubical sets are saturated, $N_1f$ is a weak equivalence. Proposition~\ref{prop:cubical_weak_equivalences} therefore implies that $f\in W_{\sf cub}$. It follows that $\gamma(f)$ is an equivalence, and therefore $\alpha$ is an equivalence. Thus $N$ is conservative.
\end{proof}

\begin{theorem}
The $\infty$-category $\DGra_\infty$ admits finite limits and the localization functor 
$\DGra \to \DGra_\infty$
is a left exact functor of $\infty$-categories of fibrant objects. Moreover, the functor $N:\DGra_\infty \to \Spc$ preserves finite limits. 
\end{theorem}
\begin{proof}
It follows from Proposition \ref{prop:fibrant_obj} and \cite[Prop.~7.5.6, Prop.~7.5.28]{Cis19}.
\end{proof}

\begin{example}
For any pointed digraph $G=(G,g)$, the square
\begin{equation}
\label{eq:Omega_homotopy_pullback}
\begin{tikzcd}
\Omega G
\ar[r]
\ar[d]
\arrow[dr, phantom, "\lrcorner",  very near start]
& 
* 
\ar[d]
\\
* 
\ar[r]
& G
\end{tikzcd}
\end{equation}
is a pullback square in $\DGra_\infty$. Indeed, consider the diagram, 
\begin{equation}
\label{eq:P_gG}
\begin{tikzcd}
\Omega G
\ar[r]
\ar[d]
\arrow[dr, phantom, "\lrcorner",  very near start]
& 
P_g G 
\ar[r]
\ar[d]
\arrow[dr, phantom, "\lrcorner",  very near start]
& 
PG 
\ar[d]
\\
* 
\ar[r,"g"]
& 
G 
\ar[r,"g \times \id"]
& 
G\times G 
\end{tikzcd}
\end{equation}
where all squares are pullbacks (Lemma  \ref{lemma:Omega_pullback}). By Lemma \ref{lemma:factorisation_lemma} the map $*\to P_g G$ is a cubical weak equivalence and the map $P_gG\to G$ is a cubical fibration. Then applying the left exact localization functor $\DGra \to \DGra_\infty$,  we obtain that the square \eqref{eq:Omega_homotopy_pullback} is a pullback square.  
\end{example} 

\subsection{Long exact sequence}  

In this subsection we show that the main result of \cite{LWYZ24} about the long exact sequence of cubical homotopy groups follows from our results. 

\begin{proposition}
\label{cor:long_exact}
Let  
\begin{equation}
\begin{tikzcd}
F 
\ar[r]
\ar[d]
\arrow[dr, phantom, "\lrcorner",  very near start]
& 
T 
\ar[d]
\\
G 
\ar[r]
& 
H
\end{tikzcd}
\end{equation}
be a pullback square in $\DGra$, where $T$ is a digraph such that $T\to *$ is a cubical weak equivalence. Assume that either $T\to H$ or $G\to H$ is a cubical fibration. Then, for any consistent choice of basepoints, there is a long exact sequence
\begin{equation}
\dots \to 
    A_1(F)\to A_1(G)\to A_1(H) \to A_0(F) \to A_0(G)\to A_0(H). 
\end{equation}
\end{proposition}
\begin{proof} 
Since either $G\to H$ is a cubical fibration, or $T\to H$ is a cubical fibration, we obtain that the image of this square in $\DGra_\infty$ is a pullback. Since $N: \DGra_\infty \to \Spc$ preserves pullbacks, applying $N$ to this square we obtain a pullback in the $\infty$-category $\Spc$, where $NT\simeq *$. Therefore $NF \to NG \to NH$ is a fiber sequence.  Then the result follows from the long sequence of homotopy groups for spaces and the isomorphism $A_*(-)\cong \pi_*(N(-))$.
\end{proof}

\begin{definition}
 For any pointed digraphs $G=(G,g)$ and $H=(H,h)$, and a pointed digraph map $\varphi:G\to H$, we define  $P_{\varphi}$ as a pullback 
 \begin{equation}
 \begin{tikzcd}
P_{\varphi}  
\ar[r]
\ar[d]
\arrow[dr, phantom, "\lrcorner",  very near start]
& 
P_h H 
\ar[d]
\\
G   
\ar[r,"\varphi"]
& 
H
\end{tikzcd}
\end{equation}  
where $P_h H$ is defined in \eqref{eq:P_gG}. Note that by Lemma \ref{lemma:factorisation_lemma}  the map $* \to P_h H$ is a cubical weak equivalence and the map $P_hH\to H$ is a cubical fibration. 
\end{definition}
\begin{corollary}[{\cite[Th.1.1]{LWYZ24}}]
\label{cor:long_exact_sequence}
For any pointed digraphs $G=(G,g)$ and $H=(H,h)$ and a pointed digraph map $\varphi: G\to H$ there is a long exact sequence
\begin{equation}
\dots \to 
    A_1(P_{\varphi})\to A_1(G)\to A_1(H) \to A_0(P_{\varphi}) \to A_0(G)\to A_0(H). 
\end{equation}
\end{corollary}

\subsection{2-covering digraphs} 

In this subsection we recall the definition of $\ell$-covering digraphs introduced in \cite{DIMZ24} and show that $2$-covering digraphs  are cubical fibrations. As a corollary we obtain that for a 2-covering $p:G\to H$ we have an isomorphism $A_n(G,g)\cong A_n(H,p(g))$ for $n\geq 2$. 

\begin{definition}[$1$-covering digraph]
A digraph map $p:G\to H$ is called $1$-covering, if for any commutative square 
\begin{equation}
\begin{tikzcd}
* 
\ar[r]
\ar[d,"i_k"]
& 
G 
\ar[d,"p"]
\\
I_1 
\ar[r]
& 
H
\end{tikzcd}
\end{equation}
where $k\in \{0,1\}$, there exists a unique map $I_1 \to G$ making the diagram commutative.
\end{definition}
\begin{definition}[Directed distance]
For two vertices $g,g'$ of a digraph $G$, we denote by ${\sf dist}(g,g')$ the infimum of lengths of paths from $g$ to $g'$.  Note that for any digraph map $\varphi:G\to H$ and any $g,g'\in G$ we have 
\begin{equation}
\dist( \varphi(g), \varphi( g')) \leq \dist(g,g').    
\end{equation} 
\end{definition}  

\begin{definition}[$\ell$-covering digraph] 
For a digraph $G$ and $k\geq 1$, we denote by $D_k(G)$ the digraph with the same set of vertices such that there is an arrow $g\to g'$ in $D_k(G)$ if there is a path of length at most $k$ from $g$ to $g'$ in $G$.
\begin{equation}
E(D_k(G)) = \{(g,g')\mid \dist(g,g')\leq k\}.    
\end{equation}
For $\ell\geq 1$, a digraph map $p:G\to H$ is called $\ell$-covering, if the induced map $D_k(p):D_k(G)\to D_k(H)$ is a $1$-covering for all $1\leq k\leq \ell$.
\end{definition}
\begin{proposition}[{\cite[Prop.3.3]{DIMZ24}}]
\label{prop:l-coverings}
For a digraph map $p:G\to H$, the following are equivalent. 
\begin{enumerate}
    \item $p$ is an $\ell$-covering.
    \item $D_\ell(p):D_\ell(G)\to D_\ell(H)$ is a $1$-covering and $G=D_\ell(p)^{-1}(H).$
    \item $p$ is $1$-covering and if $(g_0,\dots,g_n)$ and $(g'_0,\dots,g'_m)$ are paths in $G$ such that $0\leq n,m\leq \ell$ and either $g_0=g_0'$ and $p(g_n)=p(g'_m)$, or $p(g_0)=p(g'_0)$ and $g_n=g'_m$, then $g_0=g'_0$ and $g_n=g'_m$.
    \item For any $h,h'\in H$ such that $\dist_H(h,h')\leq \ell$, there is a unique bijection between fibers  $\alpha:p^{-1}(h)\to p^{-1}(h')$ such that for any $g\in p^{-1}(h)$ we have $\dist_G(g,\alpha(g))=\dist_H(h,h')$ and $\dist_G(g,g')>\ell$ for any $g'\in p^{-1}(h')\setminus \{\alpha(g)\}$. 
\end{enumerate}
\end{proposition}

\begin{lemma}
\label{lemma:2-covering-lifting}
Let $A\subseteq B$ be an induced subdigraph such that the following holds. 
\begin{enumerate}
    \item For any $b\in B$, there exists $a\in A$ with an arrow $a\to b$ in $B$. 
    \item For any $b\in B\setminus A$ and $a,a'\in A$ such that there are arrows $a\to b$ and $a'\to b$, there exists $a''\in A$  with arrows $a''\to a$ and $a''\to a'$ in $A$.
    \item For any arrow $b\to b'$ in $B$, there exists an arrow $a\to a'$ in $A$ together with arrows $a\to b$ and $a'\to b'$ in $B$.
\end{enumerate}
Then any $2$-covering $p:G\to H$ has a unique right lifting property with respect to the inclusion $A\hookrightarrow B$.
\end{lemma}
\begin{remark}
In the formulation of Lemma \ref{lemma:2-covering-lifting} some of the arrows can be degenerate. For example, in (1), if $b\in A$, then $a$ can be equal to $b$. In (2), $a''$ can be equal to $a$ or to $a'$. In (3), $b$ or $b'$ can be in $A$ and in this case we can have $a=b$ or $a'=b'$. 
\end{remark}

\begin{proof}[Proof of Lemma \ref{lemma:2-covering-lifting}]
Let $\alpha : A\to G$ and $\beta: B\to H$ be digraph maps such that $\beta|_A = p\alpha$. We need to show that there is a unique digraph map $\tilde \alpha:B\to G$ such that $\tilde \alpha|_A = \alpha$ and $p\tilde \alpha = \beta$. Uniqueness follows from (1): if $\tilde \alpha$ exists and $a\to b$ is an arrow from a vertex of $A$ to a vertex of $B$, then   $\alpha(a)\to \tilde \alpha(b)$ is the unique  lifting of the arrow $\beta(a)\to \beta(b)$.

Let us construct $\tilde \alpha.$ For any arrow $a\to b$ in $B$, where $a\in A$ and $b\in B\setminus A$, we define $\tilde \alpha_a(b)\in G$ so  that $\alpha(a)\to \tilde \alpha_a(b)$ is the  lifting of the arrow $\beta(a)\to \beta(b)$. Using (2) and Proposition \ref{prop:l-coverings}(3) we obtain that $\tilde \alpha_a(b)$ does not depend on $a.$ So we can denote it by $\tilde \alpha(b)$. Using (1) we obtain that $\tilde \alpha(b)$ is defined for all $b\in B\setminus A$. For $a\in A$ we of course define $\tilde\alpha(a)=\alpha(a)$. Let us show that $\tilde \alpha:B\to G$ is a digraph map. Take an arrow $b\to b'$ in $B$. Then by (3) there exist an arrow $a\to a'$ in $A$ together with arrows $a\to b$ and $a'\to b'$ in $B$. Consider a lifting  $\tilde \alpha(b) \to  g'$ of 
$\beta(b)\to \beta(b')$. Then $\alpha(a)\to \alpha(a')\to \tilde \alpha(b')$ and $\alpha(a)\to \tilde\alpha(b)\to g'$ are the  liftings of $\beta(a)\to \beta(a')\to \beta(b')$ and $\beta(a)\to \beta(b)\to \beta(b')$ respectively. By Proposition \ref{prop:l-coverings}(3), it follows that $g'=\tilde \alpha(b'),$ and hence, there is an arrow $\tilde \alpha(b)\to \tilde \alpha(b').$
\end{proof}

\begin{proposition}
\label{prop:2-cov}
Any 2-covering is a cubical fibration. 
\end{proposition}
\begin{proof} Take a $2$-covering $p:G\to H$. It is sufficient to show that $p$ has the unique right lifting property with respect to the map $| \sqcap^n_{i,\varepsilon} |_m \to I_m^{\otimes n}$ for any $m=4r$, $r\geq 1$, any $n\geq 0$ and any $\varepsilon\in \{0,1\}$. Assume that $\varepsilon=1$. If we set
\begin{equation}
A_k =(I_m^{\otimes (i-1)}\otimes  I_k \otimes I_m^{\otimes (n-i)}) \cup  | \sqcap^n_{i,1} |_m,
\end{equation}
 we obtain a filtration 
\begin{equation}
| \sqcap^n_{i,1} |_m = A_0 \subseteq A_1 \subseteq  \dots \subseteq A_m = I_m^{\otimes n} 
\end{equation}
such that, for any $1\leq k\leq m$, the inclusion  $A_{k-1}\subseteq A_{k}$ either satisfies the assumption of Lemma \ref{lemma:2-covering-lifting}, or its dual. Then $p$ has the unique right lifting property with respect to $|\sqcap^n_{i,1} |_m \to I_m^{\otimes n}.$  The case $\varepsilon=0$ is similar. 
\end{proof}

\begin{corollary}
For any $2$-covering  $p:G\to H$ and any $g\in G$, $p$ induces an injective map  $A_1(G,g)\to A_1(H,p(g))$ and isomorphisms 
\begin{equation}
A_n(G,g)\cong A_n(H,p(g)), \ \ n\geq 2.    
\end{equation}
\end{corollary}
\begin{proof}
It follows from Corollary  \ref{cor:long_exact}, Proposition \ref{prop:2-cov} and the fact that the fiber of a $2$-covering is discrete. 
\end{proof}

\begin{example}
For any $n\geq 3$, we denote by $C_n$ the directed $n$-cycle that can be defined as the Cayley digraph of $\mathbb Z/n$ with respect to the generator $1$. We also denote by $C_\infty$ the Cayley digraph of $\mathbb{Z}$ with respect to the generator $1$. Then the map 
\begin{equation}
C_\infty\longrightarrow C_n  
\end{equation}
is a 2-covering. It is easy to see that $C_\infty = \colim \: J_n$, where $J_n$ is the induced subdigraph on the set of vertices $\{-n,\dots,n\}$. Therefore $A_*(C_\infty,0)=\colim \:  A_*(J_n,0) \simeq *$. 
It follows that   
\begin{equation}
A_i(C_n,0) \cong 
\begin{cases}
\mathbb{Z}, & i=1,\\
0, & i\neq 1.
\end{cases}
\end{equation}
\end{example}

\section{\bf Nerve theorem for in-closed covers} 
\label{sec:nerve_theorem}

In this section, we introduce the notions of in-closed and out-closed subdigraphs, as well as the notions of in-closure and out-closure of a subdigraph. We prove a version of the nerve theorem for covers of a digraph by in-closed subdigraphs, which is a very efficient tool for computing the homotopy type of $N_1G$ for certain classes of digraphs $G$. We also show that if a digraph $G$ is the union of two of its in-closed subdigraphs, $G=H\cup H'$, then $N_1G = N_1H \sqcup_{N_1(H\cap H')} N_1H'$. This result is used in the following section to construct a structure of a category of cofibrant objects.

\begin{definition}[In-closed subdigraph]
An induced subdigraph $H\subseteq G$ is called in-closed  if every arrow of $G$ with target in $H$ also has its source in $H$. Dually $H\subseteq G$ is out-closed if every arrow of $G$ with source in $H$ also has its target in $H$. Note that $H\subseteq G$ is in-closed if and only if the  subdigraph $G\setminus H$ (induced subdigraph on  $V(G)\setminus V(H)$) is out-closed.  For a subdigraph $H\subseteq G$ we denote by 
\begin{equation}
\Out(H)\subseteq G, \hspace{1cm} \In(H) \subseteq G   
\end{equation}
the least out-closed subdigraph and the least in-closed subdigraph containing $H$ respectively. They are called out-closure and in-closure of $H$. If we want to specify $G$, we will use the notations $\Out_G(H)$ and $\In_G(H)$. Note that for any digraph map $\varphi:G\to G'$ we have 
\begin{equation}
\varphi(\Out_G(H)) \subseteq \Out_{G'}(\varphi(H)), \hspace{1cm} \varphi(\In_G(H)) \subseteq \In_{G'}(\varphi(H)). 
\end{equation}
Also note that the preimage of an in-closed (resp. out-closed) subdigraph is in-closed (resp. out-closed).  
\end{definition}

\begin{remark}
Hepworth and Roff in \cite{HR25} and \cite{HR24} denote the out-closure $\Out(H)$  by $rH$ and call it ``reach of $H$''.
\end{remark}

\begin{definition}[In-closed cover]
An in-closed (resp. out-closed) cover of a digraph $G$ is a family $(H_i)_{i\in I}$ of its in-closed (resp. out-closed) subdigraphs such that $G=\bigcup_{i\in I} H_i.$ 
\end{definition}

\begin{proposition}\label{prop:in-closed_union}
Let $G$ be a digraph and $(H_i)_{i\in I}$ be its in-closed (resp. out-closed) cover and let  $A$ be a digraph which is the in-closure (resp. out-closure) of one of its vertices. Then 
\begin{equation}
\Hom(A,G) = \bigcup_{i\in I} \Hom(A,H_i)  
\end{equation}
(here we identify $\Hom(A,H_i)$ with its image in $\Hom(A,G)$).
\end{proposition}
\begin{proof}
Let $a$ be the vertex of $A$ such that $A=\In_A(\{a\}).$ Then for any map $\varphi:A\to G$ we have $\varphi(A) = \varphi(\In_A(a)) \subseteq \In_G(\varphi(a)).$ Since $G=\bigcup_i H_i,$ there exists $i$ such that $\varphi(a)\in H_i.$ Therefore $\varphi(A)\subseteq \In_G(\varphi(a)) \subseteq H_i.$
\end{proof}

\begin{corollary}\label{cor:N_1:union}
Let $G$ be a digraph and $(H_i)_{i\in I}$ be its in-closed (resp. out-closed) cover. Then 
\begin{equation}
N_1G = \bigcup_{i\in I} N_1H_i.  
\end{equation}
\end{corollary}
\begin{proof}
It follows from the fact that $I_1^{\otimes k}$ is in-closure of its vertex $(1,\dots,1)$ and out-closure of its vertex $(0,\dots,0)$. 
\end{proof}

\begin{definition}[Nerve of a collection of subsets or subdigraphs]
\label{def:Ner}
Let $X$ be a set and $\UU=(U_i)_{i\in I}$ be a family of its subsets. For a finite subset $\sigma \subseteq I,$ we set 
$U_\sigma = \bigcap_{i\in \sigma} U_i.$
The nerve complex $\Ner(\UU)$ of this family is the (abstract) simplicial complex defined by  
\begin{equation}
\Ner(\UU) = \{\sigma \subseteq I\mid U_\sigma \neq \emptyset,\: \sigma \text{ is non-empty finite}\}.
\end{equation} 
For a digraph $G$ and a family of its induced subdigraphs $\HH=(H_i)_{i\in I}$ we denote by $\Ner(\HH)$ the nerve-complex of the family $(V(H_i))_{i\in I}$ in the set of vertices $V(G)$. 
\end{definition}

\begin{proposition}[Nerve theorem for in-closed covers]
Let $G$ be a digraph and $\HH=(H_i)_{i\in I}$ be its in-closed cover (resp. out-closed cover) such that  $N_1(\bigcap_{i\in \sigma } H_{i})$ is weakly contractible for any $\sigma\in \Ner(\HH).$ Then there is a homotopy equivalence
\begin{equation}
|N_1G| \simeq |\Ner(\HH)|. 
\end{equation}
\end{proposition}
\begin{proof}
By Proposition \ref{prop:in-closed_union} we have $N_1G= \bigcup_{i\in I} N_1H_i.$ Observe that $\bigcap_{i\in \sigma }  N_1H_{i} = N_1( \bigcap_{i\in \sigma } H_{i})$ for any finite $\sigma \subseteq I.$ Therefore, the result follows from the nerve lemma for cubical sets (see Proposition \ref{prop:simplicial_nerve_theorem} in the Appendix).
\end{proof}

\begin{proposition}\label{prop:nerve-equivalence}
Let  $\varphi:G\to G'$ be a morphism of digraphs and  $(H_i)_{i\in I}$ and $(H'_i)_{i\in I}$ be in-closed covers (resp. out-closed covers) of $G$ and $G'$ respectively such that $\varphi(H_i)\subseteq H'_i.$ Assume that for any $i_0,\dots,i_n\in I$ the map 
\begin{equation}
H_{i_0}\cap \dots \cap H_{i_n} \overset{\sim}\longrightarrow H'_{i_0} \cap \dots \cap H'_{i_n} 
\end{equation}
is a cubical weak equivalence. Then $\varphi$ is a cubical weak equivalence. 
\end{proposition}
\begin{proof}
By Proposition \ref{prop:in-closed_union} we have $N_1G= \bigcup_{i\in I} N_1H_i$ and it is easy to see that  $N_1(H_{i_0}\cap \dots \cap H_{i_n}) = N_1H_{i_0} \cap \dots \cap N_1H_{i_n}$. Then the result follows from Proposition \ref{prop:nerve_cubical_2} in the Appendix. 
\end{proof}

\begin{example}
Consider the digraph $O:=\Out_{I_4^{\otimes 2}}(\partial I_4^{\otimes 2}).$
\begin{equation*} O: \hspace{5mm}
\begin{tikzcd}[column sep = 4mm, row sep = 3mm]
(0,0)
\ar[r]
\ar[d]
& 
(0,1) 
\ar[d]
& 
(0,2) 
\ar[r]
\ar[d]
\ar[l]
& (0,3) 
\ar[d]
& 
(0,4) 
\ar[l]
\ar[d]
\\ 
(1,0) 
\ar[r]
& (1,1) & 
(1,2) 
\ar[r]
\ar[l]
& (1,3) & 
(1,4) 
\ar[l]
\\
(2,0) 
\ar[r]
\ar[d]
\ar[u]
& 
(2,1) 
\ar[d]
\ar[u]
&   & 
(2,3) 
\ar[d]
\ar[u]
& (2,4)
\ar[d]
\ar[l]
\ar[u]
\\
(3,0) 
\ar[r]
& (3,1) & 
(3,2)
\ar[r]
\ar[l]
& (3,3) & (3,4) 
\ar[l]
\\
(4,0) 
\ar[r]
\ar[u]
& 
(4,1) 
\ar[u]
& 
(4,2) 
\ar[r]
\ar[l]
\ar[u]
& 
(4,3) 
\ar[u]
& (4,4) 
\ar[l]
\ar[u]
\\
\end{tikzcd}
\end{equation*}
Take its out-closed covering defined by the 4 sets of vertices: 
\begin{equation}
\{0,1\} \times V,\ \  \{3,4\} \times V,\ \   V\times \{0,1\},\ \  V\times \{3,4\}, 
\end{equation}
 where $V=V(I_4)=\{0,\dots,4\}$. All these 4 out-closed subdigraphs and their non-empty intersections are contractible. Then the nerve of this cover is the cyclic graph with 4 vertices. Therefore 
\begin{equation}
|N_1O|\sim S^1.  
\end{equation}
Considering the intersection of this out-closed cover with $\partial I_4^{\otimes 2}$. We similarly obtain that $|N_1(\partial I_4^{\otimes 2})|\sim S^1.$ Moreover, using Proposition \ref{prop:nerve-equivalence}, we obtain that the inclusion
\begin{equation}
\partial I_4^{\otimes 2} \longrightarrow O
\end{equation}
is a cubical weak equivalence. 
\end{example}

\begin{proposition}\label{prop:union_of_two_inclosed}
Let $H,H'\subseteq G$ be two in-closed subdigraphs such that $G=H\cup H'$. Then the squares 
\begin{equation}
\begin{tikzcd}
H\cap H'
\ar[r]
\ar[d]
\arrow[dr, phantom, "\ulcorner" very near end]
& 
H
\ar[d]
\\
H' 
\ar[r]
&
G
\end{tikzcd}
\hspace{15mm}
\begin{tikzcd}
N_1(H\cap H') 
\ar[r]
\ar[d]
\arrow[dr, phantom, "\ulcorner" very near end]
& 
N_1H 
\ar[d]
\\
N_1H' 
\ar[r]
&
N_1G
\end{tikzcd}    
\end{equation}
are pushout squares in the category of digraphs and the category of cubical sets respectively. 
\end{proposition}
\begin{proof}
First we prove that the left-hand square is a pushout. It is sufficient to prove that any arrow of $G$ is either in $H$ or in $H'$. Suppose for contradiction that an arrow $(u,v)$ of $G$ belongs neither to $H$ nor to $H'$. Without loss of generality we can assume that $u\in G\setminus H$ and $v\in G\setminus H'.$ Then $v\in H$ and this contradicts the fact that $H$ is in-closed. Therefore $G=H \sqcup_{H\cap H'} H'.$ 

Corollary \ref{cor:N_1:union} implies that $N_1G=N_1H\cup N_1H'.$ One checks that $N_1H\cap N_1H'=N_1(H\cap H').$ Then $N_1G=N_1H\sqcup_{N_1(H\cap H')}N_1H'.$ 
\end{proof}

\section{\bf Homotopy colimits}  
\label{sec:homotopy_colimits}

In this section, we describe  colimits of the $\infty$-category $\DGra_\infty$ by introducing
a structure of a homotopy cocomplete  category of cofibrant objects on $\DGra$, whose weak equivalences are cubical weak equivalences and whose cofibrations are the
so-called “in-closed cofibrations.” As a corollary, we obtain that $\DGra_\infty$ is cocomplete and the derived cubical nerve functor $N : \DGra_\infty \to  \Spc$ is colimit-preserving. In the end of this section we use these results to prove the main result of the paper that $\DGra_\infty$ is equivalent to the $\infty$-category of spaces.

\subsection{In-closed inclusions} 

In this subsection we study inclusions of in-closed subdigraphs which are later used in the definition of in-closed cofibrations.  In particular, we show that in-closed inclusions form a saturated class of morphisms and that the functor $N_1 : \DGra \to \cSet$ preserves pushouts of two in-closed inclusions. 

\begin{definition}[Induced inclusion]
An injective digraph map is called induced inclusion if its image is an induced subdigraph. 
\end{definition}

\begin{definition}[In-closed inclusion]
An induced inclusion of digraphs is called in-closed (resp. out-closed) inclusion if its image is an in-closed (resp. out-closed) subdigraph. 
\end{definition}

\begin{lemma}
\label{lemma:in-closed-stable}\ 
\begin{enumerate}
\item Induced inclusions are stable under pushouts.
\item In-closed  inclusions are stable under pushouts.
\end{enumerate}
\end{lemma}
\begin{proof}
 (1) Let $H\subseteq  G$ be an induced subdigraph and $\varphi:H\to H'$ be a map. Set $G'=G \sqcup_H H'$ (see the description of the pushout in  Remark \ref{rem:pushout_along_inclusion}).   We need to prove that the subdigraph $H' \subseteq G'$ is induced. Take an arrow $(v,u)\in E(G')$, where $v,u\in V(H')$. We need to show that $(v,u)\in E(H')$. Assume that $(v,u)$ is from the image of   $E(G)\to V(G')\times V(G')$. Then there exist $(v',u')\in E(G)$ such that $\varphi'(v')=v$ and $\varphi'(u')=u.$ It follows that $v',u'\in V(H)$. Since $H$ is induced, $(v',u')\in E(H)$. Therefore, $(v,u)\in E(H')$.

(2) Now we assume that $H\subseteq G$ is in-closed and prove that $H'\subseteq G'$ is in-closed. Take an arrow $(v,u)\in E(G')$, where $u\in V(H').$ We need to prove that $v\in V(H')$. If $(v,u)\in E(H'),$ there is nothing to prove. Assume that $(v,u)$ is from the image of $E(G)\to E(G').$ Then there exists $(v',u')\in E(G)$ such that $\varphi'(v')=v$ and $\varphi'(u')=u.$ Since $u\in V(H'),$ we have $u'\in V(H)$. Using that $H$ is in-closed, we obtain $v'\in V(H).$ Therefore, $v\in V(H')$. 
\end{proof}

\begin{lemma}
\label{lemma:in-closed_transfinite}
\ 
\begin{enumerate}
    \item Induced inclusions are closed with respect to transfinite compositions. 
    \item In-closed inclusions are closed with respect to transfinite compositions. 
\end{enumerate}
\end{lemma}
\begin{proof}
(1). We can always replace in-closed inclusions with actual inclusions of subdigraphs. Therefore we can assume that we have an $\alpha$-sequence of subdigraphs i.e a collection of subdigraphs  $(G_\beta \subseteq G)_{\beta<\alpha},$  such that $G_\beta \subseteq G_{\gamma}$ for $\beta\leq \gamma<\alpha$; for any limit ordinal $\lambda$ we have $G_\lambda=\bigcup_{\beta<\lambda}G_\beta$  and $G=\bigcup_{\beta<\alpha}  G_\beta$. We need to show that if $G_\beta \subseteq G_{\beta+1}$ is induced, then $G_0\subseteq G$ is induced. Assume that $g'\to g$ is an arrow of $G$ such that $g',g\in G_0$. Consider the least $\beta$ such that $g'\to g$ is in $G_\beta$. We need to prove that $\beta=0$. If $\beta$ is the limit ordinal, then $G_\beta = \bigcup_{\gamma<\beta} G_\gamma$, then we can find $\gamma<\beta$ such that $g'\to g$ is in $G_\gamma$, which is a contradiction. If $\beta=\gamma+1$, then $g',g\in G_\gamma$ and, using that  $G_\gamma\subseteq G_{\gamma+1}$ is induced we obtain that $g'\to g$ is in $G_\gamma$, which is a contradiction. Therefore $\beta=0$.

(2). We need to show that if  $G_\beta \subseteq G_{\beta+1}$ is in-closed, for any $\beta$ such that $\beta+1<\alpha$, then  $G_0\subseteq G$ is in-closed. Assume that there is an arrow $g'\to g$ of $G$ with $g\in G_0$. Consider the least $\beta$ such that $g'\to g$ is in $G_\beta$. We need to show that $\beta=0$. The ordinal $\beta$ can not be a limit ordinal because in this case $G_\beta = \bigcup_{\gamma<\beta} G_{\gamma}$ and we would find $\gamma<\beta$ such that the arrow is in $G_\gamma$. If  $\beta=\gamma+1$ for some $\gamma,$ then $g\in G_\gamma$, and using that $G_\gamma \subseteq G_{\gamma+1}$ is in-closed we obtain that $g'\to g$ is in $G_\gamma$, which is a contradiction. Therefore $\beta=0$.
\end{proof}

\begin{lemma}\label{lemma:in-closed-retracts}
\ 
\begin{enumerate}
\item Induced inclusions are closed under retracts. 
\item In-closed inclusions are closed under retracts.
\end{enumerate}
\end{lemma}
\begin{proof}
(1) Let $H\subseteq G$ be an induced subdigraph. Assume that  $G'\subseteq G$ and $H'\subseteq G'\cap H$ are subdigraphs and there is a retraction $r:G\to G'$ such that $r(H)\subseteq H'$. We need to prove that $H'\subseteq G'$ is induced. 
If $h'_1,h'_2\in H'$ such that there is an arrow $h'_1\to h_2'$ in $G'$, then this arrow is in $H$, because $H$ is induced. Then the image of this arrow of $H$ under $r$ is the arrow $h'_1\to h'_2$ of $H'$. Hence $H'\subseteq G'$ is induced. 

(2) Assume that $h'\in H'$ and $g'\in G'$ and $g'\to h'$ is an arrow in $G'$. Since $H\subseteq G$ is in-closed, $g'\in H.$ Then $g'=r(g')\in H'.$ 
\end{proof}

\begin{definition}[Saturated class of morphisms] A class of morphisms of a category is called saturated if it is closed under pushouts, transfinite compositions and retracts. 
\end{definition}

\begin{proposition}\ 
\begin{enumerate}
    \item Induced inclusions form a saturated class of morphisms.
    \item In-closed (resp. out-closed) inclusions form a saturated class of morphisms.
\end{enumerate}
\end{proposition}
\begin{proof}
It follows from Lemmas \ref{lemma:in-closed-stable}, \ref{lemma:in-closed_transfinite},  \ref{lemma:in-closed-retracts}, and these dual versions for out-closed inclusions.
\end{proof}

\begin{proposition}
\label{cor:N_1:pushout}
Let $i:H\to G$ and $j: H \to G'$ be in-closed inclusions (resp. out-closed). Then 
\begin{equation}
\begin{tikzcd}
N_1H 
\ar[r] 
\ar[d]
\arrow[dr, phantom, "\ulcorner" very near end]
& 
N_1G 
\ar[d]
\\
N_1G' 
\ar[r]
& N_1(G \sqcup_H G')
\end{tikzcd}  
\end{equation}
is pushout square in the category of cubical sets.
\end{proposition}
\begin{proof} Set $G''=G\sqcup_H G'.$ Since $i$ and $j$ are in-closed inclusions, the maps $H\to G''$, $G\to G''$ and $G'\to G''$ are also in-closed inclusions.  Therefore we can identify them with their images $G,G'\subseteq G''$ which are in-closed subdigraphs such that $G''=G\cup G'$ and $G\cap G'=H.$ Then the result follows from Proposition \ref{prop:union_of_two_inclosed}.
\end{proof}

\begin{corollary}\label{cor:in-closed_decomp}
Let $H\subseteq G$ be an in-closed subdigraph and $O=\Out(H).$ Then the square 
\begin{equation}
\begin{tikzcd}
N_1(O\setminus H) 
\ar[r]
\ar[d]
\arrow[dr, phantom, "\ulcorner" very near end]
& 
N_1O 
\ar[d]
\\
N_1(G\setminus H) 
\ar[r]
&
N_1G
\end{tikzcd}    
\end{equation}
is a pushout square in the category of cubical sets. 
\end{corollary}

\begin{proposition}\label{prop:O-pushout}
Let $H\subseteq G$ be an in-closed subdigraph and $\varphi:H\to H'$ be a digraph map. Set  $O=\Out_G(H),$ $O'=O\sqcup_H H'$ and $G'=G\sqcup_H H'$. Then $O'=\Out_{G'}(H')$ and the square 
\begin{equation}
\begin{tikzcd}
N_1O 
\ar[r] \ar[d]
\arrow[dr, phantom, "\ulcorner" very near end]
& 
N_1O'\ar[d] \\
N_1G\ar[r] & N_1G'
\end{tikzcd}
\end{equation}
is a pushout square. 
\end{proposition}
\begin{proof}
We assume  $H\subseteq G$ is an in-closed subdigraph and $i:H\to G$ is the inclusion. Since the maps $H'\to G'$ and $O'\to G'$ are inclusions, we can identify $H'$ and $O'$ with their images in $G'.$ Since in-closed inclusions and out-closed inclusions are closed under pushouts, we obtain that $H'\subseteq G'$ is in-closed and $O'\subseteq G'$ is out-closed. Therefore $\Out_{G'}(H')\subseteq O'$. Moreover, let $q:G\to G'$ be the pushout map. Note that
\begin{equation}
    q(O)=q(\Out_G(H))\subseteq \Out_{G'}(q(H))\subseteq \Out_{G'}(H'),
\end{equation}
and $O'=q(O)\cup H'$, we deduce that $O'\subseteq \Out_{G'}(H')$. So we have $O'=\Out_{G'}(H').$ Combining  Corollary \ref{cor:in-closed_decomp} with isomorphisms (see \eqref{eq:G_H}) 
\begin{equation}
 O'\setminus H' \cong O\setminus H, \hspace{1cm} G'\setminus H' \cong G\setminus H   
\end{equation}
we obtain that the big composition square and the left-hand square of the diagram
\begin{equation}
\begin{tikzcd}
N_1(O\setminus H) 
\ar[r]
\ar[d]
& 
N_1O
\ar[r]
\ar[d]
& 
N_1O' 
\ar[d]
\\
N_1(G\setminus H) 
\ar[r]
&
N_1G
\ar[r]
&
N_1G'
\end{tikzcd}    
\end{equation}
are both pushout squares. Therefore, the right-hand square is also a pushout square.
\end{proof}

\subsection{Directed deformation retracts} 

In this subsection we introduce a notion of a directed deformation retract which is later used in the definition of in-closed cofibrations.

\begin{definition}[Directed deformation retract]
Let $H\subseteq G$ be an in-closed subdigraph such that $G=\Out_G(H)$. It is called directed deformation retract, if there is a morphism of digraphs 
\begin{equation}
 \eta:G\to G   
\end{equation}
such that the following holds: 
\begin{enumerate}
    \item[(DDR1)] there is an arrow $\eta(g)\to g$, for any $g\in G$;
    \item[(DDR2)] if $h\in H$, then  $\eta(h)=h$;
    \item[(DDR3)] if $g\in G \setminus H$, then for any  $h\in H$ we have   
    \[ {\sf dist}(h,g)={\sf dist}(h,\eta(g))+1.\]
    In particular, ${\sf dist}(h,g)=\infty$ if and only if ${\sf dist}(h,\eta(g))=\infty$. 
\end{enumerate}
An in-closed inclusion $i:H\to G$ is called directed deformation retract, if the image $i(H)\subseteq G$ is a directed deformation retract. 
\end{definition}

\begin{example}
Let $\tilde{\mathbb{N}}$ be the digraph, whose vertices are natural numbers and arrows have the form $n\to n+1$. Then the one-vertex subdigraph $H=\{0\} \subseteq  \tilde{\mathbb{N}}$ is a directed deformation retract. The map $\eta : \tilde{\mathbb{N}} \to \tilde{\mathbb{N}}$ is defined by $\eta(n)=\max\{0, n-1\}$. It is indicated by dashed arrows:
\begin{equation}
\begin{tikzcd}
{\bf 0} 
\ar[r]
& 
1 
\ar[r]
\ar[l,bend right=60,dashed]
& 
2 
\ar[r]
\ar[l,bend right=60,dashed]
& 
3
\ar[r]
\ar[l,bend right=60,dashed]
& 
\dots 
\end{tikzcd}
\end{equation}
Another example is given by $G=I_1 \otimes I_1$ and $H=\{(0,0)\}$.
\begin{equation}
\begin{tikzcd}
{\bf 00} 
\ar[r]
\ar[d]
& 
01
\ar[d]
\ar[l,bend right=60,dashed]
\\
10 
\ar[r]
\ar[u,bend left=60,dashed]
& 
11
\ar[u,bend right=60,dashed]
\end{tikzcd}
\end{equation}
For any digraph $G,$ $G\otimes \{0\}$ is a directed  deformation retract  of $G \otimes I_1.$
\end{example}

\begin{lemma}\label{lemma:def_retr_equiv}
Let $H\subseteq G$ be an in-closed subdigraph such that $G=\Out_G(H)$, and let $\eta:G\to G$ be a morphism of digraphs such that  (DDR1) and (DDR2) are satisfied. Then (DDR3) is equivalent to the fact that, for any $g\in G\setminus H$, there is a natural number $n=n(g)\geq 1$ such that 
\begin{itemize}
    \item[(a)] $\eta^{n}(g)\in H$;
    \item[(b)] for any $h\in H$ such that there is a path from $h$ to $g$, there is a shortest path from $h$ to $g$ that factors via
    $\eta^{n}(g) \to \eta^{n-1}(g) \to \dots \to g.$
\end{itemize}
\end{lemma}
\begin{proof}
Assume that $\eta$ satisfies (DDR3). Take some $g\in G\setminus H$. By induction, we obtain that, for any $i\geq 1$ such that $\eta^{i-1}(g)\in G\setminus H,$ and any $h\in H$ such that there is a path from $h$ to $g$, we have 
\begin{equation}
\dist(h,g) = \dist(h,\eta^i(g)) + i
\end{equation}
Therefore, there exists $n$ such that $\eta^n(g)\in H.$ We denote by $n=n(g)\geq 1$ the smallest $n$ such that $\eta^n(g)\in H$. Then we obtain $\dist(\eta^n(g),g)=n$. Hence $\eta^n(g)\to \dots \to \eta(g) \to g $ is a shortest path. We also have $\dist(h,g) = \dist(h,\eta^n(g))+n$ for any $h\in H$ such that $d(h,g)<\infty.$ Then there is a path from $h$ to $g$ that factors via $\eta^n(g)\to \dots \to \eta(g) \to g $. 

Now assume that (a),(b) are satisfied. Take some $g\in G\setminus H$ and consider the corresponding number $n=n(g)$. If $\dist(h,g)=\infty$, then $\dist(h,\eta(g))=\infty$, because there is an arrow $\eta(g)\to g$. If $\dist(h,g)<\infty$, there is a shortest path of the form $h=x_0\to \dots \to x_k \to \eta^n(g)\to \dots \to  \eta(g)\to g.$ Then $\dist(h,g)=  \dist(h,\eta(g))+1$.
\end{proof}

\begin{proposition} 
\label{prop:DDR-weak_equiv}
A directed deformation retract  is a cubical weak equivalence. 
\end{proposition}
\begin{proof}
Denote by $G_n\subseteq G$ the induced subdigraph consisting of vertices $g$ such that $\eta^n(g)\in H.$ Then $G_0=H$ and  $G=\bigcup G_n$. Since there exists arrows $\eta(g)\to g,$ the restriction  $\eta_n : G_n \to G_n$ of $\eta$ is homotopic to the identity map $\eta_n \simeq \id_{G_n}$. Therefore its $n$-th power is also homotopic to the identity $\eta_n^n \simeq \id_{G_n}$. If we denote by $i_n:H\to G_n$ the inclusion and by $r_n: G_n \to H$ the map defined by $r_n(g)=\eta_n^n(g),$ we obtain $r_n i_n = \id_H$ and $i_n r_n = \eta_n^n \sim \id_{G_n}.$ Note that for any vertex $h\in H$, the map $i_n:(H,h)\to (G_n,h)$ and $r_n:(G_n,h)\to (H,h)$ are pointed morphisms and the constructed homotopy equivalence $i_nr_n \sim \id_{G_n}$ is also pointed. Therefore, $i_n$ induces an isomorphism $A_*(H,h)\cong A_*(G_n,h)$, and hence, $i_n$ is a cubical weak equivalence.  

Since cubical weak equivalences satisfy 2-out-of-3 property, the inclusions $G_n\to G_{n+1}$ are also weak equivalences. Therefore the maps $N_1(G_n)\to N_1(G_{n+1})$ are anodyne. Since anodyne maps are closed under transfinite compositions, and $N_1$ preserves filtered colimits, the map $N_1(H)\to N_1(G)$ is also anodyne. Hence the map $H\to G$ is a cubical weak equivalence. 
\end{proof}

\begin{lemma}
\label{lemma:deformation-stable}
Directed deformation retracts  are closed under pushouts. 
\end{lemma}
\begin{proof} 
Let $H\subseteq G$ be a directed deformation retract. Consider a map $\varphi:H\to H'$ and set $G'=G \sqcup_{H} H'$ and $O'= O\sqcup_{H} H',$ where $O=\Out_G(H)$. We will identify $H'$ with its image in $G'$ (Remark \ref{rem:pushout_along_inclusion}). Therefore, we need to show that $H'\subseteq G'$ is a directed deformation retract.  We construct $\tilde \eta:G'\to G'$ by the universal property of the pushout
\begin{equation}
\begin{tikzcd}
H 
\ar[r,"\varphi"]
\ar[d,"i"]
\arrow[dr, phantom, "\ulcorner" very near end]
& 
H' 
\ar[d,"i'"]
\ar[rdd,"i'",bend left =1cm]
& 
\\
G 
\ar[r,"\tilde \varphi"]
\ar[rrd,"\tilde \varphi \eta"',bend right = 1cm]
& 
G' 
\ar[rd,"\tilde \eta",dashed]
& \\
& & 
G'
\end{tikzcd}
\end{equation}
Since in-closed inclusions are closed under pushouts, $H'\subseteq G'$ is in-closed.  So we need to check that $\tilde\eta$ satisfies  (DDR1)-(DDR3).  (DDR1) and (DDR2) are obvious. 

Let us prove (DDR3) using Lemma \ref{lemma:def_retr_equiv}. We take some $g'\in G'\setminus H'$, denote its preimage by $g\in G\setminus H$ and set $n=n(g'):=n(g)$. Then $\tilde \eta^n(g') = \varphi( \eta^n(g)) \in H'$.  
Take some $h'\in H'$ such that there exists a path from $h'$ to $g'$. It is sufficient to show that there is a shortest path from $h'$ to $g'$ that factors via the path $\tilde \eta^n(g') \to \dots \to \tilde \eta(g')\to g'$. Take some shortest path $h'=x_0 \to \dots\to x_k=g'$. Choose the last vertex of this path $x_l$, which is in $H'.$ Using the isomorphism $G\setminus H \cong G'\setminus H'$ we obtain that the path $x_{l+1}\to \dots \to x_k$ has a unique preimage $y_{l+1}\to \dots \to y_{k}=g$ in $G.$ The arrow $x_l\to x_{l+1}$ also has a preimage $y_l\to y_{l+1}$ in $G$, where $y_l\in H.$ Then the path $y_{l}\to \dots \to y_{k}=g$ is a preimage of a shortest path. Therefore, it is also shortest. In particular $\dist_G(y_l,g)=k-l.$ By Lemma \ref{lemma:def_retr_equiv} there exists a shortest path  $y_l=y'_l \to y'_{l+1} \to \dots \to y'_{k}=g$ in $G$ that factors via $\eta^n(g)\to \dots \to  \eta(g)\to g$.  Set $x'_{l+i}=\tilde\varphi(y'_{l+i}).$ Then $h'=x_0\to  \dots \to x_l \to x'_{l+1} \to \dots \to x'_{k}=g'$ is a shortest path that factors via  $\tilde \eta^n(g')\to \dots \to \tilde \eta(g') \to g'$.
\end{proof}

\begin{lemma}
\label{lemma:DDR-transfinite}
Directed deformation retracts are closed under transfinite compositions.    
\end{lemma}
\begin{proof}
We can always replace in-closed inclusions with actual inclusions of subdigraphs. Therefore we can assume that we have an $\alpha$-sequence of subdigraphs i.e a collection of subdigraphs  $(G_\beta \subseteq G)_{\beta<\alpha},$  such that $G_\beta \subseteq G_{\gamma}$ for $\beta\leq \gamma<\alpha$; for any limit ordinal $\lambda$ we have $G_\lambda=\bigcup_{\beta<\lambda} G_\beta$  and $G=\bigcup_{\beta<\alpha}  G_\beta$. We need to show that if $G_\beta \subseteq G_{\beta+1}$ is a directed deformation retract, for any $\beta$ such that $\beta+1<\alpha$, then $G_0\subseteq G$ is a directed deformation retract.

Note that the fact that for any limit ordinal $\lambda$ we have $G_\lambda = \bigcup_{\beta<\lambda} G_\beta$ implies that for any $g\in G\setminus G_0$ there is a unique $\beta=\beta(g)$ such that $g\in G_{\beta+1}\setminus G_\beta$. We denote by $\eta_\beta : G_{\beta+1}\to G_{\beta+1}$ the map from the definition of a directed deformation retract with respect to the subdigraph $G_\beta \subseteq G_{\beta+1}$. We need to show that $G_0\subseteq G$ is a directed deformation retract.

Since in-closed inclusions are closed under transfinite compositions,  $G_\beta\subseteq G$ is in-closed for any $\beta<\alpha$. We define $\eta:G\to G$ so that for any vertex $g\in G_{\beta+1}\setminus G_\beta$ we have $\eta(g)=\eta_\beta(g)$, and for $g\in G_0$, we have $\eta(g)=g$. Note that for any $g\in G$ we have an arrow $\eta(g)\to g$ (either degenerate or not).
Let us check that $\eta$ is a morphism of digraphs. Consider an arrow $g'\to g$ in $G$ and prove that there is an arrow $\eta(g')\to \eta(g)$. If $g\in G_0$, then $g'\in G_0$ and $\eta(g')=g'\to g=\eta(g)$ is the required arrow. Now assume that $g\in G\setminus G_0$. Take $\beta$ such that $g\in G_{\beta+1}\setminus G_\beta$. Since $G_{\beta+1}\subseteq G$ is in-closed, we obtain $g'\in G_{\beta+1}$. If $g'\in G_{\beta+1} \setminus G_{\beta}$, then $\eta_\beta(g')\to \eta_\beta(g)$ is the required arrow. If $g'\in G_{\beta}$, then $1=\dist_{G_{\beta+1}}(g',g)=\dist_{G_{\beta+1}}(g',\eta_{\beta+1}(g))+1.$ Therefore $g'=\eta_{\beta+1}(g)$. Hence $\eta(g')\to g'=\eta(g)$ is the required arrow. 

We already have the properties (DDR1) and (DDR2). 
Let's check the property (DDR3). Assume that $g\in G\setminus G_0$ and $h\in G_0$. Take $\beta$ such that $g\in G_{\beta+1}\setminus G_\beta$. Then 
\begin{align}
\dist_G(h,g) &=\dist_{G_{\beta+1}}(h,g) \\
& = \dist_{G_{\beta+1}}(h,\eta_{\beta}(g)) +1 \\
&= \dist_{G}(h,\eta(g)) +1.
\end{align}
\end{proof}

\begin{lemma}
\label{lemma:DDR_retracts}
Directed deformation retracts are closed under retracts. 
\end{lemma}
\begin{proof}
Let $H\subseteq G$ be a directed deformation retract and $H'\subseteq H$ and $G'\subseteq G$ be subdigraphs with a retraction $r:G\to G'$ such that $r(H)\subseteq H'$. We need to prove that $H'\subseteq G'$ is a directed deformation retract. Since in-closed inclusions are closed under retracts, $H'\subseteq G'$ is in-closed. Let $\eta:G\to G$ be a morphism of digraphs from the definition of directed deformation retract. Define $\eta' :G'\to G'$ by the formula $\eta' = r\circ \eta|_{G'}$. Let us check the properties (DDR1)-(DDR3) for $\eta'$ with respect to the subdigraph $H'\subseteq G'$. Take $g'\in G'$. Then we have an arrow $\eta(g')\to g'$. Since $r(g')=g',$ its image is an arrow $r(\eta(g'))\to g'.$ Then (DDR1) is satisfied. If $h'\in H'$, then $\eta(h')=h'$ and $r(h')=h'$. Therefore (DDR2) is also satisfied. Take $h'\in H'$ and $g'\in G'\setminus H'$. Note that, since $G'\subseteq G$ is a retract, we have $\dist_{G'}(h',g')=\dist_{G}(h',g')$. Therefore 
\begin{align}
\dist_{G'}(h',g') & = \dist_{G}(h',g') \\ 
& =\dist_{G}(h',\eta(g')) +1 \\
&\geq \dist_{G'}(r(h'),r(\eta(g'))) +1\\
&=\dist_{G'}(h',r(\eta(g'))) +1
\end{align}
Since there is an arrow $r(\eta(g'))\to g'$, we obtain $\dist_{G'}(h',g') \leq  \dist_{G'}(h',r(\eta(g'))) + 1$. Then $\dist_{G'}(h',g') =  \dist_{G'}(h',r(\eta(g'))) + 1$ and the property (DDR3) is satisfied. 
\end{proof}

\begin{proposition}
Directed deformation retracts form a saturated class of morphisms.
\end{proposition}
\begin{proof}
It follows from Lemmas \ref{lemma:DDR_retracts}, \ref{lemma:DDR-transfinite}, \ref{lemma:deformation-stable}. 
\end{proof}

\subsection{In-closed cofibrations} 

In this subsection we introduce the  notion of in-closed cofibrations which is used in the definition of a category of cofibrant objects on $\DGra$, show that the class of in-closed cofibrations is saturated and prove that $N_1$ sends pushouts along in-closed cofibrations to homotopy pushouts. 

\begin{definition}[ODDR-subdigraph, in-closed cofibration] An in-closed subdigraph $H\subseteq G$ is called ODDR-subdigraph (out-closure directed deformation retract) if $H$ is a directed deformation retract of $\Out(H)$. An in-closed inclusion $i:H\to G$ is called in-closed cofibration, if $i(H)\subseteq G$ is an ODDR-subdigraph. 
\end{definition}

\begin{lemma}\label{lemma:in-closed-cof_stable}
In-closed cofibrations are closed under pushouts. 
\end{lemma}
\begin{proof}
Let $H\subseteq G$ be an ODDR-subdigraph and $\varphi : H\to H'$ be a digraph map. Set $O=\Out_G(H)$, $G'=G\sqcup_H H'$ and $O'=O \sqcup_{H} H'$. By Proposition \ref{prop:O-pushout}, $O'=\Out_{G'}(H')$.
Since directed deformation retracts are closed under pushouts, $H'\subseteq O'$ is a directed deformation retract. The statement follows. 
\end{proof}

\begin{lemma}
\label{lemma:in-closed-cof-transfinite}
In-closed cofibrations are closed under transfinite compositions. 
\end{lemma}
\begin{proof}
As usual we assume that an $\alpha$-sequence of in-closed inclusions is represented by inclusions of digraphs  $(G_\beta \subseteq G)_{\beta<\alpha},$  such that $G_\beta \subseteq G_{\gamma}$ for $\beta\leq \gamma<\alpha$; for any limit ordinal $\lambda$ we have $G_\lambda=\bigcup_{\beta<\lambda} G_\beta$;  and $G=\bigcup_{\beta<\alpha}  G_\beta$. We assume that $G_\beta \subseteq G_{\beta+1}$ are ODDR-subdigraphs, for any $\beta$ such that $\beta+1<\alpha$, and prove that $G_0\subseteq G$ is an ODDR-subdigraph.

Set $O=\Out_{G}(G_0)$. We define $H_\beta$ for $\beta<\alpha$ so that $H_0=G_0$, $H_{\beta+1} =\Out_{G_{\beta+1}}(H_\beta) $ and $H_\lambda = \bigcup_{\beta<\lambda} H_\beta$, for a limit ordinal $\lambda$. We need to prove that $H_0\to O$ is a directed deformation retract. Since directed deformation retracts are closed under transfinite compositions, it is sufficient to prove that $H_\beta \subseteq H_{\beta+1}$ is a directed deformation retract; for any limit ordinal $\lambda$ we have $H_\lambda = \bigcup_{\beta<\lambda} H_\beta$; and  $O=\bigcup_{\beta<\alpha} H_\beta$.  

First we prove that  $H_\beta\subseteq G_\beta$ is out-closed and that 
\begin{equation}
H_{\beta} \cap  G_{\beta'} = H_{\beta'}, \hspace{1cm} \beta'\leq \beta <\alpha.
\end{equation}
The proof is by induction on $\beta$. Assume that $\beta=\gamma+1$. Then $H_\beta\subseteq G_\beta$ is out-closed by the definition. For $\beta'\leq \gamma$ we have $H_\gamma \cap G_{\beta'}=H_{\beta'}$. Since $H_{\beta'}$ is out-closed in $G_{\beta'}$, we obtain $H_\beta \cap G_{\beta'} =\Out_{G_\beta}(H_\gamma) \cap G_{\beta'} = H_{\beta '}$. For the case $\beta'=\beta$ the equation $ H_\beta \cap G_{\beta'}=H_{\beta'}$ is obvious. Now assume that  $\beta$ is a limit ordinal. Then we have $G_\beta = \bigcup_{\gamma<\beta} G_{\gamma}$ and $H_\beta = \bigcup_{\gamma<\beta} H_{\gamma}$.  Assume that there is an arrow $g\to g'$, where $g\in H_\beta$ and $g'\in G_\beta$. Then there exists $\gamma<\beta$ such that $g\to g'$ is in $G_\gamma$ and there exists $ \gamma \leq \gamma'<\beta$ such that $g\in H_{\gamma'}$. Since $H_{\gamma'} \cap G_\gamma = H_\gamma$ we obtain $g\in H_\gamma.$ Using that $H_\gamma\subseteq G_\gamma$ is out-closed, we obtain $g'\in H_\gamma.$ Therefore $H_\beta \subseteq G_\beta$ is out-closed. The fact that for $\beta'<\beta$ we have $H_\beta \cap G_{\beta'}=H_{\beta'}$ follows from the distributivity 
$\left(\bigcup_{\beta'\leq \gamma < \beta} H_\gamma \right) \cap G_{\beta'} = \bigcup_{\beta' \leq \gamma < \beta} (H_\gamma \cap G_{\beta'}) = H_{\beta'}. $

Let's prove that $H_\beta \subseteq H_{\beta+1}$ is a directed deformation retract. Denote by $\eta_\beta :\Out_{G_{\beta+1}}(G_\beta)\to \Out_{G_{\beta+1}}(G_\beta)$ the map from the definition of a directed deformation retract. We claim that $\eta_\beta$ can be restricted to a map $\tilde \eta_\beta:H_{\beta+1}\to H_{\beta+1}$. Indeed, for a vertex $h\in H_{\beta+1}$, we have a path from a vertex of $H_\beta$ to $h$. The image of this path under $\eta_\beta$ is a path from the same vertex of $H_\beta$ to $\eta_\beta(h)$. Therefore $\eta_\beta(h)\in H_{\beta+1}$. Using that  $H_{\beta+1} \cap G_\beta = H_\beta$, we see that the properties (DDR1)-(DDR3) for $\tilde \eta_\beta$ with respect to the subdigraph $H_\beta \subseteq H_{\beta+1}$ follow from the properties of $\eta_\beta$ with respect to the subdigraph $G_\beta \subseteq G_{\beta+1}$.   

Let us prove that, for a limit ordinal $\lambda<\alpha$, we have $H_\lambda = \bigcup_{\beta<\lambda} H_\beta$. Indeed, $H_\lambda = H_\lambda \cap (\bigcup_{\beta<\lambda} G_\beta )= \bigcup_{\beta<\lambda} (H_\lambda \cap G_\beta)= \bigcup_{\beta<\lambda} H_\beta$. 

Let us prove that $O=\bigcup H_\beta$. By induction one checks that $H_\beta \subseteq O.$ Let us prove $O\subseteq \bigcup H_\beta$. Take a vertex $o$ of $O.$ Then there is a path from a vertex of $H_0$ to $o$. 
Take  $\beta<\alpha$ such that $o\in G_\beta$.  Then there is a path from a vertex of $H_\beta$ to $o$. Using that $H_\beta \subseteq G_\beta$ is out-closed, we obtain $o\in H_\beta$. 
\end{proof}

\begin{lemma}
\label{lemma:in-closed-cof-retracts}
In-closed cofibrations are closed under retracts. 
\end{lemma}
\begin{proof}
Let $H\subseteq G$ be an ODDR-subdigraph. Assume that $G'\subseteq G$ and $H'\subseteq G'\cap H$ are subdigraphs such that there is a retraction $r:G\to G'$ such that $r(H)\subseteq H'$. We need to show that $H'\subseteq G'$ is an ODDR-subdigraph. Set $O=\Out_G(H)$ and $O'=\Out_{G'}(H')$. Then $H\subseteq O$ is a directed deformation retract. Applying $r$ to any path from a vertex of $H$ to a vertex of $G$ yields a path from a vertex of $H'$ to a vertex of $G'$. Therefore $r(O)\subseteq O'$. Then the statement follows from the fact that directed deformation retracts are closed under retracts. 
\end{proof}

\begin{proposition}
\label{prop:in-closed-cof-saturated}
In-closed cofibrations form a saturated class of maps. 
\end{proposition}
\begin{proof}
It follows from Lemmas \ref{lemma:in-closed-cof_stable}, \ref{lemma:in-closed-cof-transfinite}, \ref{lemma:in-closed-cof-retracts}. 
\end{proof}

\begin{theorem}
\label{th:Hur_cof_pushout} 
The functor $N_1:\DGra \to \cSet$ sends pushouts along in-closed cofibrations to homotopy pushouts. 
\end{theorem}
\begin{proof} 
We assume  $H\subseteq G$ is an ODDR-subdigraph, denote by  $i:H\to G$ the inclusion. We need to show that $N_1$ sends pushouts along $i$ to homotopy pushouts. Consider a map $\varphi:H\to H'$. We set $O=\Out_G(H)$,  $G' = G \sqcup_H H'$, $O'= O \sqcup_H H'$. Since the maps $H'\to G'$ and $O'\to G'$ are inclusions, we can identify $H'$ and $O'$ with their images in $G'.$ Since in-closed cofibrations are stable under pushouts,  $H'\subseteq G'$ is an ODDR-subdigraph. 
Proposition \ref{prop:O-pushout} implies that the square 
\begin{equation}
\begin{tikzcd}
{N_1O} \ar[r] \ar[d]
\arrow[dr, phantom, "\ulcorner" very near end]
& 
{N_1O'}\ar[d] \\
{N_1G}\ar[r] & {N_1G'}
\end{tikzcd}
\end{equation}
a pushout square. Moreover, since $N_1O\to N_1G$ is a cofibration, it is a homotopy pushout square.  

Using that $H\subseteq G$ and $H'\subseteq G'$ are ODDR-subdigraphs, by Proposition \ref{prop:DDR-weak_equiv} we obtain the following commutative diagram with vertical maps all cubical weak equivalences.
\begin{equation}
\begin{tikzcd}
    {N_1G} \ar[d,equal] & {N_1H} \ar[l]\ar[r]\ar[d,"\sim"] & {N_1H'} \ar[d,"\sim"] \\
    {N_1G} & {N_1O} \ar[l] \ar[r] & {N_1O'}
\end{tikzcd}    
\end{equation}
Therefore, the homotopy pushout of the upper row coincides with the homotopy pushout of the bottom row and is equal to $N_1G'$.  
\end{proof}

\subsection{Category of cofibrant objects}

In this subsection we show that $\DGra$ equipped with classes of cubical weak equivalences and in-closed cofibrations is a homotopy cocomplete category of cofibrant objects. Using this we prove that $\DGra_\infty$ admits small colimits and $N:\DGra_\infty \to \Spc$ preserves them. 

\begin{definition}[Homotopy cocomplete $\infty$-category of cofibrant objects] 
An $\infty$-category of cofibrant objects $C$ is homotopy cocomplete \cite[Def. 7.7.2]{Cis19} if it has small coproducts and the classes of  cofibrations and acyclic cofibrations are closed under coproducts. Let $C$ and $D$ be homotopy cocomplete $\infty$-categories of cofibrant objects.  A right exact functor $F:C\to D$ is called homotopy cocontinuous, if it preserves coproducts. 
\end{definition}

\begin{lemma}
\label{lemma:I_5}
For any digraph $G$, the maps 
\begin{equation}
G \sqcup  G \longrightarrow G\otimes I_4 \longrightarrow G
\end{equation}
form a factorization of the codiagonal $G\sqcup G\to G$ as an in-closed cofibration followed by a cubical weak equivalence.  
\end{lemma}
\begin{proof}
Since $G\otimes I_4 \to G$ is a homotopy equivalence, it is a cubical weak equivalence. The image of the map $G\sqcup G \to G\otimes I_4$ is $G\otimes \{0,4\}$. Its out-closure is $G\otimes (I_{0,1} \sqcup I_{3,4})$, where $I_{0,1}$ and $I_{3,4}$ are induced subdigraphs of $I_4$ spanned by vertices $0,1$ and $3,4$ respectively. So the result follows from the fact that  $G\otimes \{0,4\}\subseteq G\otimes (I_{0,1} \sqcup I_{3,4})$ is an ODDR-subdigraph.   
\end{proof}

\begin{proposition}
\label{prop:digraphs-cofibration_cat}
The category of digraphs $\DGra$ equipped with the classes of cubical weak equivalences and in-closed cofibrations is a homotopy cocomplete category of cofibrant objects.   
\end{proposition}
\begin{proof}
Let us prove the axioms of a category of cofibrant objects. $\DGra$ has an initial object $\emptyset$. Moreover, $\emptyset\to G$ is an in-closed cofibration, for any $G$. Therefore all digraphs are cofibrant.  Since in-closed cofibrations form a saturated class of maps (Proposition \ref{prop:in-closed-cof-saturated}) they contain all isomorphisms, and are closed under compositions and  pushouts.  Cubical weak equivalences satisfy 2-out-of-3 property by Corollary \ref{cor:transf_comp}.  The fact that acyclic cofibrations are closed under pushouts follows from Theorem \ref{th:Hur_cof_pushout}. For any digraph $G$,  the maps $G \sqcup G \to G \otimes I_4 \to G$ form a factorization of the codiagonal $G\sqcup G\to G$ as an in-closed cofibration followed by a cubical weak equivalence (Lemma \ref{lemma:I_5}). 
Then by Factorization Lemma (Lemma \ref{lemma:factorisation_lemma}) the category $\DGra$ is a category of cofibrant objects. 

Since in-closed cofibrations form a saturated class, they are closed under coproducts. Cubical weak equivalences are closed under coproducts by Corollary \ref{cor:transf_comp}. Therefore the category of cofibrant objects is homotopy cocomplete.
\end{proof}

\begin{theorem}
\label{th:colimits}
The $\infty$-category $\DGra_\infty$ is cocomplete, and the localization functor $\DGra  \to \DGra_\infty$
is homotopy cocontinuous. Moreover, the derived cubical nerve functor $N : \DGra_\infty \to \Spc$ is colimit-preserving.
\end{theorem}
\begin{proof}
Since $\DGra$ is a homotopy cocomplete category of cofibrant objects (Proposition \ref{prop:digraphs-cofibration_cat}), by \cite[Prop.7.7.4]{Cis19}, we obtain that $\DGra_\infty$ is cocomplete and $\DGra\to \DGra_\infty$ is homotopy cocontinuous. 

Consider the functor 
$\tilde N_1 : \DGra \longrightarrow \Spc$
defined as the composition of $N_1$ with the localization functor $\cSet\to \Spc$. We claim that $\tilde N_1$ is homotopy cocontinuous. First we need to check that it is right exact. To see this, note that $\tilde N_1$ preserves the initial object because it sends the empty digraph to the empty space;   $\tilde N_1$ preserves cofibrations and acyclic cofibrations because all morphisms of $\Spc$ are cofibrations and $\tilde N_1$ sends acyclic cofibrations to equivalences; $\tilde N_1$ sends pushouts along in-closed cofibrations to pushouts in $\Spc$ by Theorem \ref{th:Hur_cof_pushout}. Then $\tilde N_1$ is homotopy cocontinuous because it preserves coproducts. Therefore the fact that $N:\DGra_\infty \to \Spc$ is colimit-preserving  follows from \cite[Cor.7.7.5]{Cis19}.
\end{proof} 

\section{\bf Equivalence with the \texorpdfstring{$\infty$}{}-category of spaces} 
\label{section:equivalence}

In this section we show that the functor $N : \DGra_\infty \to \Spc$ induced by $N_1$ is equivalent to the mapping space functor from the point 
\begin{equation}
N\simeq \DGra_\infty(*,-).    
\end{equation}
Then we use it to prove our main result that $N$ is an equivalence of $\infty$-categories. 

\subsection{Calculus of fractions} In this subsection we recall the idea of calculus of fractions from \cite[\S 7.2]{Cis19} that allows to compute mapping spaces in localizations of $\infty$-categories.  

\begin{definition}[Right calculus of fractions] 
\label{def:RCF}
Let $C$ be an $\infty$-category, $W$ be a class of morphisms of $C$ and $x$ be an object of $C$. A right calculus of fractions at $x$ (see \cite[Def.~7.2.2,~7.2.6 and \S 7.2.7]{Cis19}) is a functor $\pi : I \to C$ from an $\infty$-category $I$ such that 
\begin{itemize}
    \item[(a)] $I$ has a terminal object $i_0$ and $\pi(i_0)=x$;
    \item[(b)] $\pi$ sends any map $i\to i_0$ to $W$;
    \item[(c)] the colimit  
    \[\underset{i\in I^\op }\colim\: C(\pi(i) , - ) = \colim ( I^\op \to C^\op \to  \Fun(C,\Spc) )\]
    treated as a functor $C\to \Spc$ sends $W$ to equivalences in $\Spc$. 
\end{itemize}
\end{definition}

\begin{remark}
In Definition \ref{def:RCF} we combine \cite[Def.7.2.6]{Cis19} with the remark in \cite[Def.7.2.7]{Cis19} saying that the fact that the map $\pi:I\to C$ is $W$-local is equivalent to the property (c) in our definition. Here we also use that, since $i_0$ is terminal in $I$, the functor $I_{/i_0}\to I$ is an equivalence and the colimit of the functor $(I_{/i_0})^\op\to \Fun(C,\Spc)$ is canonically equivalent to the colimit of the functor $I^\op \to \Fun(C,\Spc)$.  
\end{remark}

\begin{theorem}[{\cite[Th.7.2.8]{Cis19}}]
\label{th:calculus_of_fractions}
 Let $C$ be an $\infty$-category, $W$ be a class of morphisms of $C$, 
\begin{equation}
 \gamma:C\longrightarrow C[W^{-1}]   
\end{equation}
 be the localization functor, and $x$ be an object of  $C$. Assume that $\pi: I \to C$ is a right calculus of fractions at $x$. Then there is a canonical equivalence 
 \begin{equation}
   \underset{i\in I^\op }\colim\: C(\pi(i) , - ) \simeq C[W^{-1}](\gamma(x),\gamma(-))  
 \end{equation}
 in the $\infty$-category $\Fun(C,\Spc).$
\end{theorem}
\begin{remark}
Note that there is a more general version of Theorem \ref{th:calculus_of_fractions} in an unpublished preprint \cite{ACK25}.
\end{remark}

\subsection{Derived cubical nerve functor} In this subsection we show that the derived cubical nerve functor $N:\DGra_\infty\to \Spc$ induced by $N_1:\DGra \to \cSet$ is equivalent to the mapping space functor from the point  $\DGra_\infty(*,-)$.

\begin{theorem}
\label{th:derived_nerve}
There is a natural equivalence  
\begin{equation}
N \simeq \DGra_\infty(*,-)
\end{equation}
of functors $\DGra_\infty \to \Spc$, where $*$ is the point digraph. 
\end{theorem}
\begin{proof} Here we use Theorem \ref{th:calculus_of_fractions}. We claim that 
\begin{equation}
I_1^{\otimes \bullet} : \Box \longrightarrow \DGra
\end{equation}
is a right calculus of fractions at $*$. We need to check the properties (a),(b),(c) from Definition \ref{def:RCF}. (a): $\Box$ has the terminal object $[1]^0$ and $I_1^{\otimes 0}=*$. (b): The maps $I_1^{\otimes n}\to *$ are cubical weak equivalences because $I_1^{\otimes n}$ is contractible.

Let us prove (c). We denote by $i:\Set\to \Spc$ the inclusion of the category of sets to the $\infty$-category of spaces. Then for a cubical set $X:\Box^\op\to \Set$ we have 
\begin{equation}
\colim \: iX \simeq LX,
\end{equation}
where $L:\cSet\to \Spc$ is the localization functor \cite[Cor.3.14]{ACK25}. Moreover, this equivalence is natural in $X$. Therefore, we have
\begin{equation}
\label{eq:N-colim}
\underset{ \Box^\op}\colim \:  i\DGra(I_1^{\otimes \bullet}, - )\simeq LN_1\simeq N\gamma
\end{equation}
in $\Fun(\DGra,\Spc),$ where $\gamma:\DGra\to \DGra_\infty$ is the localization functor. It follows that $\colim_{\Box^\op} \:  \DGra(I_1^{\otimes \bullet}, - ) $ sends weak equivalences to equivalences in $\Spc$. 

Then by Theorem \ref{th:calculus_of_fractions}, we have
\begin{equation}
    N\gamma \simeq \DGra_\infty(\gamma(*),\gamma(-)).
\end{equation}
Finally, since 
\begin{equation}
    \gamma^*:\Fun(\DGra_\infty,\Spc)\to \Fun(\DGra,\Spc)
\end{equation}
is fully faithful by the universal property of localization, we conclude that $N\simeq \DGra_\infty(*,-).$

\end{proof}

\subsection{The main result}
In this subsection we prove our main result that $\DGra_\infty$ is equivalent to the $\infty$-category of spaces. 

\begin{remark}[Universal property $\Spc$]
\label{rem:universal_proprty_of_Spc}
For a cocomplete $\infty$-category $C$, the evaluation at the one-point space defines an equivalence of $\infty$-categories 
\begin{equation}
\Fun^{\colim}(\Spc,C) \simeq C, 
\end{equation} 
where $\Fun^{\colim}$ denotes the full subcategory in the $\infty$-category of functors spanned by colimit-preserving  functors \cite[Th.5.1.5.6]{Lur09}. Moreover, the category of colimit-preserving functors coincides with the full subcategory of left adjoint functors 
\begin{equation}
\Fun^L(\Spc,C) = \Fun^{\colim}(\Spc,C)
\end{equation}
(see \cite[Cor.~5.5.2.9 and Rem.~5.5.2.10]{Lur09}). In particular, for any object $X$ of $C$, there exists a unique (up to equivalence) adjunction 
\begin{equation}
F_X : \Spc \leftrightarrows C : G_X,
\end{equation}
where $F_X$ is the only colimit-preserving functor preserving colimits such that $F_X(*)\simeq X$. Note that the equivalence $ \Id_{\Spc}\simeq \Spc(*,-)$ implies that \begin{equation}
G_X \simeq C(X, -).
\end{equation}
\end{remark}
\begin{definition} 
Since the $\infty$-category $\DGra_\infty$ is cocomplete, by the universal property of $\Spc$, there exists a colimit-preserving functor 
\begin{equation}
R : \Spc \longrightarrow \DGra_\infty
\end{equation}
such that $R(*)\simeq *$.  
\end{definition}
\begin{theorem}
\label{th:equivalence_to_spaces}
The functor  $R:\Spc \to \DGra_\infty$ and the  derived cubical nerve functor $N: \DGra_\infty\to \Spc$ are mutually inverse equivalences of $\infty$-categories 
\begin{equation}
\Spc \simeq \DGra_\infty.
\end{equation}
\end{theorem}
\begin{proof} By Theorem \ref{th:derived_nerve} we have $N\simeq \DGra_\infty(*,-).$ Then  
Remark  \ref{rem:universal_proprty_of_Spc} implies  that the functors 
\begin{equation}
R : \Spc \leftrightarrows  \DGra_\infty: N
\end{equation}
are adjoint. Since $N$ is colimit-preserving (Theorem \ref{th:colimits}), 
$NR:\Spc\to \Spc$ is also colimit-preserving and  $NR(*)\simeq *.$ Using the equivalence  $\Fun^\colim(\Spc,\Spc)\simeq \Spc$, we obtain that the unit $\eta: \Id_{\Spc} \to NR$ is an equivalence. 

Let us check that the counit $ \varepsilon: RN \to \Id_{\DGra_\infty}$ is an equivalence. Since $N:\DGra_\infty\to \Spc$ is conservative by Proposition~\ref{prop:N_is_conservative}, it is sufficient to prove that $N\varepsilon : NRN \to N$ is an equivalence. This follows from the commutativity of the triangle
\begin{equation}
\begin{tikzcd}
N
\ar[rr,"\id_N"]
\ar[rd,"\eta N"']
& & 
N 
\\
& 
NRN
\ar[ru,"N\varepsilon"']
&
\end{tikzcd}
\end{equation}
and the 2-out-of-3 property of the class of equivalences. 
\end{proof}

\section{\bf Appendix. Nerve theorem for cubical sets}

In this section we prove the nerve theorem for cubical sets which is similar to the nerve theorem for simplicial sets that can be found in \cite[\S 12]{IX24}.  

\begin{definition}[Triangulation of a cubical set]
Note that there are well-defined maps of simplicial sets 
\begin{equation}
\max,\min : \Delta^1 \times \Delta^1 \to \Delta^1.    
\end{equation}
They allow us to define a cocubical simplicial set
\begin{equation}
\Box \longrightarrow \sSet, \hspace{1cm} [1]^n \mapsto (\Delta^1)^n 
\end{equation}
that induces an adjunction that we denote by 
\begin{equation}
T : \cSet \leftrightarrows \sSet : N_T.
\end{equation}
It is known that this adjunction is a Quillen equivalence between the Grothendieck model structure on $\cSet$ and the Kan-Quillen model structure on $\sSet$ \cite[Th.6.26]{DKLS24}, \cite[Prop.8.4.30]{Cis06}. The functor $T$ respects the geometric realizations \cite[Lemma 2.24]{CK23}
\begin{equation}
|TX| \cong |X|,
\end{equation}
and preserves monomorphisms \cite[Prop.1.31]{DKLS24}, \cite[Lem.8.4.29]{Cis06}.
\end{definition}
\begin{lemma}
\label{lemma:square_of_mono}
Assume that 
\begin{equation}
\begin{tikzcd}
A 
\ar[r]
\ar[d]
& 
B 
\ar[d]
\\
C 
\ar[r]
& D
\end{tikzcd}
\end{equation}
is a commutative square in the category of sets and that all the maps are monomorphisms.
If the square is a pushout  then it is also a pullback. 
\end{lemma}
\begin{proof}
We can assume that $A\subseteq B,C \subseteq D$ and the maps are inclusions. Then the fact that $D$ is a pushout means that the map $B \sqcup_A C \to D$ is a bijection. This implies that $A=B\cap C$. 
\end{proof}

\begin{lemma}
\label{lemma:intersections}
Assume that  $X$ is a cubical set and $U_0,\dots,U_n \subseteq X$ are its cubical subsets. Let us identify $TU_i$ with its image in $TX$. Then
\begin{equation}
T\left( \bigcap_{i=0}^n U_i \right) = \bigcap_{i=0}^n TU_i.
\end{equation}
\end{lemma}
\begin{proof} It is sufficient to prove the statement for $n=1$, as the general case follows by induction. 
Since $T$ preserves pushouts and monomorphisms, the square 
\begin{equation}
\begin{tikzcd}
T(U_0\cap U_1) 
\ar[r]
\ar[d]
& 
TU_0 
\ar[d]
\\
TU_1
\ar[r]
& 
T(U_0\cup U_1)
\end{tikzcd}
\end{equation}
is a pushout square all of whose maps are monomorphisms. Since limits and colimits in $\sSet$ are degreewise, using Lemma \ref{lemma:square_of_mono}, we obtain that this square is a pullback. Therefore $T(U_0\cap U_1) = TU_0 \cap TU_1.$
\end{proof}

\begin{lemma}
\label{lemma:weak_Tf}
A map of cubical sets $f$ is a weak equivalence if and only if $Tf$ is a weak equivalence of simplicial sets.
\end{lemma}
\begin{proof}
It follows from Corollary \ref{cor:weak_equiv} and the isomorphism $|Tf|\cong |f|$.
\end{proof}

\begin{proposition}\label{prop:simplicial_nerve_theorem}
Let $X$ be a cubical set and $\UU=(U_i)_{i\in I}$ be a family of its cubical subsets such that $X=\bigcup_{i\in I} U_i$ and  $U_\sigma$ is weakly contractible for any $\sigma\in \Ner(\UU).$ Then there is a homotopy equivalence
\begin{equation}
|X| \simeq |\Ner(\UU)|. 
\end{equation}
(see Definition \ref{def:Ner} for the definition of $\Ner(\UU)$).
\end{proposition}
\begin{proof}
This follows from the analogous statement for simplicial sets \cite[Th. 12.6]{IX24}, and the fact that $T:\cSet\to \sSet$ preserves monomorphisms and intersections (Lemma \ref{lemma:intersections}) and $|TX|\cong |X|$.  
\end{proof}

\begin{proposition}
\label{prop:nerve_cubical_2}
Let  $\varphi:X\to Y$ be a morphism of cubical sets and let  $\UU=(U_i)_{i\in I}$ and $\VV=(V_i)_{i\in I}$ be  covers of $X$ and $Y$ respectively such that $\varphi(U_i)\subseteq V_i.$ Assume that for any $i_0,\dots,i_n\in I$ the map 
\begin{equation}
U_{i_0}\cap \dots \cap U_{i_n} \overset{\sim}\longrightarrow V_{i_0} \cap \dots \cap V_{i_n} 
\end{equation}
is a weak equivalence. Then $\varphi$ is a weak equivalence. 
\end{proposition}
\begin{proof}
The statement follows from the analogous statement for simplicial sets \cite[Prop.12.7]{IX24}, Lemma \ref{lemma:weak_Tf} and the fact that $T:\cSet\to \sSet$ preserves  monomorphisms and intersections (Lemma \ref{lemma:intersections}) and $|TX|\cong |X|$.
\end{proof}

\end{document}